\documentclass[preprint,11pt]{article}
\usepackage{amssymb}
\usepackage{amsfonts}
\usepackage{mathrsfs}
\usepackage{amssymb,amsfonts,amsmath}
\usepackage{graphicx,graphics,epstopdf}
\usepackage{url}
\graphicspath{{./figures/}}
\usepackage{latexsym}
\usepackage[]{algorithm}
\usepackage{algpseudocode,algorithmicx}
\usepackage{esint}
\usepackage{multirow}
\usepackage{ctable} 
\usepackage{subfigure}
\usepackage{times}
\usepackage{multicol,enumerate}
\usepackage{epstopdf}
\usepackage{color,xcolor}
\usepackage{fullpage}
\usepackage{setspace}
\usepackage[english]{babel}
\usepackage[version=3]{mhchem}
\usepackage{amsmath, amssymb, amsthm}
\usepackage{verbatim, latexsym}
\usepackage{colortbl}
\usepackage{multirow}
\usepackage[OT2,OT1]{fontenc}
\usepackage{float}
\definecolor{black}{rgb}{0,0,0}

\definecolor{red}{rgb}{1,0,0}
\newcommand\red[1]{\textcolor{red}{#1}}
\definecolor{blue}{rgb}{0,0,1}

\newcommand\cyr
{
	\renewcommand\rmdefault{wncyr}
	\renewcommand\sfdefault{wncyss}
	\renewcommand\encodingdefault{OT2}
	\normalfont
	\selectfont
}
\DeclareTextFontCommand{\textcyr}{\cyr}

\def\XXint#1#2#3{{\setbox0=\hbox{$#1{#2#3}{\int}$ }
		\vcenter{\hbox{$#2#3$ }}\kern-.6\wd0}}

\newcommand{\etamaxmin}[1]{\eta_{\text{#1}}}

\newcommand{\dx}{\,\mathrm{d}x}
\newcommand{\ds}{\,\mathrm{d}s}
\newcommand{\prnt}[1]{\left( #1 \right)}

\newcommand{\norm}[1]{\left\|#1\right\|}

\newcommand{\normE}[2]{\norm{#1}_{H^{1}_{\kappa}\prnt{#2}}}
\newcommand{\seminormE}[2]{|#1|_{H^{1}_{\kappa}\prnt{#2}}}
\newcommand{\normL}[2]{\norm{#1}_{L^2\prnt{#2}}}

\newcommand{\normLii}[2]{\norm{#1}_{L^2_{\widetilde{\kappa}^{-1}}\prnt{#2}}}

\newcommand{\normI}[2]{\norm{#1}_{L^\infty\prnt{#2}}}

\newtheorem{theorem}{Theorem}[section]

\newtheorem{assumption}{Assumption}[section]
\newtheorem{remark}{Remark}[section]
\newtheorem{lemma}{Lemma}[section]

\numberwithin{equation}{section}
\usepackage{soul}
\setstcolor{red}

\title{Wavelet-based Edge Multiscale Parareal Algorithm for Parabolic Equations with Heterogeneous Coefficients}
\author{Guanglian Li\thanks{Corresponding author. Bernoulli Institute, University of Groningen, Nejinborgh 9, 9747 AG, Groningen, The Netherlands. (\texttt{lotusli0707@gmail.com}, \texttt{guanglian.li@rug.nl}). GL acknowledges the
		support from the Royal Society (London, United Kingdom) through a Newton international fellowship.}\, and Jiuhua Hu\thanks{Department of Mathematics, Texas A\&M University,
		College Station, TX 77843-3368, USA. (\texttt{E-mail: jhu@math.tamu.edu})}
}
\begin{document}
	\maketitle
	\begin{abstract}
		We propose in this paper the Wavelet-based Edge Multiscale Parareal (WEMP) Algorithm to efficiently solve parabolic equations with heterogeneous coefficients. This algorithm combines the advantages of multiscale methods that can deal with heterogeneity in the spacial domain effectively, and the strength of parareal algorithms for speeding up time evolution problems when sufficient processors are available. We derive the convergence rate of this algorithm in terms of the mesh size in the spatial domain, the level parameter used in the multiscale method, the coarse-scale time step and the fine-scale time step. Several numerical tests are presented to demonstrate the performance of our algorithm, which verify our theoretical results perfectly.
	\end{abstract}
	\section{Introduction}
	We consider in this paper a new efficient multiscale parareal algorithm for parabolic problems with heterogeneous coefficients.
	We first formulate the heterogeneous parabolic problems to present our new multiscale methods. Let $D\subset
	\mathbb{R}^d$ ($d=1,2,3$) be an open bounded Lipschitz domain. We seek a function $u(\cdot,t)\in V:=H^{1}_{0}(D)$ such that
	\begin{equation}\label{eqn:pde}
	\begin{aligned}
	\frac{\partial u}{\partial t}-\nabla\cdot(\kappa\nabla u)&=f &&\quad\text{ in }D\times (0,T],\\
	u(\cdot,0)&=u_0 &&\quad\text{ in }D,\\
	u&=0 &&\quad\text{ on } \partial D\times [0,T],
	\end{aligned}
	\end{equation}
	where the force term $f\in L^2([0,T]; \dot{H}^2(D))$ satisfying $\partial_{t}f\in L^2([0,T]; L^2(D))$, the initial data $u_0\in \dot{H}^3(D)\cap H^1_0(D)$ and the permeability coefficient $\kappa\in C^{\infty}(D)$ with $\alpha\leq\kappa(x)
	\leq\beta$ almost everywhere for some lower bound $\alpha>0$ and upper bound $\beta>\alpha$. Furthermore, the compatibility condition holds: $f(\cdot,0)+\nabla\cdot(\kappa\nabla u_0)\in H^1_0(D)$. Here, $\dot{H}^3(D)\subset L^2(D)$ is a Hilbert space to be defined in Section \ref{section:problem setting}. We denote by $\Lambda:=
	\frac{\beta}{\alpha}$ the ratio of these bounds, {which reflects the contrast of the coefficient $\kappa$}. To simplify the notation, let $I:=[0,T]$. Note that
	the existence of multiple scales in the coefficient $\kappa$ rends directly solving Problem \eqref{eqn:pde} challenging, since
	resolving the problem to the finest scale would incur huge computational cost.
	
	The accurate description of many important applications, e.g., composite materials, porous media and reservoir simulation, involves mathematical models with heterogeneous coefficients. In order to adequately describe the intrinsic complex properties in practical scenarios, the heterogeneous coefficients can have
	both multiple inseparable scales and high-contrast. Due to this disparity of scales, the classical numerical treatment becomes prohibitively expensive
	and even intractable for many multiscale applications. Nonetheless, motivated by the broad spectrum of practical applications, a large number of multiscale model reduction techniques, e.g., multiscale finite element methods (MsFEMs),
	heterogeneous multiscale methods (HMMs), variational multiscale methods, flux norm approach, generalized multiscale finite element methods (GMsFEMs) and localized orthogonal decomposition (LOD), have been proposed in the literature \cite{MR1455261,MR1979846,MR1660141,MR2721592, egh12, MR3246801, li2017error} over
	the last few decades. They have achieved great success in the efficient and accurate simulation of heterogeneous problems. Recently, a so-called Wavelet-based Edge Multiscale Finite Element Method (WEMsFEM), c.f. Algorithm \ref{algorithm:wavelet}, was proposed within the framework of GMsFEMs \cite{egh12} that facilitates deriving a rigorous convergence rate with merely mild assumptions \cite{li2019convergence, fu2018,fu2019wavelet}. The main idea of this method is to utilize wavelets as the basis functions over the coarse edges, and transform the approximate rate over the edges to the convergence rate in each local region. Then the Partition of Unity Method (PUM) \cite{Babuska2} is applied to derive the global convergence rate. The motivation for using wavelets as the ansatz space is that due to the existence of heterogeneity, the solution has a low regularity, and wavelets are known to be efficient in approximating functions with low regularity. We will apply this method in this paper to handle the heterogeneity in the spatial domain $D$.
	
Furthermore, motivated by the great demand for an efficient solver with high accuracy as well as a reasonable computing time in many practical applications, e.g., financial mathematics \cite{Bal_Mayday_AmericanPut}, fluid mechanics and fluid-structure interaction \cite{ time_decomposed_parallel2003, time_parallel2006, parareal_NS_mayday}, oceanography \cite{parareal_ocean_model2008}, chemistry \cite{parareal_ModelReduction_mayday2007, parareal_ChemicalKinetics2010} and quantum chemistry \cite{parareal_QuantumSystem_mayday2007}, and the increasing computational capacity of current computers, a variety of efficient numerical schemes exploiting parallel computing architectures emerge during the last few decades. Among them, the parareal algorithm is one of the most popular and success algorithms. The parareal algorithm facilitates speeding up the numerical simulation of the solutions to time dependent equations on the condition of sufficient processors \cite{bal2005convergence}, which is an iterative solver based on a cheap inaccurate sequential coarse-scale time solver and expensive accurate fine-scale time solvers that can be performed in parallel. It was introduced by Lions, Mayday and Turinici \cite{lions_mayday_2001_parareal}. The convergence analysis is studied for nonlinear system of ordinary differential equations and partial differential equations \cite{gander2008nonlinear, Bal_Mayday_AmericanPut}. Recently, new parareal algorithms are developed to solve problems involving discontinuous right-hand sides \cite{gander_2019_DiscontinuousSource, Gander_periodic_discontinousInput}. Coupling of parareal algorithm and some other techniques has been developed in many literatures, see \cite{parareal_multigrid_2014, parareal_integrator_2010, 2013Micro_Macro_parareal, 2016PararealMultiscale}. One of these is  the coupling of parareal algorithm and model reduction techniques. In \cite{2013Micro_Macro_parareal}, a micro-macro parareal algorithm for the time-parallel integration of multiscale-in-time systems is introduced to solve singularly perturbed ordinary differential equations. One contribution of this paper is that the fast variables are eliminated from the coarse propagator, therefore, the associated algorithm only evolves the slow variables. A new coupling strategy for the parareal algorithm with multiscale integrators is introduced in \cite{2016PararealMultiscale}.
	
In this paper, we incorporate the parareal algorithm into WEMsFEM to numerically calculate the time evolution problems efficiently. This new algorithm is called WEMP Algorithm, c.f. Algorithm \ref{algorithm:wavelet+parareal}. This algorithm is divided into two steps: a multiscale space $V_{\text{ms},\ell}^{\rm EW}$ based on WEMsFEM with $\ell$ as the wavelets level parameter is constructed in the first step, and then we apply the parareal algorithm by using $V_{\text{ms},\ell}^{\rm EW}$ as the ansatz space in the second step to obtain the solution more efficiently. The convergence rate of this algorithm is studied in Theorem \ref{prop:wavelet-based}. We proved
	\begin{equation*}
		\begin{aligned}
		\normL{u(\cdot,T^n)-U_{k}^{n}}{D}
		&\lesssim (H+2^{-\ell/2}\|\kappa\|_{L^{\infty}(\mathcal{F}_H)})H|u_0|_2+\delta t\Big(|u_0|_3+\|f\|_{L^2(I;\dot{H}^2(D))}+\|\partial_{t}f\|_{L^2(I;L^2(D))}\Big)\\
		&+
		\Big(\frac{1}{T^{n-1}}\Big)^{k+1}\frac{1}{k!}
		(\Delta T)^{k+1} \normL{u_0}{D},
		\end{aligned}
		\end{equation*}
where $u(\cdot,T^n)$ and $U_{k}^{n}$ are the numerical solution at $T^n=n\times\Delta T$ for $n=2,\cdots$ derived from the backward Euler conforming Galerkin method and WEMP algorithm. The notations $\Delta T$ and $\delta t$ represent for the coarse time step size and fine time step size, respectively. $H$, $\ell$ and $k$ are the space domain mesh size, the level parameter and iteration number.
This implies that taking $\ell=\mathcal{O}(|\log H|)$ and $k=\mathcal{O}(|\log\delta t/\log\Delta T|)$, we recover $\mathcal{O}(H^2+\delta t)$ error, which actually is the error for the backward Euler conforming Galerkin method. Furthermore, the singularity of the solution for $t\to 0$ is reflected in the coefficient of the last term, namely, $\Big(\frac{1}{T^{n-1}}\Big)^{k+1}$.

To demonstrate the performance of our proposed algorithm we present several numerical tests using backward Euler and Crank-Nicolson schemes for the fine time step solver, respectively. Our numerical tests indicate similar convergence as derived in the theoretical results. Furthermore, we take different coarse time steps and observe similar convergence behavior.
	
The paper is organized as follows. We present the basics on the discretization of Problem \eqref{eqn:pde} and the framework of Generalized Multiscale Finite Element Methods (GMsFEMs) in Section \ref{section:problem setting}. Then in Section \ref{sec:multiscale}, we present the construction of multiscale space $V_{\text{ms},\ell}^{\rm EW}$ by WEMsFEMs and its approximation properties. Our main proposed algorithm is presented in Section \ref{sec:WEMP}. Its convergence rate is derived in Section \ref{sec:convergence}. Extensive numerical tests are presented in Section \ref{sec:num}. Finally, we complete our paper with concluding remarks in Section \ref{sec:conclusion}.
	
\section{Problem setting} \label{section:problem setting}
We present in this section the discretization of problem \eqref{eqn:pde}. Firstly, we define the Hilbert space $\dot{H}^s(D)$, which is analogous to \cite[Chapter 3]{thomee1984galerkin}.

Let $\{(\lambda_m,\phi_m)\}_{m=1}^{\infty}$ be the eigenpairs of the following eigenvalue problems with the eigenvalues arranged in a nondecreasing order,
\begin{align*}
\mathcal{L}\phi_m&:=-\nabla\cdot(\kappa\nabla\phi_m)=\lambda_m \phi_m&& \text{ in } D\\
\phi_m&=0 &&\text{ on }\partial D.
\end{align*}
Note that the eigenfunctions $\{\phi_m\}_{m=1}^{\infty}$ form an orthonormal basis in $L^2(D)$, and consequently, each $v\in L^2(D)$ admits the representation
$v=\sum_{m=1}^{\infty}(v,\phi_m)_D\phi_m$ with $(\cdot,\cdot)_D$ being the inner product in $L^2(D)$. The Hilbert space $\dot{H}^s(D)\subset L^2(D)$ is defined by
\begin{align}\label{def:dotH}
\dot{H}^s(D)=\{v\in L^2(D): \sum_{m=1}^{\infty}\lambda_m^s|(v,\phi_m)_D|^2<\infty\}.
\end{align}
The associated norm in $\dot{H}^s(D)$ is $|v|_s=( \sum_{m=1}^{\infty}\lambda_m^s|(v,\phi_m)_D|^2)^{1/2}$.
\begin{remark}
Since the initial data $u_0\in \dot{H}^3(D)\cap H_1^0(D)$, we obtain
\begin{align}\label{eq:reg-u0}
\normL{\mathcal{L}u_0}{D}=|u_0|_2.
\end{align}
Indeed, $u_0$ allows the expression
\[
u_0=\sum_{m=1}^{\infty}(u_0,\phi_m)_D\phi_m.
\]
Taking $L^2(D)$-norm after Operating $\mathcal{L}$ on both sides and utilize the definition \eqref{def:dotH}, we obtain the desired assertion \eqref{eq:reg-u0}.
\end{remark}

Next we recap the regularity results to problem \eqref{eqn:pde}, which can be found, e.g., in \cite{citeulike:8310485}:
\begin{align}\label{eq:regularity-parabolic}
\sum\limits_{j=0}^{2}\int_{0}^{T}\Big|\frac{\partial^j u}{\partial t^j}\Big|_{2(1-j)+2}^2\mathrm{d}t
\lesssim |u_0|_3^2+\|f\|_{L^2(I;\dot{H}^2(D))}^2+\|\partial_{t}f\|_{L^2(I;L^2(D))}^2.
\end{align}
To discretize problem \eqref{eqn:pde}, we first introduce fine and coarse grids. Let $\mathcal{T}_H$ be a regular partition of the domain $D$ into finite elements (triangles, quadrilaterals, tetrahedral, etc.) with a mesh size $H$. We refer to this partition as coarse grids, and its elements as the coarse elements. Then each coarse element is further partitioned into a union of connected fine grid blocks. The fine-grid partition is denoted by $\mathcal{T}_h$ with $h$ being its mesh size. Let $\mathcal{F}_h$ (or $\mathcal{F}_H$) be the collection of all edges in $\mathcal{T}_h$ (or $\mathcal{T}_H$). Over the fine mesh $\mathcal{T}_h$, let $V_h$ be the conforming piecewise linear finite element space:
\[
V_h:=\{v\in V: V|_{E}\in \mathcal{P}_{1}(E) \text{ for all } E\in \mathcal{T}_h\},
\]
where $\mathcal{P}_1(E)$ denotes the space of linear polynomials on the fine element $E\in\mathcal{T}_h$.

The time interval $I:=[0,T]$ is decomposed into a sequence of coarse subintervals $[T^n,T^{n+1}]$ for $n=0,1,\cdots, M_{\Delta}$ of size $\Delta T$ with $\Delta T:=T/M_{\Delta}$ for some $M_{\Delta}\in \mathbb{N}_{+}$ and $T^0:=0$. Each coarse time interval $[T^{n}, T^{n+1}]$ is further discretized with a fine time step $\delta t$.
Let $t_n=n\times\delta t$ for $n=0,1,\cdots,M_{\delta}$ with $M_{\delta}:=T\times\delta t^{-1}$. Note that $\Delta T\gg \delta t$. To simplify the notations, backward Euler method is utilized to discretize the time variable, and we use conforming Galerkin method for the discretization in the spatial variable throughout this paper. Then the fine-scale solution $u_h^{n}\in V_h$ for $n=1,2,\cdots,M_{\delta}$ satisfies
	\begin{equation}\label{eqn:weakform_h}
	\left\{
	\begin{aligned}
	(\frac{u_h^{n}-u_h^{n-1}}{\delta t},v_h)_D+a(u_h^{n},v_h)&=(f(\cdot,t_n),v_h)_{D} \quad \text{ for all } v_h\in V_h,\\
	u_h^0&=I_h u_0.
	\end{aligned}
	\right.
	\end{equation}
	Here, the bilinear form $a(\cdot,\cdot)$ on $V\times V$ is defined by
	\[
	a(v_1,v_2):=\int_{D}\kappa \nabla v_1\cdot \nabla v_2\dx \text{ for all } v_1,v_2\in V.
	\]
	$I_h$ is a proper projection from $V$ to $V_h$. Furthermore, we define the energy norm $\|v\|_{H^1_{\kappa}(D)}:=\sqrt{a(v,v)}$ for all $v\in V$.
	
	The fine-scale solution $u_h^{n}$ will serve as a reference solution in Section \ref{sec:num}. Note that due to the presence of multiple scales in the coefficient $\kappa$, the fine-scale mesh size $h$ should be commensurate with the smallest scale and thus it can be very small in order to obtain an accurate solution. This necessarily involves huge computational complexity, and more efficient methods are in great demand.
	
	
	In this work, we are concerned with flow problems with high-contrast heterogeneous coefficients,
	which involve multiscale permeability fields, e.g., permeability fields with vugs and faults, and
	furthermore, can be parameter-dependent, e.g., viscosity. Under such scenario, the computation of the
	fine-scale solution $u_h$ is vulnerable to high computational
	complexity, and one has to resort to multiscale methods. The GMsFEM has been extremely successful
	for solving multiscale flow problems, which we briefly recap below. 
	
	The GMsFEM aims at solving Problem \eqref{eqn:pde} on the coarse mesh $\mathcal{T}_{H}$
	cheaply, which, meanwhile, maintains a certain accuracy compared to the fine-scale solution $u_h$. To describe the
	GMsFEM, we need a few notations. The vertices of $\mathcal{T}_H$
	are denoted by $\{O_i\}_{i=1}^{N}$, with $N$ being the total number of coarse nodes.
	The coarse neighborhood associated with the node $O_i$ is denoted by
	\begin{equation} \label{neighborhood}
	\omega_i:=\bigcup\{ K_j\in\mathcal{T}_H: ~~~ O_i\in \overline{K}_j\}.
	\end{equation}
	We refer to Figure~\ref{schematic} for an illustration of neighborhoods and elements subordinated to the coarse
	discretization $\mathcal{T}_H$. Throughout, we use $\omega_i$ to denote a coarse neighborhood. Furthermore, let $\mathcal{F}_h(\partial \omega_i)$ (or $\mathcal{F}_H(\partial \omega_i)$) be the restriction of $\mathcal{F}_h$ on $\partial\omega_i$ (or $\mathcal{F}_H$ on $\partial\omega_i$).
	
	\begin{figure}[htb]
		\centering
		\includegraphics[width=0.65\textwidth]{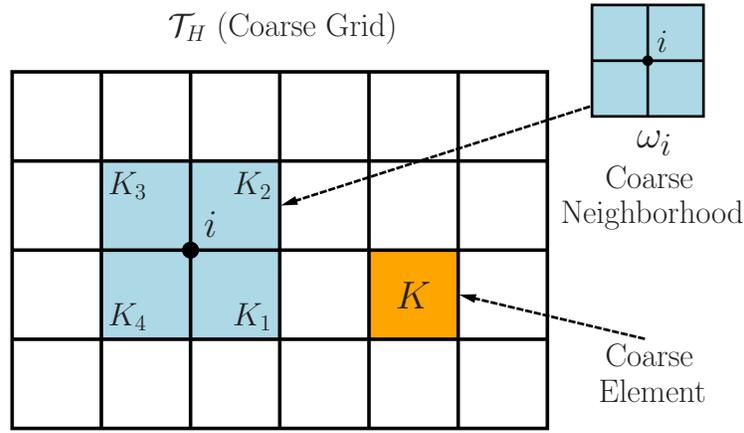}
		\caption{Illustration of a coarse neighborhood and coarse element.}
		\label{schematic}
	\end{figure}
	
	Next, we outline the GMsFEM with a conforming Galerkin (CG) formulation. Let $1\leq i\leq N$ be a certain coarse node. Note that $\omega_i$ is
	the support of the multiscale basis functions to be identified, and $\ell_i\in \mathbb{N}_{+}$ is the number of those multiscale basis functions associated with $\omega_i$. They are denoted as $\psi_k^{\omega_i}$ for $k=1,\cdots,\ell_i$.  Throughout, the subscript $i$ denotes the $i$-th coarse node or coarse neighborhood.
	Generally, the GMsFEM utilizes multiple basis functions per coarse neighborhood $\omega_i$,
	and the index $k$ represents the numbering of these basis functions.
	In turn, the CG multiscale solution $u_{\text{ms}}$ is sought as $u_{\text{ms}}=\sum\limits_{i=1}^{N}\sum\limits_{k=1}^{\ell_i} c_{k}^i \psi_{k}^{\omega_i}$.
	Once the basis functions $\psi_k^{\omega_i}$ are identified, the multiscale solution  $u_{\text{ms}}^{n}\in V_{\text{ms}}$ for $n=1,\cdots,M_{\delta}$ satisfies
	\begin{equation}\label{eq:multiscale}
	\left\{
	\begin{aligned}
	(\frac{u_{\text{ms}}^{n}-u_{\text{ms}}^{n-1}}{\delta t},v_{\text{ms}})_D+a(u_{\text{ms}}^{n},v_{\text{ms}})&=(f(\cdot,t_n),v_{\text{ms}})_{D}
	\quad \text{for all} \, \, v_{\text{ms}}\in
	V_{\text{ms}},\\
	u_{\text{ms}}^0&=I_{\text{ms}} u_0,
	\end{aligned}
	\right.
	\end{equation}
	where $V_{\text{ms}}$ denotes the multiscale space spanned by these multiscale basis functions and $I_{\text{ms}}$ is a projection operator from $V$ to $V_{\text{ms}}$.

Note that we need to build the multiscale space $V_{\text{ms}}$ to solve for $u_{\text{ms}}^{n}$ from \eqref{eq:multiscale} for $n=1,\cdots,M_{\delta}$. The construction of $V_{\text{ms}}$ will be presented in Section \ref{sec:multiscale}. Note also that we need a very tiny fine-scale time step $\delta t$ to guarantee a reasonable approximation property of  $u_{\text{ms}}^{n}$ to $u(\cdot,t_n)$ for $n=1,\cdots,M_{\delta}$ due to, e.g., the singularity of the solution $u(\cdot,t)$ at $t=0$ when the source term $f$ fails to belong to certain smooth functional space. Consequently, the computational complexity of the multiscale method \eqref{eq:multiscale} can be extremely expensive. For this reason, we present in Section \ref{sec:WEMP} a multiscale algorithm incorporated with the parareal algorithm to reduce further this part of computational cost.
	
We end this section with assumptions on the permeability field $\kappa$:
	\begin{assumption}[Structure of $D$ and $\kappa$]\label{ass:coeff}
		Let $D$ be a domain with a $C^{1,\alpha}$ $(0<\alpha<1)$ boundary $\partial D$,
		and $\{D_i\}_{i=1}^m\subset D$ be $m$ pairwise disjoint strictly convex open subsets, {each with a $C^{1,\alpha}$ boundary
			$\Gamma_i:=\partial D_i$,} and denote $D_0=D\backslash \overline{\cup_{i=1}^{m} D_i}$. 
		Let the permeability coefficient $\kappa$ be piecewise regular function defined by
		\begin{equation}
		\kappa=\left\{
		\begin{aligned}
		&\eta_{i}(x) &\text{ in } D_{i},\\
		&1 &\text{ in }D_0.
		\end{aligned}
		\right.
		\end{equation}
		Here $\eta_i\in C^{\mu}(\bar{D_i})$ with $\mu\in (0,1)$ for $i=1,\cdots,m$. Denote $\etamaxmin{min}:=\min\limits_{i}\{\min\limits_{x\in D_i}\{\eta_i(x)\}\}\geq 1$ and $\etamaxmin{max}:=\max\limits_{i}\{\|\eta_i\|_{C_0(D_i)}\}$.
	\end{assumption}
\section{Multiscale space construction}\label{sec:multiscale}
This section is concerned with the construction of the multiscale space by means of the Wavelet-based Edge Multiscale Finite Element Methods (WEMsFEM) \cite{li2019convergence,fu2019wavelet}.
	
The algorithm is presented in Algorithm \ref{algorithm:wavelet}. Given the level parameter $\ell\in \mathbb{N}$, and the type of wavelets on each edge of the coarse neighborhood $\omega_i$, one can obtain the local multiscale space $V_{i,\ell}$ on $\omega_i$ by solving $2^{\ell+2}$  local problems as in Step 2. In Step 3, we can use these local multiscale space to build the global multiscale space $V_{\text{ms},\ell}^{\rm EW}$ by multiplying the partition of unity functions $\chi_i$. Finally, we can solve the multiscale problem \eqref{eq:multiscale} using this global multiscale space.

To this end, we begin with an initial coarse space $V^{\text{init}}_0 = \text{span}\{ \chi_i \}_{i=1}^{N}$.
The functions $\chi_i$ are the standard multiscale basis functions on each coarse element $K\in \mathcal{T}_{H}$ defined by
\begin{alignat}{2} \label{pou}
-\nabla\cdot(\kappa(x)\nabla\chi_i) &= 0  &&\quad\text{ in }\;\;K, \\
\chi_i &= g_i &&\quad\text{ on }\partial K, \nonumber
\end{alignat}
where $g_i$ is affine over $\partial K$ with $g_i(O_j)=\delta_{ij}$ for all $i,j=1,\cdots, N$. Recall that $\{O_j\}_{j=1}^{N}$ are the set of coarse nodes on $\mathcal{T}_{H}$. Next we define the weighted coefficient:
\begin{equation}\label{defn:tildeKappa}
\widetilde{\kappa} =H^2 \kappa \sum_{i=1}^{N}  | \nabla \chi_i |^2.
\end{equation}
Furthermore, let $\widetilde{\kappa}^{-1}$ be defined by
\begin{equation}\label{eq:inv-tildeKappa}
\widetilde{\kappa}^{-1}(x):=
\left\{
\begin{aligned}
&\widetilde{\kappa}^{-1}, \quad &&\text{ when } \widetilde{\kappa}(x)\ne 0,\\
&1, \quad &&\text{ otherwise }.
\end{aligned}
\right.
\end{equation}
	\begin{algorithm}[H]
		\caption{Wavelet-based Edge Multiscale Finte Element Method (WEMsFEM)}
		\label{algorithm:wavelet}
		\textbf{Input}: The level parameter $\ell\in \mathbb{N}$; coarse neighborhood $\omega_i$ and its four coarse edges $\Gamma_{i,k}$ with
		$k=1,2,3,4$, i.e., $\cup_{k=1}^{4}\Gamma_{i,k}=\partial\omega_i$;
		the subspace $V_{\ell,k}\subset L^2(\Gamma_{i,k})$ up to level $\ell$ on each coarse edge $\Gamma_{i,k}$.
		
		\textbf{Output}: Multiscale solution $u_{\text{ms},\ell}^{\text{EW}}$.
		\begin{algorithmic}[1]
			\State Denote
			$V_{i,\ell}:=\oplus_{k=1}^{4}V_{\ell,k}.$
			Then the number of basis functions in $V_{i,\ell}$ is $4\times 2^{\ell}=2^{\ell+2}$.
			Denote these basis
			functions as $v_k$ for $k=1,\cdots, 2^{\ell+2}$.
			\State Calculate local multiscale basis $\mathcal{L}^{-1}_i (v_k)$ for all $k=1,\cdots,2^{\ell+2}$.
			Here, $\mathcal{L}^{-1}_i (v_k):=v$ satisfies:
			$
			\left\{ \begin{aligned}
			\mathcal{L}_i v&:=-\nabla\cdot(\kappa\nabla v)=0&& \mbox{in }\omega_i,\\
			v&=v_k&& \mbox{on }\partial\omega_i.
			\end{aligned}\right.
			$
\State Solve one local problem.

$
\left\{
\begin{aligned}
-\nabla\cdot(\kappa\nabla v^{i})&=\frac{\widetilde{\kappa}}{\int_{\omega_i}\widetilde{\kappa}\dx} \quad&&\text{ in } \omega_i,\\
-\kappa\frac{\partial v^{i}}{\partial n}&=|\partial\omega_i|^{-1}\quad&&\text{ on }\partial \omega_i.
\end{aligned}
\right.
$
			\State Build global multiscale space.
			$
			V_{\text{ms},\ell}^{\rm EW}  := \text{span} \{\chi_i\mathcal{L}^{-1}_i(v_k),\chi_i v^{i} : \,  \, 1 \leq i \leq N,\,\,\, 1 \leq k \leq  2^{\ell+2}\}.
			$
			\State Solve for \eqref{eq:multiscale} by conforming Galerkin method in
			$V_{\text{ms},\ell}^{\rm EW}$ to obtain $u_{\text{ms},\ell}^{\text{EW},n}$ for $n=1,\cdots,M_{\delta}$.
		\end{algorithmic}
	\end{algorithm}

	In the following, we define the $L^2(\partial\omega_i)$-orthogonal projection $\mathcal{P}_{i,\ell}$ onto the local multiscale space up to level $\ell$ by $\mathcal{P}_{i,\ell}: L^2(\partial\omega_i)\to V_{i,\ell}$ satisfying
	\begin{align}\label{eq:projectionEDGE}
	\mathcal{P}_{i,\ell}(v):=\sum\limits_{j=1}^{2^{\ell+2}}(v,\psi_{j})_{\partial\omega_i}\mathcal{L}_{i}^{-1}(\psi_{j}) \quad \text{ for all }v\in L^2(\partial\omega_i).
	\end{align}
	Here, we denote $\psi_{j}$ for $j=1,\cdots, 2^{\ell+2}$ as the Haar wavelets defined on the four edges of $\omega_i$ of level $\ell$ and the local operator $\mathcal{L}_i$ is defined as in Algorithm \ref{algorithm:wavelet}.
	
	Let $\mathcal{L}:=-\nabla\cdot(\kappa\nabla \cdot)$ be the elliptic operator defined on $H^1_{\kappa}(D)$, and  $\mathcal{P}_{\ell}$ be the projection onto the multiscale space $V_{\text{ms},\ell}^{\rm EW}$ given by
	\begin{align}\label{eq:projectionglo}
	\mathcal{P}_{\ell}(v):=\sum_{i=1}^{N}\chi_i\mathcal{P}_{i,\ell}(v) \text{ for all }v\in H^1_{\kappa}(D).
	\end{align}
	Then \cite[Proposition 5.2]{fu2018} implies that for any $v\in L^2(D)\cap L^2_{\widetilde{\kappa}^{-1}}(D)$, it holds
	\begin{align}
	\normE{\mathcal{L}^{-1}v-\mathcal{P}_{\ell}(\mathcal{L}^{-1}v)}{D}&\lesssim H\normI{\widetilde{\kappa}}{D}^{1/2}\normLii{v}{D}+2^{-\ell/2}\|\kappa\|_{L^{\infty}(\mathcal{F}_H)}
	\|v\|_{L^2(D)}. \label{eq:glo-energy}
	\end{align}
Furthermore, the following result holds.
	\begin{lemma}[Approximation properties of $P_{\ell}$ ]
		There holds
		\begin{align}
		\forall v\in L^2(D):\quad\normL{\mathcal{L}^{-1}v-\mathcal{P}_{\ell}(\mathcal{L}^{-1}v)}{D}&\lesssim (H+2^{-\ell/2}\|\kappa\|_{L^{\infty}(\mathcal{F}_H)})H\normL{v}{D}
		.\label{eq:glo-L2rietz}
		\end{align}
	\end{lemma}
	\begin{proof}
This assertion can be derived from the proofs of \cite[Lemma 5.1 and Proposition 5.2]{fu2018}.
	\end{proof}

	\section{Wavelet-based Edge Multiscale Parareal Algorithm}\label{sec:WEMP}
	We construct in this section the Wavelet-based Edge Multiscale Parareal (WEMP) Algorithm, cf. Algorithm \ref{algorithm:wavelet+parareal}, which is divided into two main steps. In the first step, the multiscale space $V_{\text{ms},\ell}^{\rm EW}$ for $\ell\in\mathbb{N}_{+}$ is built based on Section \ref{sec:multiscale}. This multiscale space serves as the trial space and test space for our conforming Galerkin method, cf. \eqref{eq:multiscale}. Then the parareal algorithm is utilized in the second step to solve the problem.
	
We first recap a few terminologies commonly appeared in parareal algorithm.
	
	The one step coarse solver on the time domain $(0,T)$ is
	\begin{align}
	U^{n+1}&=E_{\Delta }(T^n, U^n), \quad U^{0}=I_h u_0\nonumber,\\
	U^{n+1}_{\text{ms},\ell}&=E_{\Delta }^{{\text{ms},\ell}}(T^n, U^n_{\text{ms},\ell}), \quad U^{0}_{\text{ms},\ell}=\mathcal{P}_{\ell}u_0, \label{eq:coarseSolver}
	\end{align}
	which yields $U^{n+1}$ (or $U^{n+1}_{\text{ms},\ell}$) as a coarse approximation to $u(\cdot,T^{n+1})$, provided with an approximation $U^n$ (or $U^{n}_{\text{ms},\ell}$) of $u(\cdot,T^n)$.
	In matrix form, it reads
	\begin{align*}
	U^{n+1}&=(M+\Delta T \times A)^{-1}M(U^n+\Delta T\times F^{n+1}),\\
	U^{n+1}_{\text{ms},\ell}&=\Phi_{\text{ms},\ell}(\Phi^{T}_{\text{ms},\ell}M\Phi_{\text{ms},\ell}+\Delta T \times \Phi^{T}_{\text{ms},\ell}A\Phi_{\text{ms},\ell})^{-1}\Phi^{T}_{\text{ms},\ell}M(U^n_{\text{ms},\ell}+\Delta T\times F^{n+1}).
	\end{align*}
Here, $A$ and $M$ are the mass matrices and stiffness matrices corresponding to the discretization of the elliptic operator $-\nabla\cdot(\kappa\nabla \cdot)$ in the finite element space $V_h:=\text{span}\{\phi_1,\cdots,\phi_{\text{dof}_h}\}$. Here, $\text{dof}_h$ denotes the dimension of $V_h$. $(F^{n+1})_i:=\int_{D}f(\cdot,t_{n+1})\phi_i\dx$ for all $i=1,\cdots,\text{dof}_h$.
$\Phi_{\text{ms},\ell}$ denotes a matrix with columns composed of the coefficients of multiscale basis functions in $V_{\text{ms},\ell}^{\rm EW}$ in the finite element space $V_h$.  	

The one step fine solver
	\begin{align}
	\psi&=\mathcal{F}_{\delta }(s,\sigma,\phi)\nonumber,\\
	\psi_{\text{ms},\ell}&=\mathcal{F}^{{\text{ms},\ell}}_{\delta }(s,\sigma,\phi),\label{eq:fineSolver}
	\end{align}
 yields an approximation $\psi(\cdot,s+\sigma)$ (or $\psi_{\text{ms},\ell}(\cdot,s+\sigma)$) to the solution $u(\cdot,s+\sigma)$ with the initial condition $\psi(\cdot,s)=\phi$ (or $\psi_{\text{ms},\ell}(\cdot,s)=\mathcal{P}_{\ell}(\phi)$) and a uniform discrete time step $\delta $ for all $s\in (0,T)$ and $\sigma\in (0,T-s)$ in the infinite dimensional space $V$ (or in the ansatz space $V_{\text{ms},\ell}^{\rm EW}$) with $s/\delta t\in\mathbb{N}_{+}$.
	
	We also define the semi-discretization in space solver
	\begin{align}
u_{\text{ms},\ell}(\cdot,s+\sigma)&=\mathcal{F}^{{\text{ms},\ell}}(s,\sigma,\phi),\label{eq:exactSolver}
	\end{align}
	which yields an approximation $u_{\text{ms},\ell}(\cdot,s+\sigma)$ to the solution $u(\cdot,s+\sigma)$ with initial condition $u_{\text{ms},\ell}(\cdot,s)=\mathcal{P}_{\ell}(\phi)$ for all $s\in (0,T)$, $\sigma\in (0,T-s)$ in the ansatz space $V_{\text{ms},\ell}^{\rm EW}$. We will denote $\bar{E}_{\Delta }^{{\text{ms},\ell}}(T^n, U^n_{\text{ms},\ell})$ as the one step coarse solver with $f=0$. We will define $\bar{\mathcal{F}}^{{\text{ms},\ell}}(s,\sigma,\phi)$ and $\bar{\mathcal{F}}^{{\text{ms},\ell}}_{\delta}(s,\sigma,\phi)$ analogously.
	
	Note that the cheap multiscale coarse solver $E_{\Delta}^{\text{ms},\ell}$ is sequentially utilized over the global time interval $I$ to provide a rough approximation to $u(\cdot,T^{n+1})$, while the expensive accurate multiscale fine solver $\mathcal{F}_{\delta }^{\text{ms},\ell}$ is applied in each subinterval $[T^n,T^{n+1}]$ for $n=0,1,\cdots, M_{\Delta}-1$ independently. This local fine solver will embed more detailed information to the approximation of $u(\cdot,T^{n+1})$, which usually differs from the one obtained from the global coarse solver. In the process of parareal algorithm, a correction operator is very important to improve the approximation to $u(\cdot,T^{n+1})$ based on the discrepancy between the coarse solver and fine solver, which is defined by
	\[
	\mathcal{S}(T^{n},U^{n}_{\text{ms},\ell}):=\mathcal{F}^{{\text{ms},\ell}}_{\delta }(T^{n},\Delta T,U^{n}_{\text{ms},\ell})-E^{{\text{ms},\ell}}_{\Delta }(T^n, U^n_{\text{ms},\ell}) \quad\text{and} \quad U^{0}_{\text{ms},\ell}=\mathcal{P}_{\ell}u_0
	\]
	for all $n=0,1,\cdots,M_{\Delta}-1$.
	
	Now we are ready to present our main algorithm, i.e., Algorithm \ref{algorithm:wavelet+parareal}. To obtain a good approximation to the solution of \eqref{eqn:pde} at discrete time points $\{T^{n}\}$ for $n=1,\cdots, M_{\Delta}$, we first construct a proper multiscale space $V_{\text{ms},\ell}^{\rm EW}$ based on the WEMsFEM, i.e., Algorithm \ref{algorithm:wavelet}, which corresponds to Steps 1 to 3. This allows one to solve \eqref{eq:multiscale} using the constructed multiscale space $V_{\text{ms},\ell}^{\rm EW}$ and obtain an intermediate solution $u_{\text{ms},\ell}^{\text{EW},n}$ with certain accuracy depending on the spatial coarse mesh size $H$ and level parameter $\ell$. This solution will only be utilized in the convergence analysis.
	
	In order to further reduce the computational cost, we apply the parareal algorithm in the following. Given the iteration parameter $k$, we apply the global coarse solver \eqref{eq:coarseSolver} in Step 6 to obtain $U_k^{n+1}$, which is an approximation to the intermediate solution $u_{\text{ms},\ell}^{\text{EW},n+1}$ from Algorithm \ref{algorithm:wavelet}. Using the coarse solution $U_k^{n}$ as the initial condition, the fine solver \eqref{eq:fineSolver} subsequently is used to calculate the fine solution $u_k^{n+1}$ in paralell on each local time subinterval $[T^n, T^{n+1}]$.
	Then we calculate the discrepancy between the coarse solution and the fine solution in Step 8 on each discrete coarse time point $T^n$ for $n=1,2,\cdots,M_{\Delta}$, and denote it as $\mathcal{S}(T^{n-1},U^{n-1}_{k})$.
	Subsequently, this jump term is utilized in Step 9 to update the coarse solution via the global coarse solver \eqref{eq:coarseSolver}. This process will be performed iteratively until certain tolerance on the jump terms is satisfied.
	\begin{algorithm}
		\caption{Wavelet-based Edge Multiscale Parareal (WEMP) Algorithm}
		\label{algorithm:wavelet+parareal}
		{\bf Input:} The initial data $u_0$, the source term $f$; tolerance $\epsilon$; the level parameter $\ell\in \mathbb{N}$; coarse neighborhood $\omega_i$ and its four coarse edges $\Gamma_{i,j}$ with
		$j=1,2,3,4$, i.e., $\cup_{j=1}^{4}\Gamma_{i,j}=\partial\omega_i$;
		the subspace $V_{\ell,j}^{i}\subset L^2(\Gamma_{i,j})$ up to level $\ell$ on each coarse edge $\Gamma_{i,j}$.\\
		{\bf Output:} $U$.
		\begin{algorithmic}[1]
			\State Denote
			$V_{i,\ell}:=\oplus_{k=1}^{4}V_{\ell,k}^{i}.$
			Then the number of basis functions in $V_{i,\ell}$ is $4\times 2^{\ell}=2^{\ell+2}$.
			Denote these basis
			functions as $v_k$ for $k=1,\cdots, 2^{\ell+2}$.
			\State Calculate local multiscale basis $\mathcal{L}^{-1}_i (v_m)$ for all $m=1,\cdots,2^{\ell+2}$. Here, $\mathcal{L}^{-1}_i (v_m):=v$ satisfies:
			$
			\left\{ \begin{aligned}
			\mathcal{L}_i v&:=-\nabla\cdot(\kappa\nabla v)=0&& \mbox{in }\omega_i,\\
			v&=v_m&& \mbox{on }\partial\omega_i.
			\end{aligned}\right.
			$
			\State Build global multiscale space.
			$
			V_{\text{ms},\ell}^{\rm EW}  := \text{span} \{\chi_i\mathcal{L}^{-1}_i(v_k),\chi_i v^{i} : \,  \, 1 \leq i \leq N,\,\,\, 1 \leq k \leq  2^{\ell+2}\}.
			$
			\State $k=0$, err$=1$.
			\While{err$>\epsilon$}
			\State Compute $U_k^{n+1}$ for $n=0,\cdots, M_{\Delta}-1$:
			\begin{align*}
			U^{n+1}_k&=E_{\Delta}^{\text{ms},\ell}(T^n, U^n_k),\\
			U_k^{0}&=\mathcal{P}_{\ell}u_{0}.
			\end{align*}
			\State Compute $u_k^{n+1}$ for $n=0,\cdots, M_{\Delta}-1$ on each local time subinterval [$T^n$,$T^{n+1}$]:
			\begin{align*}
			u_k^{n+1}=\mathcal{F}_{\delta}^{\text{ms},\ell}(T^{n},\Delta T,U_{k}^{n}).
			\end{align*}
			\State Compute the jumps for $n=1,\cdots, M_{\Delta}$:
			\[
			\mathcal{S}(T^{n-1},U^{n-1}_{k}):=u_{k}^{n}-U_{k}^n.
			\]
			\State Compute the corrected coarse solutions $U^{n+1}_{k+1}$
			for $n=0,\cdots, M_{\Delta}-1$:
			\begin{align*}
			U^{n+1}_{k+1}&=\mathcal{S}(T^n,U^n_{k})+E_{\Delta}^{\text{ms},\ell}(T^n, U^n_{k+1}),\\
			U_{k+1}^{0}&=\mathcal{P}_{\ell}u_{0}.
			\end{align*}
			\State Calculate the error:
			\[{\rm err}:=1/{M_\Delta}\sum\limits_{n=1}^{M_{\Delta}}\|U_{k+1}^{n}-U_{k}^{n}\|_{\ell_2}.\]
			$k\gets k+1$
			\EndWhile
			\State $U_n:=U^{n}_{k}$ and $U:=[U_0,\cdots,U_{M_{\Delta}}]$.
		\end{algorithmic}
	\end{algorithm}

	\section{Convergence study}\label{sec:convergence}
	We present in this section the convergence analysis for Algorithm \ref{algorithm:wavelet+parareal}. To this end, we first prove the boundedness and Lipschitz continuity properties of the coarse solver $E_{\Delta }^{{\text{ms},\ell}}$ and the jump operator $\mathcal{S}$ in the multiscale space $V_{\text{ms},\ell}^{\rm EW}$:
	\begin{lemma}\label{lemma:parareal}
		For all $n\in \{1,\cdots,M_{\Delta}-1\}$, the following properties hold.
		\begin{itemize}
			\item[1.]The one step coarse solver $E_{\Delta }^{{\text{ms},\ell}}$ is Lipschitz in $V_{\text{ms},\ell}^{\rm EW}$. For all $v_1,v_2\in V_{\text{ms},\ell}^{\rm EW}$, there holds
			\[
			\quad\normL{E_{\Delta }^{{\text{ms},\ell}}(T^n,v_1)-E_{\Delta }^{{\text{ms},\ell}}(T^n,v_2)}{D}\leq
			\normL{v_1-v_2}{D}.
			\]
			\item[2.] The jump operator $\mathcal{S}$ is an approximation of order 1 with Lipschitz regularity. For all $v_1,v_2\in V_{\text{ms},\ell}^{\rm EW}$, there holds
			\begin{align}\label{eq:app-jump}
			\quad&\normL{\mathcal{S}(T^n,v_1)-\mathcal{S}(T^n,v_2)}{D}\leq \frac{1}{T^n T^{n+1}}(\Delta T)^2\normL{v_1-v_2}{D}.
			\end{align}
		\end{itemize}
	\end{lemma}
	\begin{proof}
		1. Let $e_{{\text{ms}}}^{n+1}:=E_{\Delta }^{{\text{ms},\ell}}(T^n,v_1)-E_{\Delta }^{{\text{ms},\ell}}(T^n,v_2)$, then it holds
		\begin{align*}
		\forall w_{\text{ms}}\in V_{\text{ms},\ell}^{\rm EW}:\int_{D}e_{\text{ms}}^{n+1} w_{\text{ms}}\dx+\Delta T\int_{D}\kappa \nabla e_{\text{ms}}^{n+1}\cdot \nabla w_{\text{ms}}\dx
		=\int_{D}(v_1-v_2) w_{\text{ms}}\dx.
		\end{align*}
		Choosing $w_{\text{ms}}:=e_{{\text{ms}}}^{n+1}$ leads to
		\begin{align*}
		\normL{e_{{\text{ms}}}^{n+1}}{D}^2+\Delta T\normE{e_{{\text{ms}}}^{n+1}}{D}^2=\int_{D}e_{{\text{ms}}}^{n+1}(v_1-v_2) \dx.
		\end{align*}
		Finally an application of the Young's inequality proves the first assertion.
		
		\noindent 2. To prove the second assertion, let
		\begin{align*}
		e_{\text{ms}}^{n+1}&:=\mathcal{S}(T^n,v_1)-\mathcal{S}(T^n,v_2)\\
		&=\Big(\mathcal{F}^{{\text{ms},\ell}}_{\delta }(T^{n},\Delta T,v_1)-\mathcal{F}^{{\text{ms},\ell}}_{\delta }(T^{n},\Delta T,v_2)\Big)-\Big(E^{{\text{ms},\ell}}_{\Delta }(T^n, v_1)-E^{{\text{ms},\ell}}_{\Delta }(T^n, v_2)\Big)\\
		&=\bar{\mathcal{F}}^{{\text{ms},\ell}}_{\delta }(T^{n},\Delta T,v_1-v_2)-\bar{E}^{{\text{ms},\ell}}_{\Delta }(T^n, v_1-v_2)\\
		&=\Big(\bar{\mathcal{F}}^{{\text{ms},\ell}}_{\delta }(T^{n},\Delta T,v_1-v_2)-\bar{\mathcal{F}}^{{\text{ms},\ell}}(T^{n},\Delta T,v_1-v_2)\Big)-\Big(\bar{E}^{{\text{ms},\ell}}_{\Delta }(T^n, v_1-v_2)\\
&-\bar{\mathcal{F}}^{{\text{ms},\ell}}(T^{n},\Delta T,v_1-v_2)\Big)\\
		&=:e_{\text{ms},\delta}^{n+1}-e_{\text{ms},\Delta}^{n+1}.
		\end{align*}
		To estimate $e_{\text{ms}}^{n+1}$, we only need to derive the estimate for $e_{\text{ms},\delta}^{n+1}$ and $e_{\text{ms},\Delta}^{n+1}$, separately.
		
To this end, let  $v_{\text{ms},i}^{n+1}:=v_{\text{ms},i}({\cdot,T^{n+1}}):=\mathcal{F}^{{\text{ms},\ell}}(T^n, \Delta T, v_i)$ for $i=1,2$, we first construct the equation for $e_{\text{ms},\Delta}^{n+1}$ by the definitions of the coarse solver \eqref{eq:coarseSolver} and fine solver \eqref{eq:fineSolver}. There holds
		\begin{align*}
		\forall w_{\text{ms}}\in V_{\text{ms},\ell}^{\rm EW}:\int_{D}e_{\text{ms},\Delta}^{n+1} w_{\text{ms}}\dx+\Delta T\int_{D}\kappa \nabla e_{\text{ms},\Delta}^{n+1}\cdot \nabla w_{\text{ms}}\dx
		=\int_{D}w_0\cdot w_{\text{ms}}\dx.
		\end{align*}
		Here,
		\begin{align*}
		w_{0}&:=\Delta T\Big(-\partial_{t}v_{\text{ms},1}|_{t=T^{n+1}}+\frac{v_{\text{ms},1}^{n+1}-v_1}{\Delta T}+\partial_{t}v_{\text{ms},2}|_{t=T^{n+1}}-\frac{v_{\text{ms},2}^{n+1}-v_1}{\Delta T}\Big)\\
		&=-\int_{T^n}^{T^{n+1}}(s-T^n)\partial_{ss}(v_{\text{ms},1}-v_{\text{ms},2})(\cdot,s)\ds.
		\end{align*}
		Note that
		\[
		\normL{w_{0}}{D}\leq \Delta T\int_{T^n}^{T^{n+1}}\normL{\partial_{ss}(v_{\text{ms},1}-v_{\text{ms},2})(\cdot,s)}{D}\ds.
		\]
		An adaptation of the proof to \cite[Lemma 3.2]{thomee1984galerkin} shows
		\[
		\normL{\partial_{tt}(v_{\text{ms},1}-v_{\text{ms},2})(\cdot,t)}{D}\lesssim t^{-2}\normL{v_1-v_2}{D} \text{ for all }t> 0.
		\]
		Consequently, we derive
		\[
		\normL{w_{0}}{D}\leq \frac{1}{T^{n}T^{n+1}}(\Delta T)^2\normL{v_1-v_2}{D}.
		\]
		Choosing $w_{\text{ms}}:=e_{{\text{ms},\Delta}}^{n+1}$ leads to
		\begin{align*}
		\normL{e_{{\text{ms},\Delta}}^{n+1}}{D}^2+\Delta T\normE{e_{{\text{ms},\Delta}}^{n+1}}{D}^2=\int_{D}e_{{\text{ms},\Delta}}^{n+1}w_{0} \dx.
		\end{align*}
		Consequently, an application of the Young's inequality implies
		\begin{align*}
			\normL{e_{{\text{ms},\Delta}}^{n+1}}{D}\leq \frac{1}{T^{n}T^{n+1}}(\Delta T)^2\normL{v_1-v_2}{D}.
		\end{align*}
		Analogously, we can obtain the estimate for $e_{{\text{ms},\delta}}^{n+1}$, which reads
		\begin{align*}
		\normL{e_{{\text{ms},\delta}}^{n+1}}{D}\leq \frac{1}{T^{n}T^{n+1}}\delta t\Delta T\normL{v_1-v_2}{D}.
		\end{align*}
		Note that $\delta t\ll \Delta T$, then a combination of the two estimates above with the triangle inequality, shows the second assertion.		
	\end{proof}
	\begin{remark}
	Lemma \ref{lemma:parareal} indicates that the approximation property of the jump operator $\mathcal{S}(T^n,\cdot)$ deteriorates when $T^n$ is small.
	\end{remark}
We present in the next theorem the convergence rate of Algorithm \ref{algorithm:wavelet+parareal} to Problems \eqref{eqn:pde} in pointwise-in-time in $L^2(D)$-norm. To derive it, we first decompose the error from Algorithm \ref{algorithm:wavelet+parareal} as a summation of the error from WEMsFEM and the error from parareal algorithm. Then we estimate the former by energy method using argument analogous to \cite[Theorem 1.5]{thomee1984galerkin}, and the latter can be estimated in a similar manner as in \cite[Theorem 1]{gander2008nonlinear}.
\begin{theorem}
		[Error estimate of Algorithm \ref{algorithm:wavelet+parareal} to Problem \eqref{eqn:pde} in $L^2(D)$-norm]\label{prop:wavelet-based}
		Let Assumption \ref{ass:coeff} hold. Assume that the source term $f\in L^2([0,T]; \dot{H}^2(D))$ satisfying $\partial_{t}f\in L^2([0,T]; L^2(D))$
and initial data $u_0\in \dot{H}^3(D)\cap H^1_0(D)$. Let $\ell\in \mathbb{N}_{+}$ be the level parameter. The coarse time step size and fine time step size are $\Delta T$ and $\delta t$. Let $u(\cdot,t)\in V$ be the solution to Problem \eqref{eqn:pde} and let $U_k^n$ be the solution from Algorithm \ref{algorithm:wavelet+parareal} with iteration $k\in\mathbb{N}_{+}$.
 There holds
		\begin{equation}\label{eq:waveletErr}
		\begin{aligned}
		\normL{u(\cdot,T^n)-U_{k}^{n}}{D}
		&\lesssim (H+2^{-\ell/2}\|\kappa\|_{L^{\infty}(\mathcal{F}_H)})H|u_0|_2+\delta t\Big(|u_0|_3+\|f\|_{L^2(I;\dot{H}^2(D))}+\|\partial_{t}f\|_{L^2(I;L^2(D))}\Big)\\
		&+
		\Big(\frac{1}{T^{n-1}}\Big)^{k+1}\frac{1}{k!}
		(\Delta T)^{k+1} \normL{u_0}{D}.
		\end{aligned}
		\end{equation}
	\end{theorem}
	\begin{proof}
		We first define the multiscale solution to Problem \eqref{eqn:pde} using Algorithm \ref{algorithm:wavelet}. Find $u_{\text{ms},\ell}^{\text{EW},m}\in V_{\text{ms},\ell}^{\rm EW}$ for $m=1,\cdots,M_{\delta}$, satisfying
		\begin{align}\label{eq:multiscaleSoln}
		\forall w_{\text{ms}}\in V_{\text{ms},\ell}^{\rm EW}:(\frac{u_{\text{ms},\ell}^{\text{EW},m}-u_{\text{ms},\ell}^{\text{EW},m-1}}{\delta t},w_{\text{ms}})_D+a(u_{\text{ms},\ell}^{\text{EW},m},w_{\text{ms}})&=(f(\cdot,t_m),w_{\text{ms}})_{D} \\
		u_{\text{ms},\ell}^{\text{EW},0}&=\mathcal{P}_{\ell}(u_0).\nonumber
		\end{align}
		Then we only need to estimate $\normL{u(\cdot,T^n)-u_{\text{ms},\ell}^{\text{EW},m}}{D}$ and $\normL{u_{\text{ms},\ell}^{\text{EW},m}-U_{k}^{n}}{D}$ for $m:=\Delta T/{\delta t} \times n$. Note that $T^n=t_m$. Therefore, we can replace $T^n$ with $t_m$. Similarly, let $m\rq{}:=\Delta T/{\delta t} \times (n-1)$, then it holds $t_{m\rq{}}=T^{n-1}$.
		
		\noindent Step 1. To estimate the first term $\normL{u(\cdot,t_m)-u_{\text{ms},\ell}^{\text{EW},m}}{D}$,
		let $e_{\text{ms}}^{m}:=u(\cdot,t_m)-u_{\text{ms},\ell}^{\text{EW},m}$, then for all $j\geq 1$ it holds
		\begin{align*}
		\forall w_{\text{ms}}\in V_{\text{ms},\ell}^{\rm EW}:\int_{D}(e_{\text{ms}}^{j}-e_{\text{ms}}^{j-1}) w_{\text{ms}}\dx+\delta t\int_{D}\kappa \nabla e_{\text{ms}}^{j}\cdot \nabla w_{\text{ms}}\dx
		&=\int_{D}w^{j}\cdot w_{\text{ms}}\dx\\
		e_{\text{ms}}^{0}&=u_0-\mathcal{P}_{\ell}(u_0)
		\end{align*}
		with
		\[
		w^{j}=u(\cdot,t_{j})-u(\cdot,t_{j-1})-\delta t\times\partial_{t} u(\cdot,t_{j}).
		\]
		Subsequently, similar proof to \cite[Theorem 1.5]{thomee1984galerkin} leads to
		\begin{align}
		\normL{e_{{\text{ms}}}^{m}}{D}&+\delta t\normE{e_{{\text{ms}}}^{m}}{D}
\lesssim\normL{e_{\text{ms}}^{0}}{D}+\delta t\int_{0}^{T}\normL{\partial_{ss}u}{D}\ds\nonumber\\
&\lesssim(H+2^{-\ell/2}\|\kappa\|_{L^{\infty}(\mathcal{F}_H)})H|u_0|_2+\delta t\int_{0}^{T}\normL{\partial_{ss}u}{D}\ds\nonumber\\
&\lesssim (H+2^{-\ell/2}\|\kappa\|_{L^{\infty}(\mathcal{F}_H)})H|u_0|_2+\delta t\Big(|u_0|_3+\|f\|_{L^2(I;\dot{H}^2(D))}+\|\partial_{t}f\|_{L^2(I;L^2(D))}\Big).
\label{eq:impt1}
		\end{align}
Here, we have applied estimates \eqref{eq:glo-L2rietz}, \eqref{eq:reg-u0} and \eqref{eq:regularity-parabolic} in the last inequality.

		Step 2. We estimate the error induced by parareal algorithm in the multiscale method, i.e., the second term $\normL{u_{\text{ms},\ell}^{\text{EW},m}-U_{k}^{n}}{D}$. We will prove:
		\begin{align}\label{eq:parareal-est}
		\normL{u_{\text{ms},\ell}^{\text{EW},m}-U_{k}^{n}}{D}\lesssim \Big(\frac{1}{T^{n-1}}\Big)^{k+1}\frac{1}{k!}
		(\Delta T)^{k+1} \normL{u_0}{D}.
		\end{align}
 We can obtain from Algorithm \ref{algorithm:wavelet+parareal}:
		\begin{align*}	u_{\text{ms},\ell}^{\text{EW},m}-U_{k+1}^{n}=\mathcal{S}(T^{n-1},u_{\text{ms},\ell}^{\text{EW},m\rq{}})-\mathcal{S}(T^{n-1},U_{k}^{n-1})
		+E_{\Delta }^{{\text{ms},\ell}}(T^{n-1},u_{\text{ms},\ell}^{\text{EW},m\rq{}})-E_{\Delta }^{{\text{ms},\ell}}(T^{n-1},U_{k+1}^{n-1}).
		\end{align*}
		Consequently, an application of Lemma \ref{lemma:parareal} leads to
		\begin{align*}
		&\normL{u_{\text{ms},\ell}^{\text{EW},m}-U_{k+1}^{n}}{D}\\
		&\leq\normL{\mathcal{S}(T^{n-1},u_{\text{ms},\ell}^{\text{EW},m\rq{}})-\mathcal{S}(T^{n-1},U_{k}^{n-1})}{D}
		+\normL{E_{\Delta }^{{\text{ms},\ell}}(T^{n-1},u_{\text{ms},\ell}^{\text{EW},m\rq{}})-E_{\Delta }^{{\text{ms},\ell}}(T^{n-1},U_{k+1}^{n-1})}{D}\\
		&\lesssim \Delta T \int_{T^{n-1}}^{T^n}\normL{
		\partial_{ss}\bar{\mathcal{F}}^{{\text{ms},\ell}}(T^{n-1}, \Delta T, u_{\text{ms},\ell}^{\text{EW},m\rq{}}-U_{k}^{n-1})
		}{D}\mathrm{d}s
		+\normL{u_{\text{ms},\ell}^{\text{EW},m\rq{}}-U_{k+1}^{n-1}}{D}\\
&\lesssim \frac{1}{T^{n-1}T^n}(\Delta T)^2\normL{u_{\text{ms},\ell}^{\text{EW},m\rq{}}-U_{k}^{n-1}}{D}
		+\normL{u_{\text{ms},\ell}^{\text{EW},m\rq{}}-U_{k+1}^{n-1}}{D}.
		\end{align*}
		Whereas similar trick as in the proof of \cite[Theorem 1]{gander2008nonlinear} can be utilized here to prove estimate \eqref{eq:parareal-est}. Then a combination with \eqref{eq:impt1} results in the desired estimate.
	\end{proof}
Note that Theorem \ref{prop:wavelet-based} indicates that the Error estimate for Algorithm \ref{algorithm:wavelet+parareal} to Problems \eqref{eqn:pde} in $L^2(D)$-norm will deteriorate when the time step approaches the original $t=0$. This blow-up of error is produced by the parareal algorithm (Step 2 in the proof to Theorem \ref{prop:wavelet-based}), which essentially arises from the approximation property of the jump operator \eqref{eq:app-jump}. This estimate can be improved to
\begin{align}\label{eq:app-jump-imp}
			\quad&\normL{\mathcal{S}(T^n,v_1)-\mathcal{S}(T^n,v_2)}{D}\leq (\Delta T)^2|v_1-v_2|_{4},
			\end{align}
when $v_1,v_2\in \dot{H}_4(D)$.

However, the estimate above has different norms on both sides of the inequality. This makes the argument in Step 2, proof to Theorem \ref{prop:wavelet-based} invalid.

When iteration $k=0$, an application to Step 1, proof to Theorem \ref{prop:wavelet-based} results in
\begin{equation*}
		\normL{u(\cdot,T^n)-U_{0}^{n}}{D}
		\lesssim (H+2^{-\ell/2}\|\kappa\|_{L^{\infty}(\mathcal{F}_H)})H|u_0|_2+\Delta t\Big(|u_0|_3+\|f\|_{L^2(I;\dot{H}^2(D))}+\|\partial_{t}f\|_{L^2(I;L^2(D))}\Big).		
		\end{equation*}		
\begin{remark}
Algorithm \ref{algorithm:wavelet+parareal} outweighs Algorithm \ref{algorithm:wavelet} only when the former achieves similar accuracy to the latter within a very few iteration $k\ll M_{\Delta }$. Therefore, we are not interested in the case when $k\geq M_{\Delta }$ or the error at time level $T^n$ with $k\geq n$.
\end{remark}
	\section{Numerical Results}\label{sec:num}
	In this section, we perform a series of numerical experiments to demonstrate the performance of the proposed WEMP Algorithm, i.e. Algorithm \ref{algorithm:wavelet+parareal}.
	
	We consider the parabolic equation (\ref{eqn:pde}) in the space domain $D:=[0,1]^2$ and the time domain $[0,T]=[0,1]$.  The permeability coefficient $\kappa$ is depicted in Figure \ref{fig:permeability&InitialSol} (left figure), which is high-contrast and heterogenous. The smooth initial data tested in our numerical experiments is chosen to be $u_0:=x(1-x)y(1-y)$. We refer to Figure \ref{fig:permeability&InitialSol} (right figure) for an illustration.
	
	\begin{figure}[H]
		\centering
		\includegraphics[trim={1.6cm 0.4cm 1.6cm 0.4cm},clip,width=0.4 \textwidth]{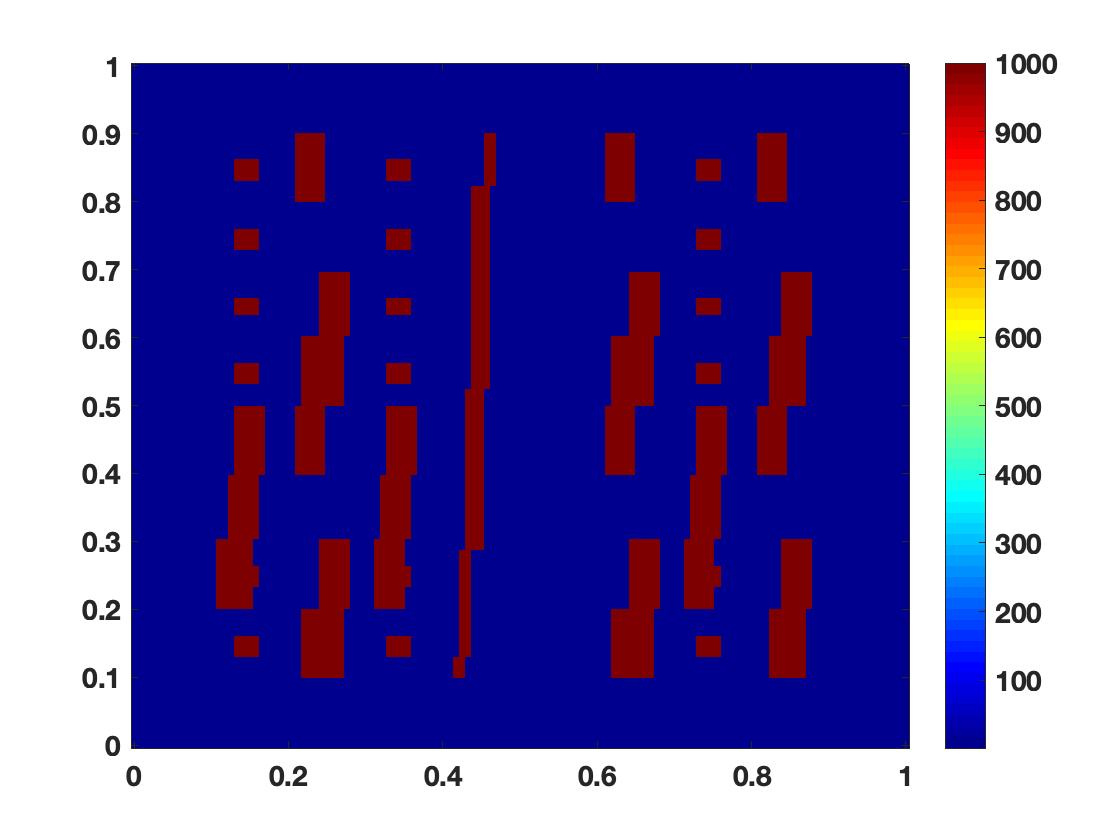}
		\includegraphics[trim={1.6cm 0.4cm 1.6cm 0.4cm},clip,width=0.4 \textwidth]{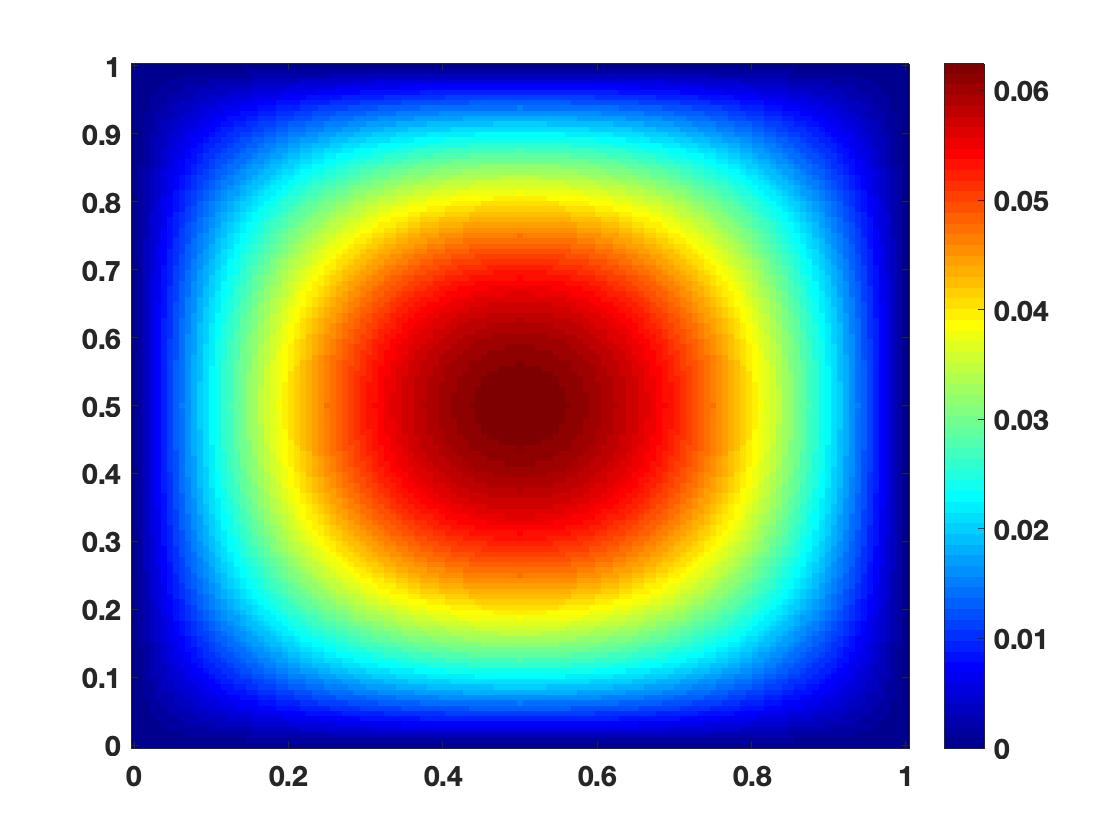}
		\caption{The heterogeneous permeability field $\kappa$ and the initial data $u_0=x(1-x)y(1-y)$.}
		\label{fig:permeability&InitialSol}
	\end{figure}
Let $\mathcal{T}_H$ be a decomposition of the domain $D$ into non-overlapping shape-regular rectangular elements with maximal mesh size $H:=2^{-4}$. These coarse rectangular elements are further partitioned into a collection of connected fine rectangular elements $\mathcal{T}_h$ using fine mesh size $h:=2^{-7}$. Similarly, we define $V_h$ to be a conforming piecewise affine finite element associated with $\mathcal{T}_h$. In our numerical experiments, space meshes $\mathcal{T}_H$ and $\mathcal{T}_h$ are fixed. To keep our presentation concise, we will only present the numerical results with a fixed level parameter $\ell:=2$.
 The temporal discretization is presented in Section \ref{section:problem setting} with $T:=1$. The coarse time step size and fine time step size are $\Delta T$ and $\delta t$. Note that $\delta t\ll \Delta T$.
		
We introduce the following notations to calculate the errors. The relative errors for the multiscale solution in $L^2(D)$-norm and $H^1_{\kappa}(D)$-norm are
 \[
 \text{Rel}^{\text{EW}}_{L^2} (T^n):=\frac{ \norm{u^n_h-u_{\text{ms},\ell}^{\text{EW,n}}}_{L^2(D)}}{ \norm{u^n_h} _{L^2(D) }}\times 100 \quad\text{ and }\quad
  \text{Rel}^{\text{EW}}_{H_{\kappa}^1} (T^n):=\frac{ \norm{u^n_h-u_{\text{ms},\ell}^{\text{EW,n}}}_{H_{\kappa}^1(D)}}{ \norm{u^n_h} _{H_{\kappa}^1(D) }}\times 100.
  \]
 Analogously, the relative errors for our proposed algorithm with iteration $k\in \mathbb{N}$  in $L^2(D)$-norm and $H^1_{\kappa}(D)$-norm are
  \[\text{Rel}^k_{L^2}(T^n):=\frac{ \norm{u^m_h-U_k^n}_{L^2(D)}}{ \norm{u^m_h} _{L^2(D) }}\times 100
  \quad\text{ and }\quad
  \text{Rel}^k_{H_{\kappa}^1}(T^n):=\frac{ \norm{u^m_h-U_k^n}_{H_{\kappa}^1(D)}}{ \norm{u^m_h} _{H_{\kappa}^1(D) }}
  \times 100\]
  with $m:=\Delta T/{\delta t} \times n$.
	
	 Our numerical experiements include testing nonzero source term in section  \ref{subsection:Numerical tests with nonzero source term} and zero source term in section \ref{subsection: Numerical tests with zero source term}. We test what differences backward Euler Galerkin Method and Crank-Nicolson Galerkin Method will make to the error. We also study how fine solver and coarse solver will influence the error and iteration number.
	
	\subsection{Numerical tests with nonzero source term} \label{subsection:Numerical tests with nonzero source term}
	To define nonzero source term, we take time-dependent smooth function
	\[
	 f(x,y,t ):=200\pi ^2 \sin (\pi x)\sin(\pi y)\sin (10\pi tx).
	\]
 Since there is no analytic solution to system (\ref{eqn:pde}), we need to find an approximation of the exact solutions. To this end,  we take time step size $\delta t=10^{-4}$ and use backward Euler Galerkin Method in \eqref{eqn:weakform_h} to obtain the reference solutions $u^n_h$. Note that we use a much finer time step size to simulate the
reference solution. We plot the reference solutions $u^n_h$ for $n=10^3, 3\times 10^3, 5\times 10^3$ and $10^4$ in Figure \ref{fig:fine_solution}.
	  \begin{figure}[H]
		\centering
		\includegraphics[trim={3cm 1.8cm 2.8cm 2cm},clip,width=0.24 \textwidth]{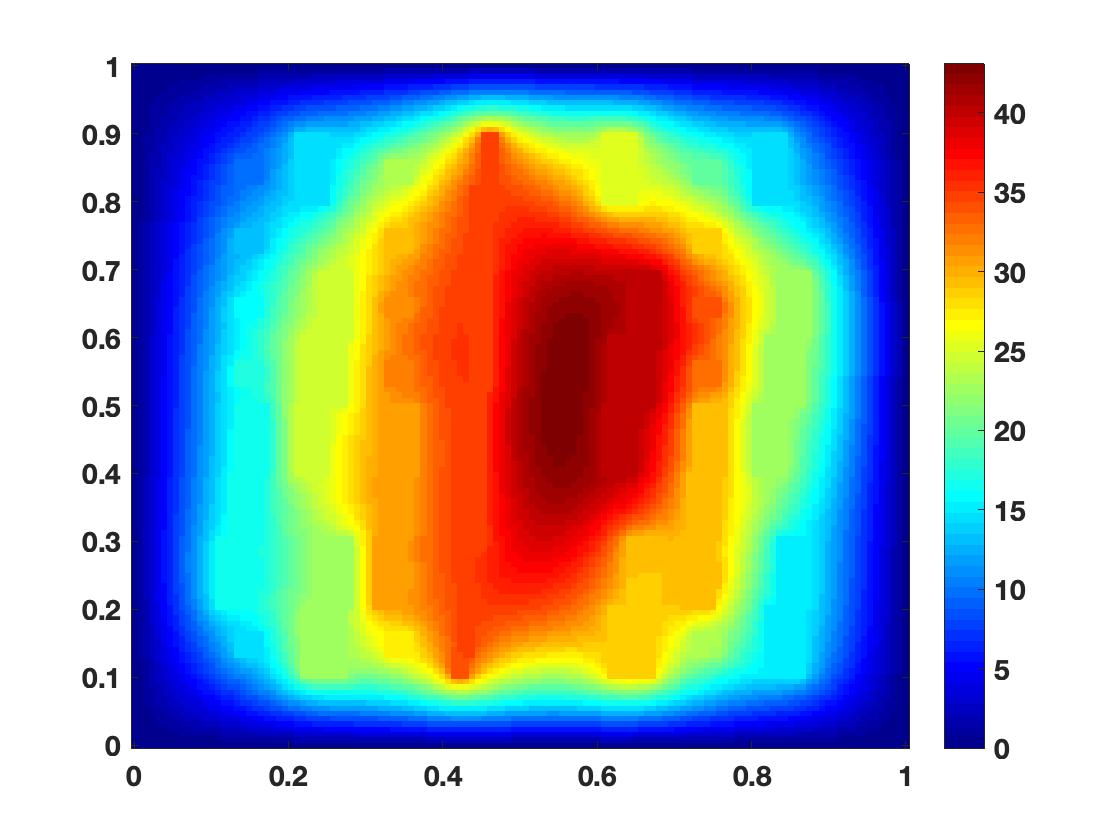}
		\includegraphics[trim={3cm 1.8cm 2.8cm 2cm},clip,width=0.24 \textwidth]{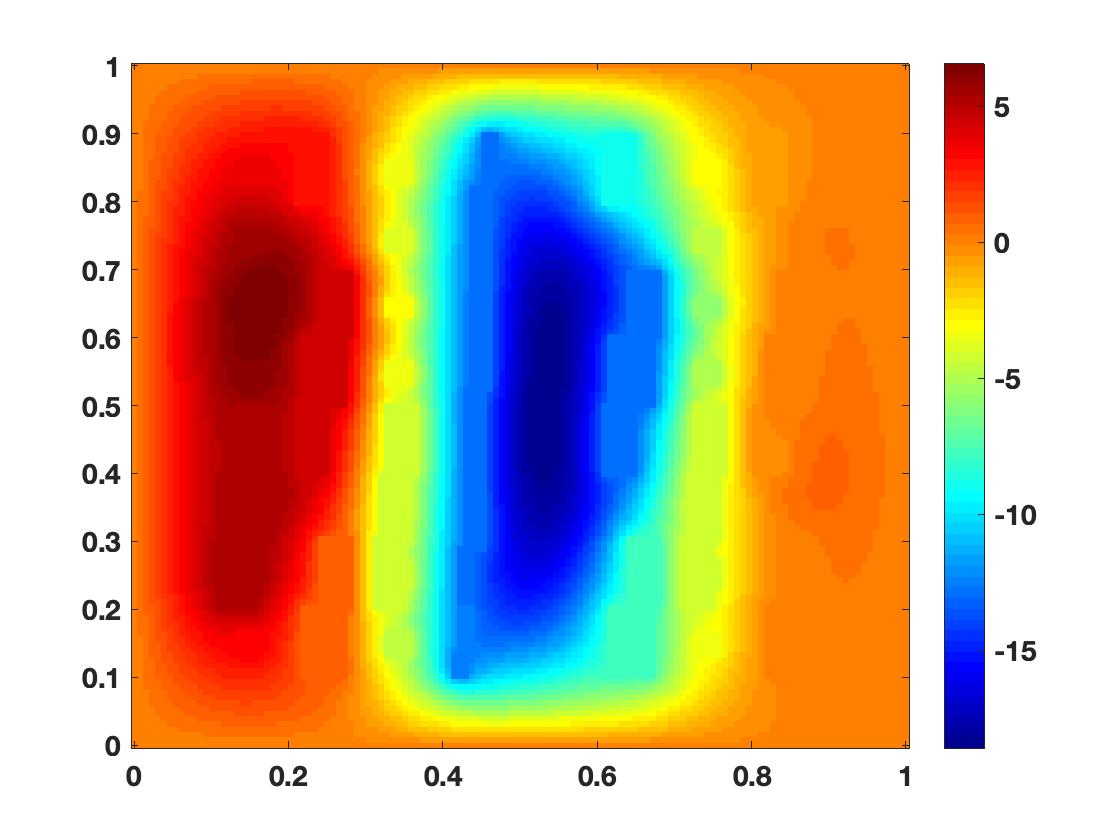}
		\includegraphics[trim={3cm 1.8cm 2.8cm 2cm},clip,width=0.24 \textwidth]{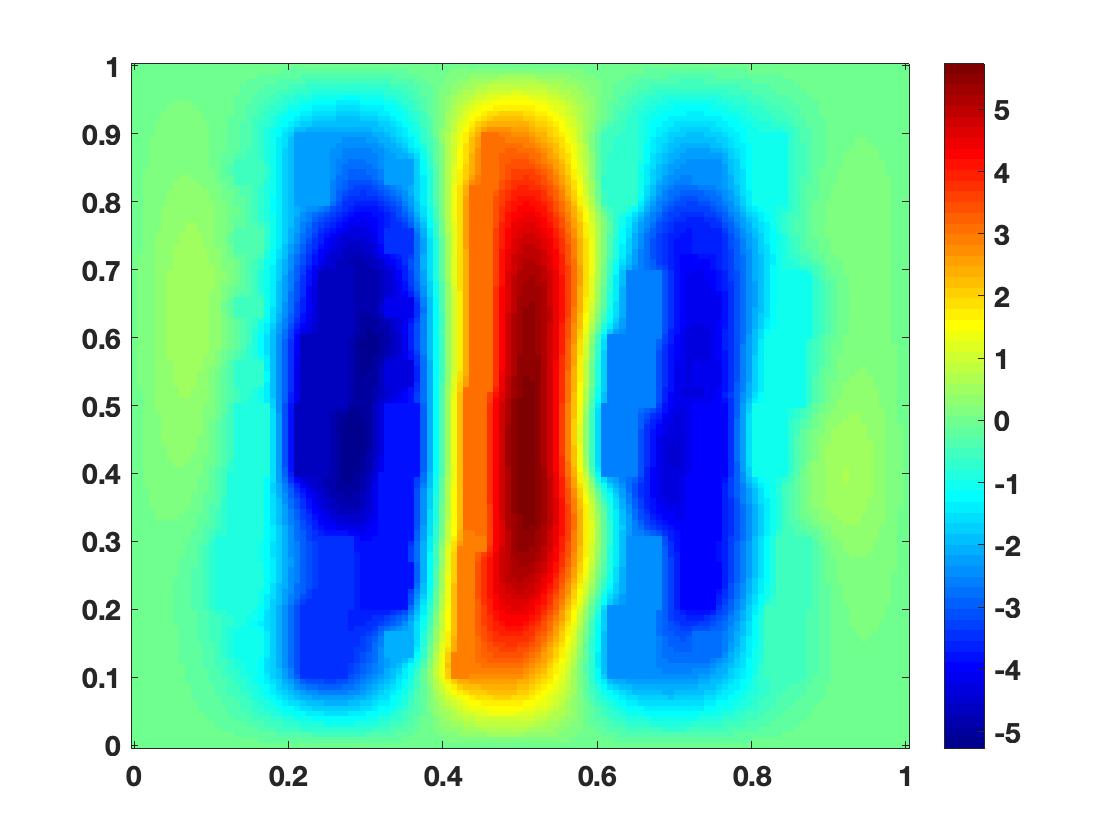}
		\includegraphics[trim={3cm 1.8cm 2.8cm 2cm},clip,width=0.24 \textwidth]{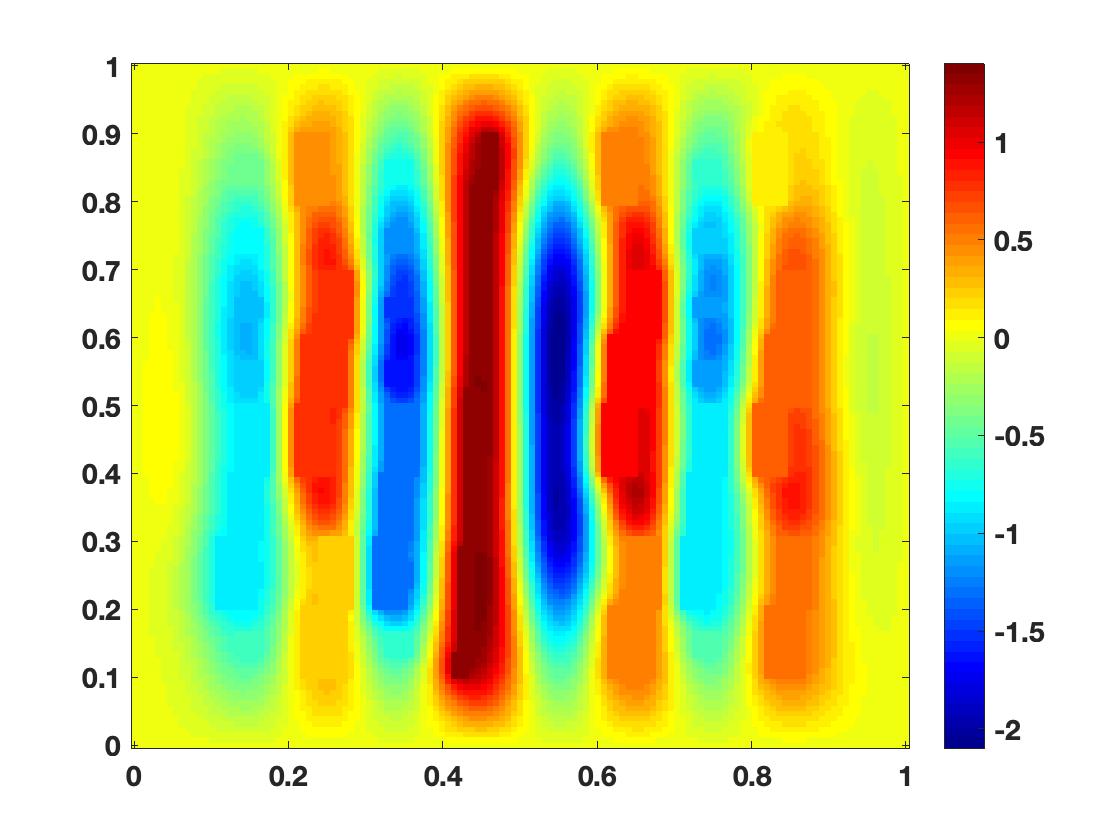}
		\caption{ Numerical solution $u_h^n$ to \eqref{eqn:weakform_h} for $n=10^3, 3\times 10^3, 5\times 10^3$ and $10^4$ with $\delta t=10^{-4}$.}
		\label{fig:fine_solution}
	\end{figure}
In the rest of this subsection, we will present numerical tests using backward Euler scheme with $\frac{\Delta T}{\delta t}=100$ in Experiment 1, Crank-Nicolson scheme with $\frac{\Delta T}{\delta t}=100$ in Experiment 2 and backward Euler scheme with $\frac{\Delta T}{\delta t}=10$ in Experiment 3. For all the three experiments, our proposed algorithm, i.e. Algorithm \ref{algorithm:wavelet+parareal}, can generate numerical solutions by a few iterations at least of the same accuracy as the multiscale solutions from Algorithm \ref{algorithm:wavelet}.
	
\subsubsection*{Experiment 1: Backward Euler with $\frac{\Delta T}{\delta t}=100$}
We test in this experiment the performance of Algorithm \ref{algorithm:wavelet+parareal} with a fine time step $\delta t=10^{-3}$ and a coarse time step $\Delta T=0.1$. The backward Euler scheme is utilized for the time discretization.

The multiscale solutions $u_{\text{ms},\ell}^{\text{EW,m}}$ for $m=100,300,500,1000$ from Algorithm \ref{algorithm:wavelet} are presented in Figure \ref{fig:EWsol_NonzeroSource_level2_BE_coarser}. In comparison, we present the numerical solutions $ U_k^n$ for $n=1, 3, 5, 10$ from Algorithm \ref{algorithm:wavelet+parareal} with iteration number $k=0,1$ and $2$ in Figure \ref{fig:PararealSol_NonzeroSource_level2_BE_coarser}. One can observe that $ U_k^n$ converges to the multiscale solution $u_{\text{ms},\ell}^{\text{EW,n}}$ as the iteration $k$ increases.
\begin{figure}[H]
		\centering
		\includegraphics[trim={3cm 1.8cm 2.8cm 2cm},clip,width=0.24 \textwidth]{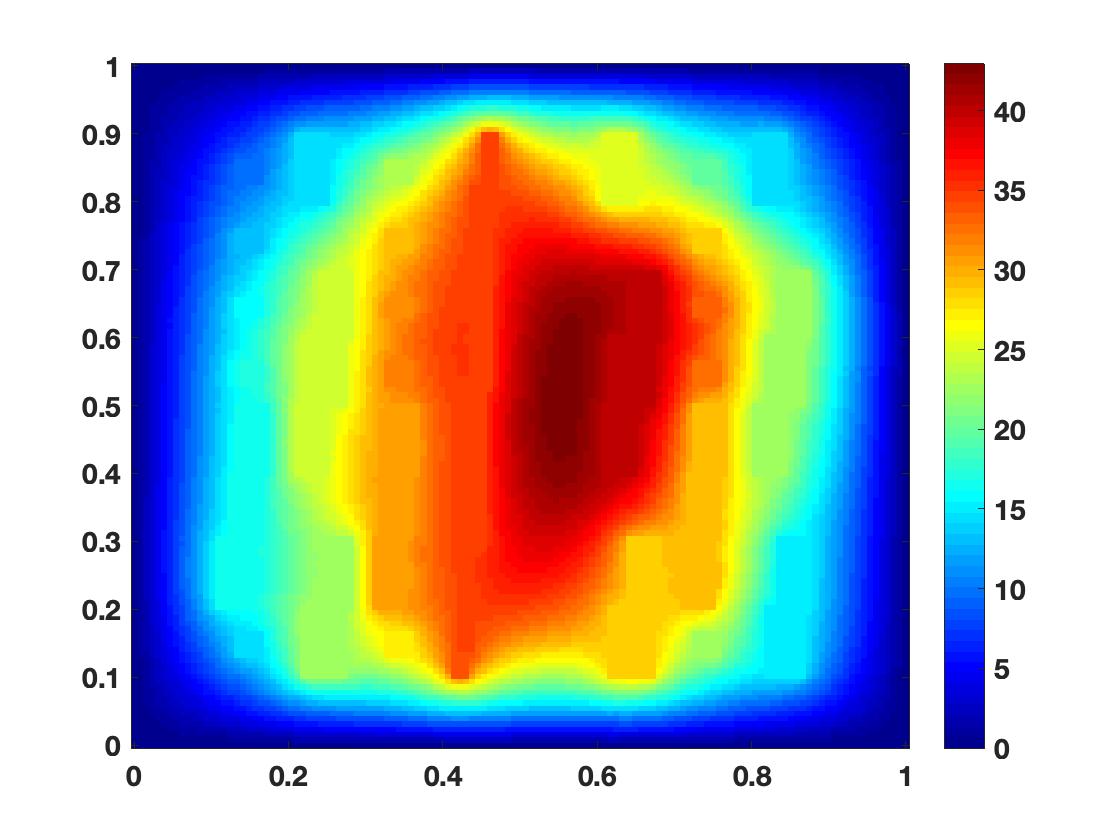}
		\includegraphics[trim={3cm 1.8cm 2.8cm 2cm},clip,width=0.24 \textwidth]{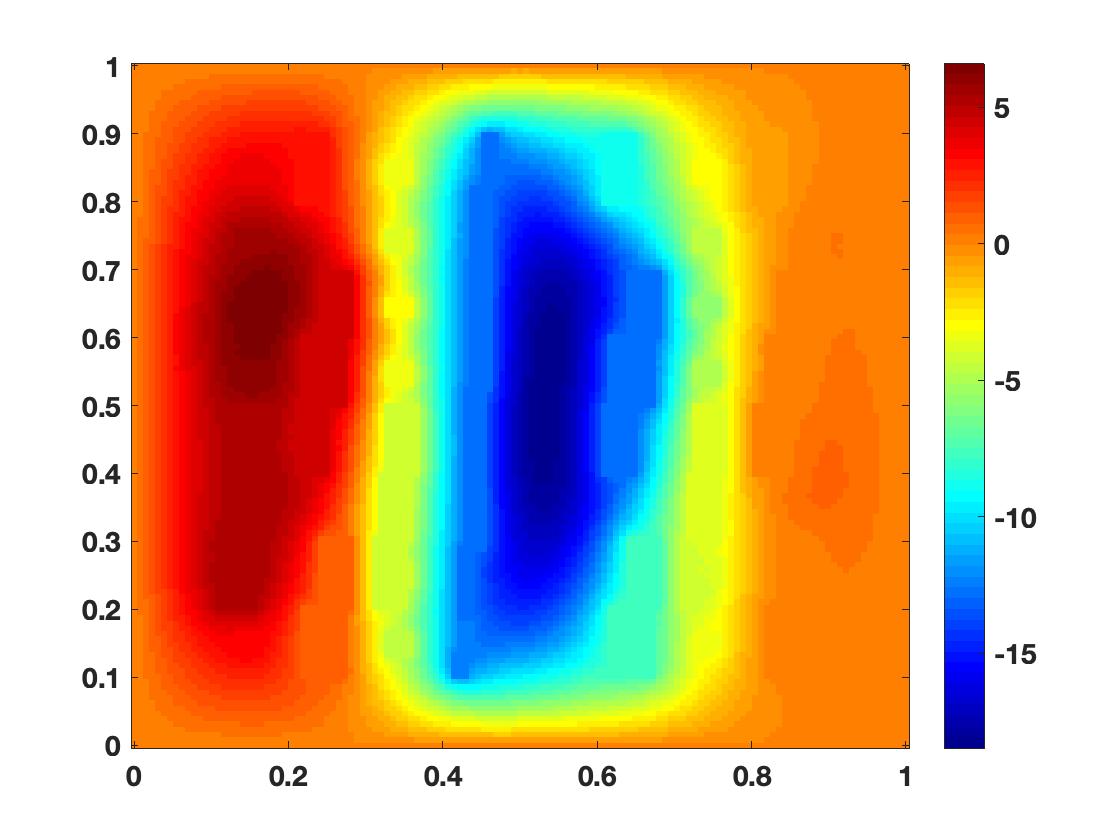}
		\includegraphics[trim={3cm 1.8cm 2.8cm 2cm},clip,width=0.24 \textwidth]{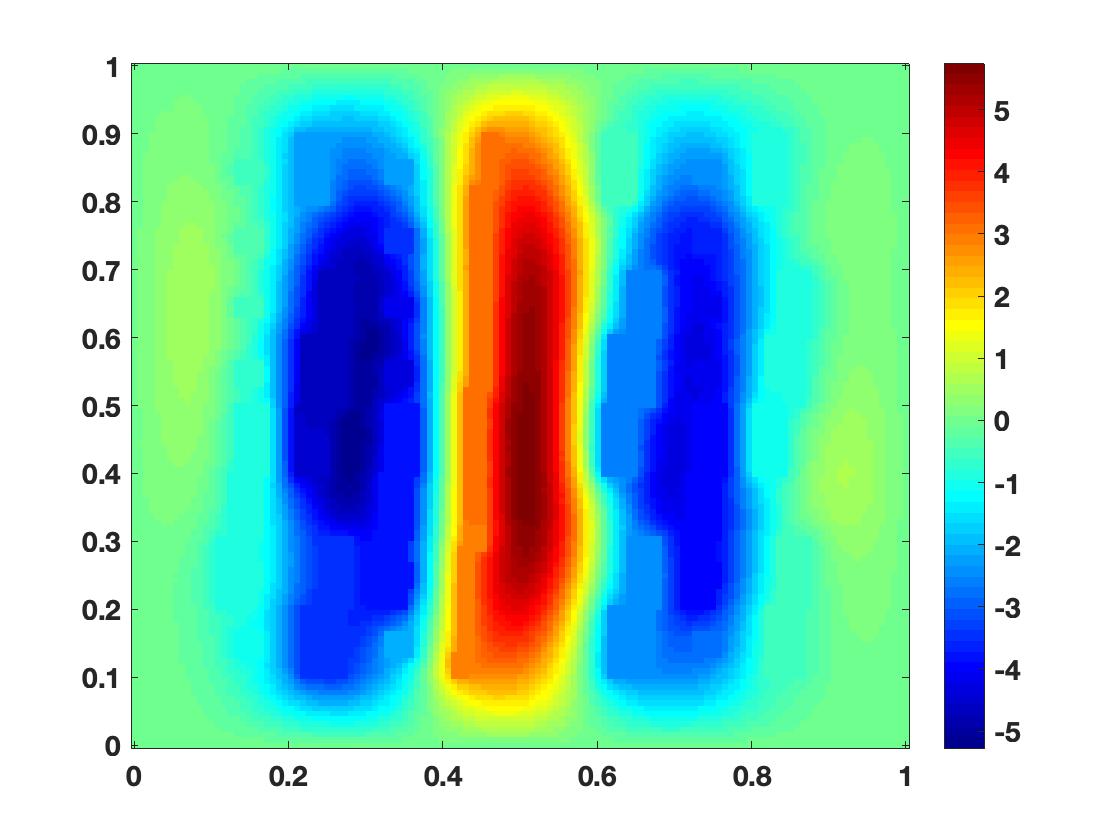}
		\includegraphics[trim={3cm 1.8cm 2.8cm 2cm},clip,width=0.24 \textwidth]{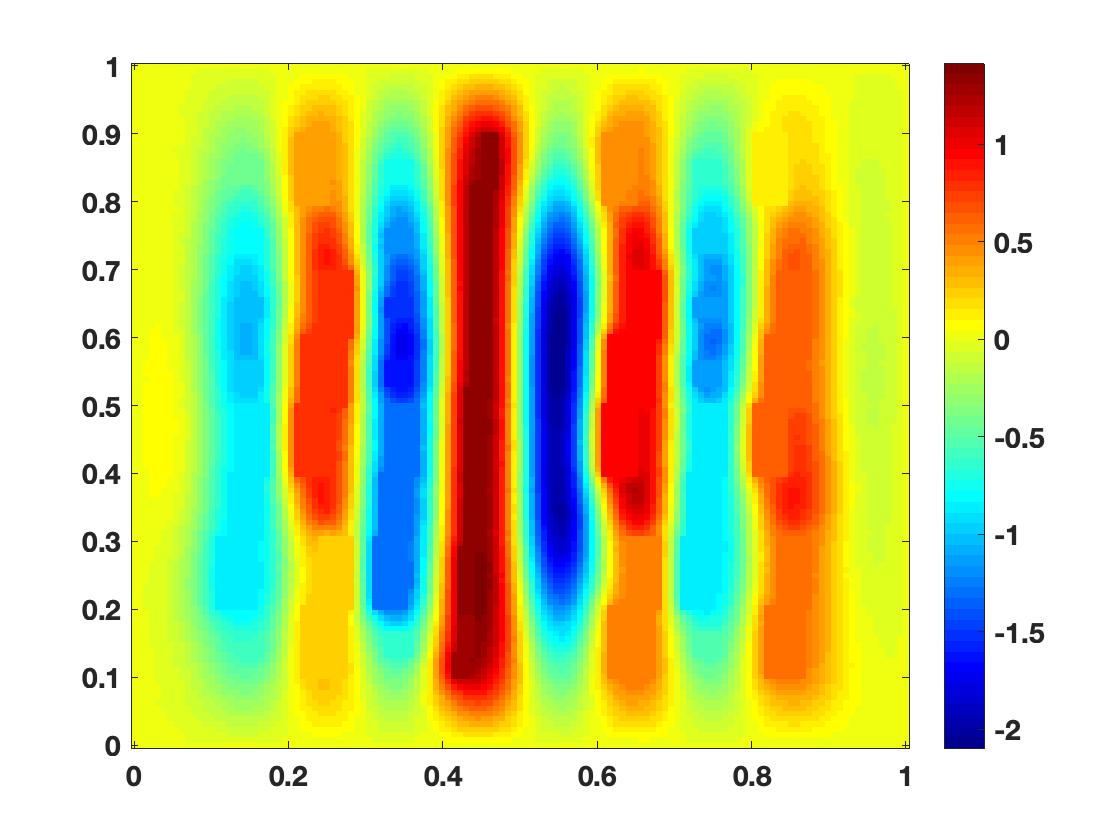}
\caption{Multiscale solution from Algorithm \ref{algorithm:wavelet} with $\delta t=10^{-3}$ and $\ell=2$, backward Euler scheme: $u_{\text{ms},\ell}^{\text{EW,100}}$, $u_{\text{ms},\ell}^{\text{EW,300}}$, $u_{\text{ms},\ell}^{\text{EW,500}}$ and $u_{\text{ms},\ell}^{\text{EW,1000}}$.}\label{fig:EWsol_NonzeroSource_level2_BE_coarser}
	\end{figure}	
\begin{figure}[H]
		\centering
		\includegraphics[trim={3cm 1.8cm 2.8cm 2cm},clip,width=0.24 \textwidth]{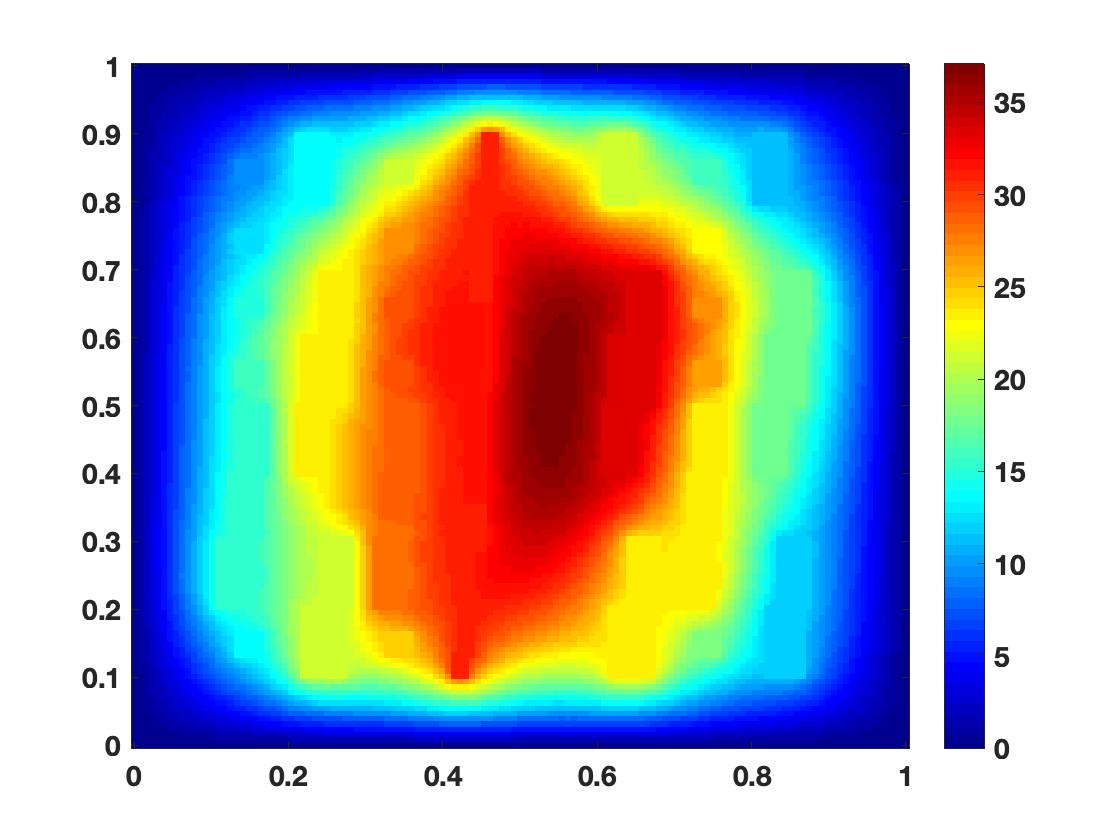}
		\includegraphics[trim={3cm 1.8cm 2.8cm 2cm},clip,width=0.24 \textwidth]{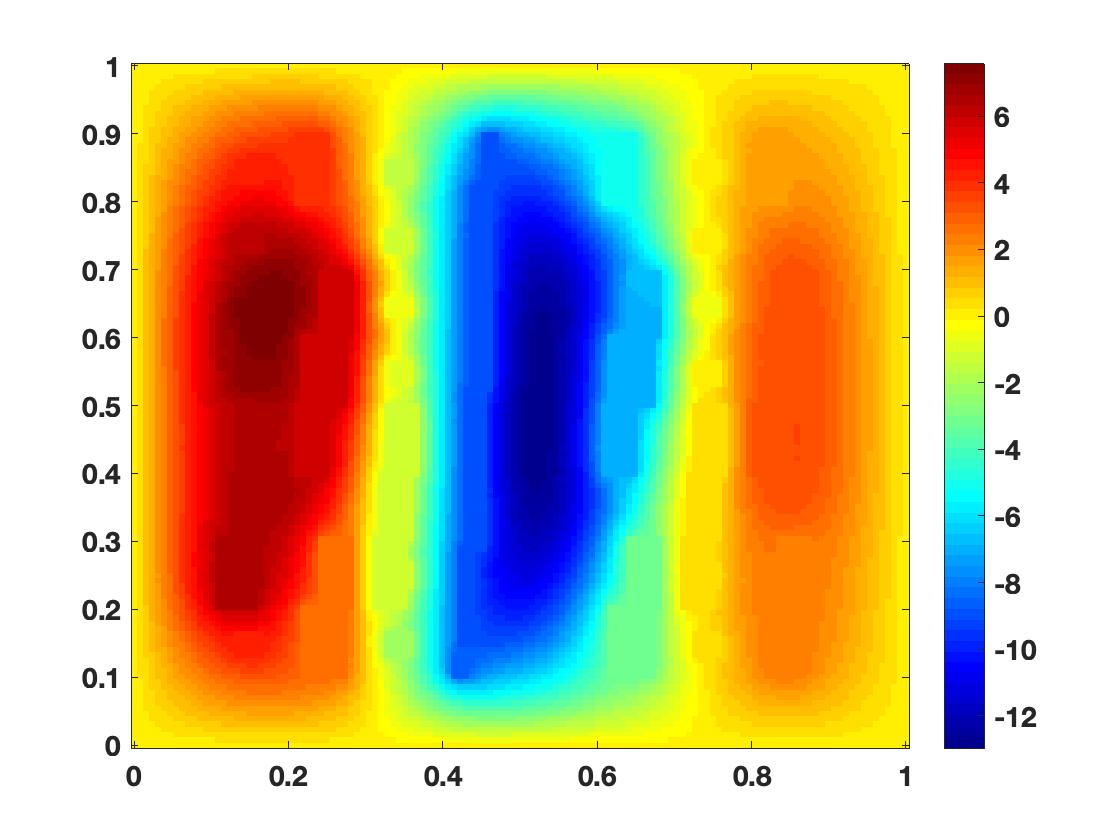}
		\includegraphics[trim={3cm 1.8cm 2.8cm 2cm},clip,width=0.24 \textwidth]{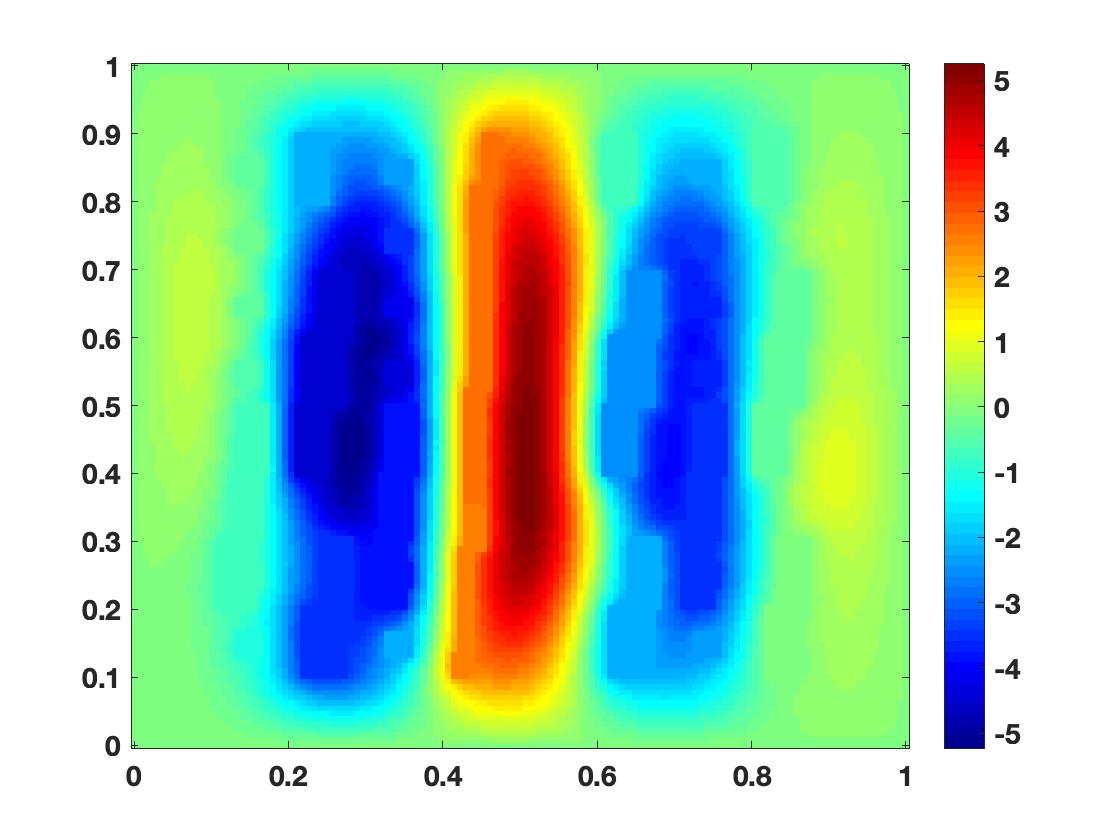}
		\includegraphics[trim={3cm 1.8cm 2.8cm 2cm},clip,width=0.24 \textwidth]{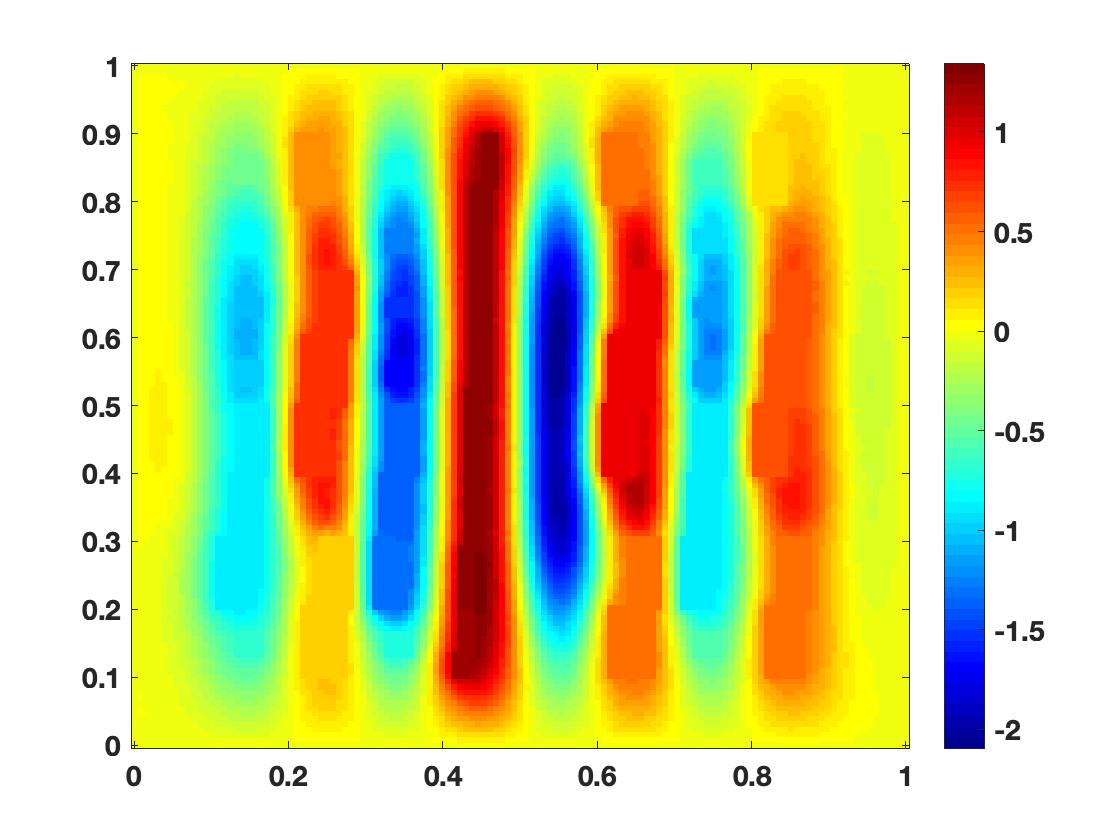}\\
		\includegraphics[trim={3cm 1.8cm 2.8cm 2cm},clip,width=0.24 \textwidth]{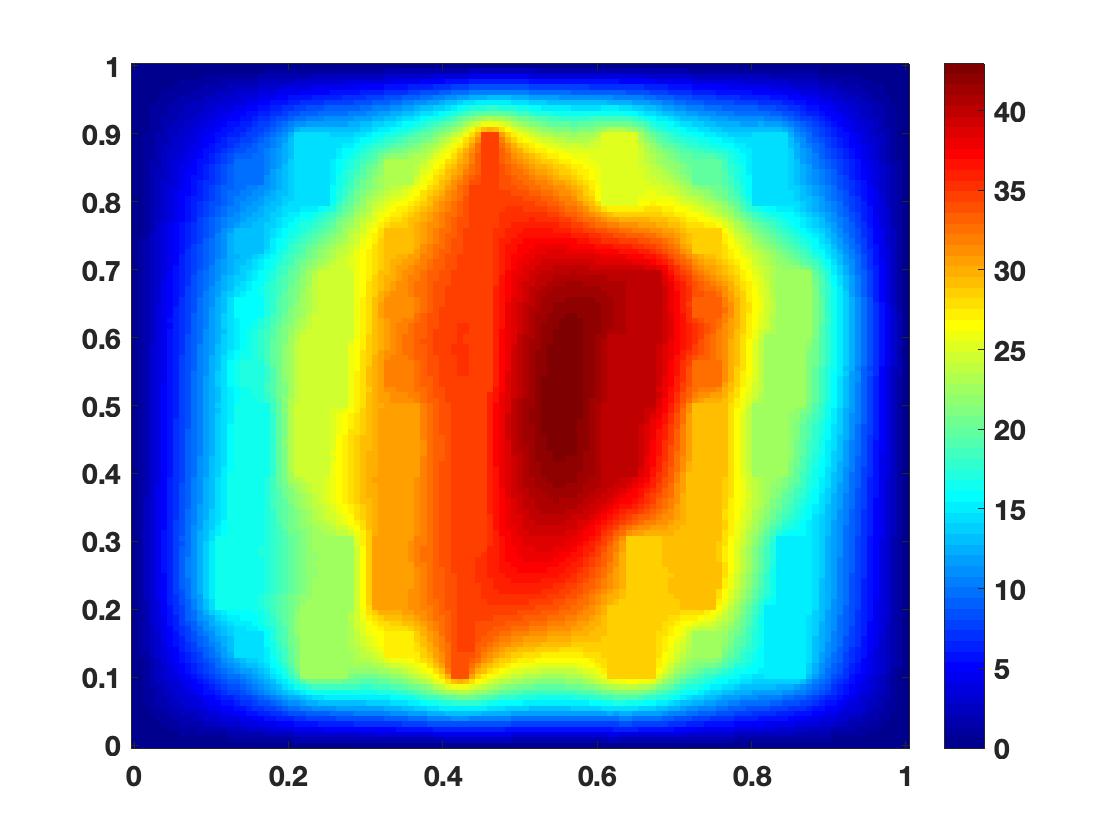}
		\includegraphics[trim={3cm 1.8cm 2.8cm 2cm},clip,width=0.24 \textwidth]{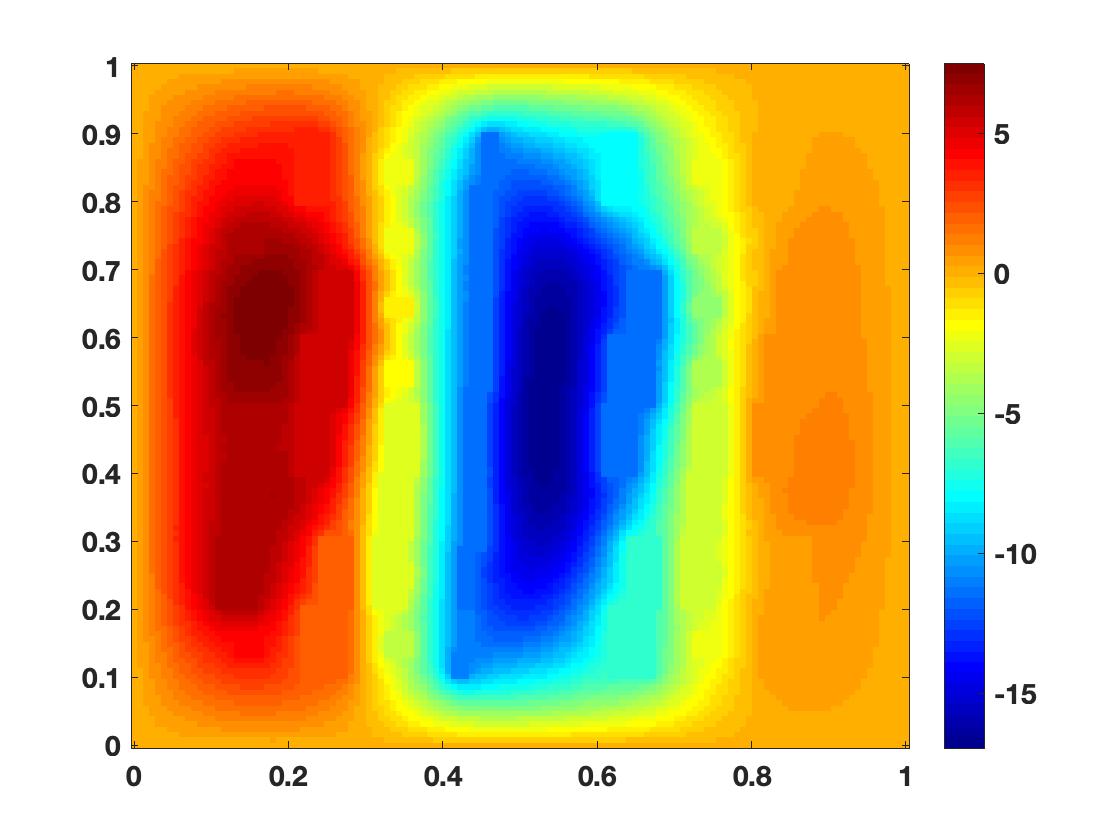}
		\includegraphics[trim={3cm 1.8cm 2.8cm 2cm},clip,width=0.24 \textwidth]{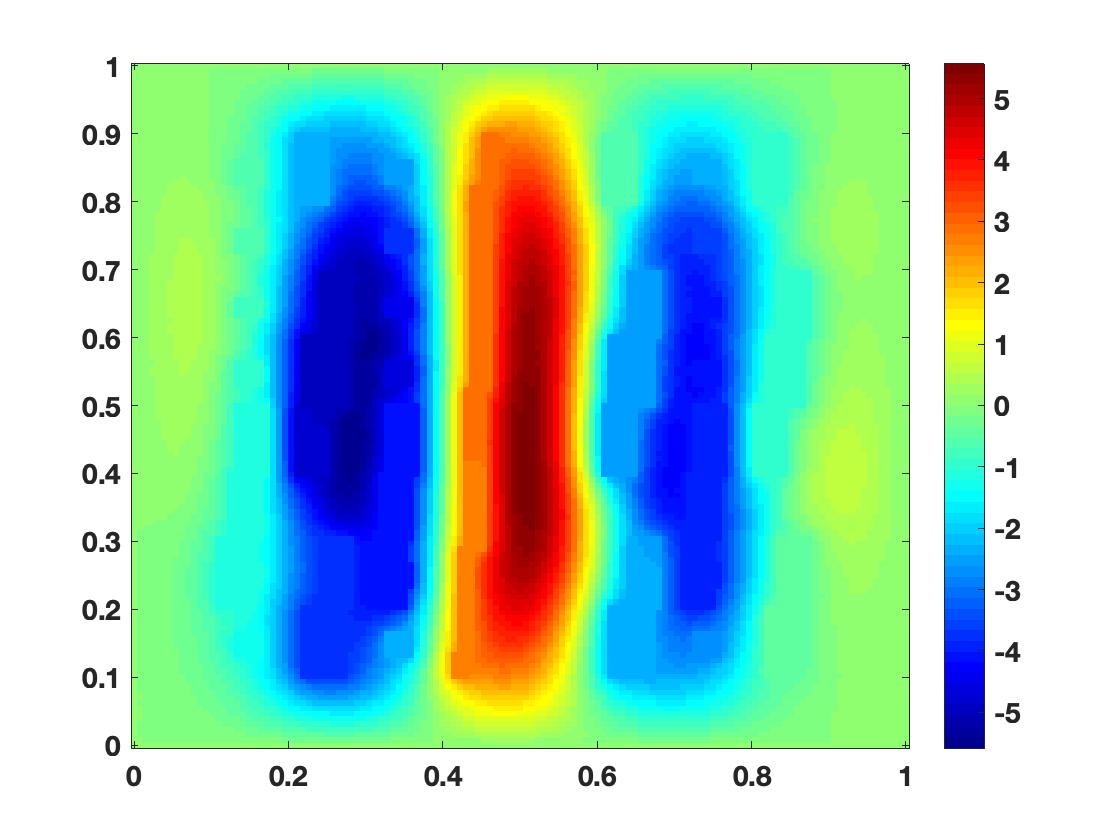}
		\includegraphics[trim={3cm 1.8cm 2.8cm 2cm},clip,width=0.24 \textwidth]{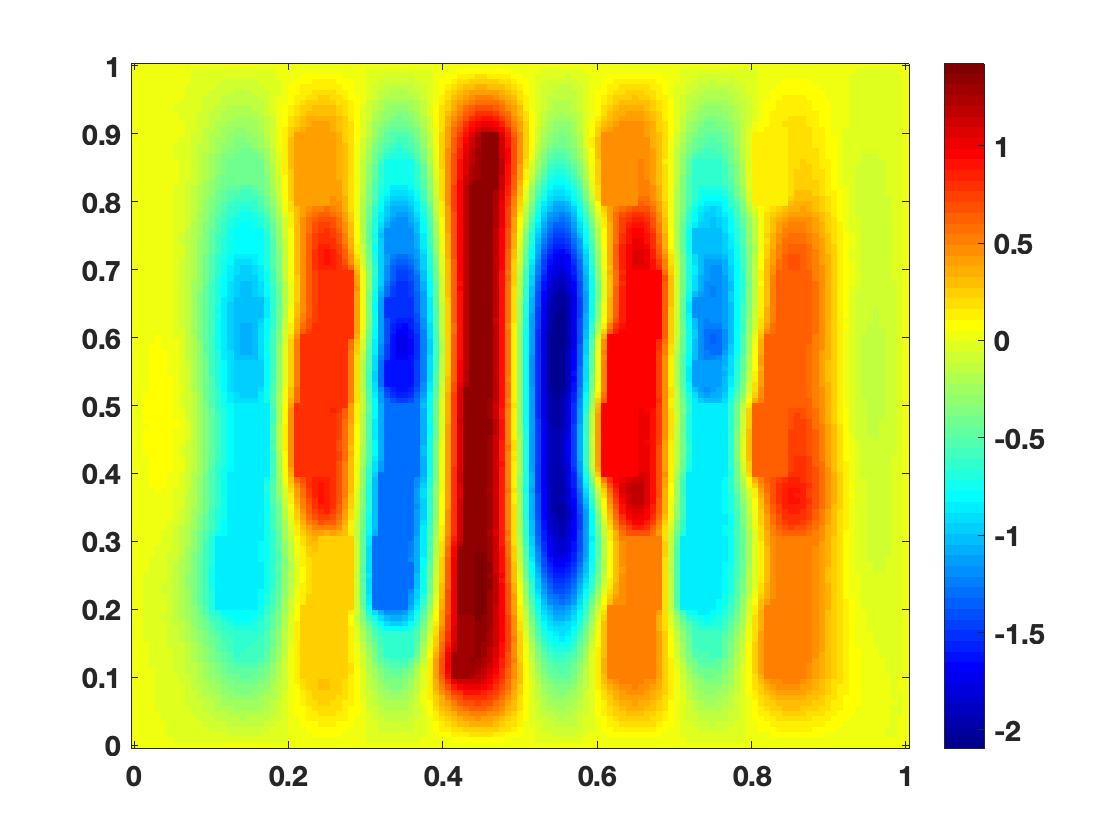}\\
		\includegraphics[trim={3cm 1.8cm 2.8cm 2cm},clip,width=0.24 \textwidth]{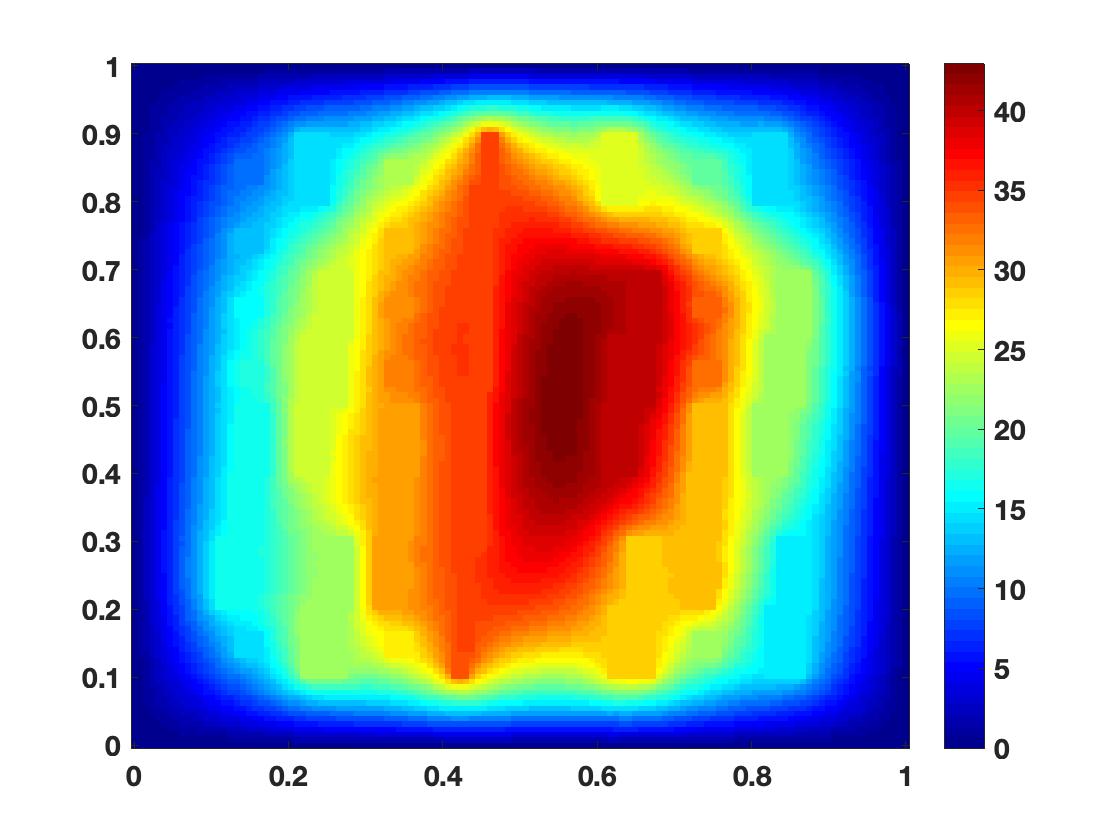}
		\includegraphics[trim={3cm 1.8cm 2.8cm 2cm},clip,width=0.24 \textwidth]{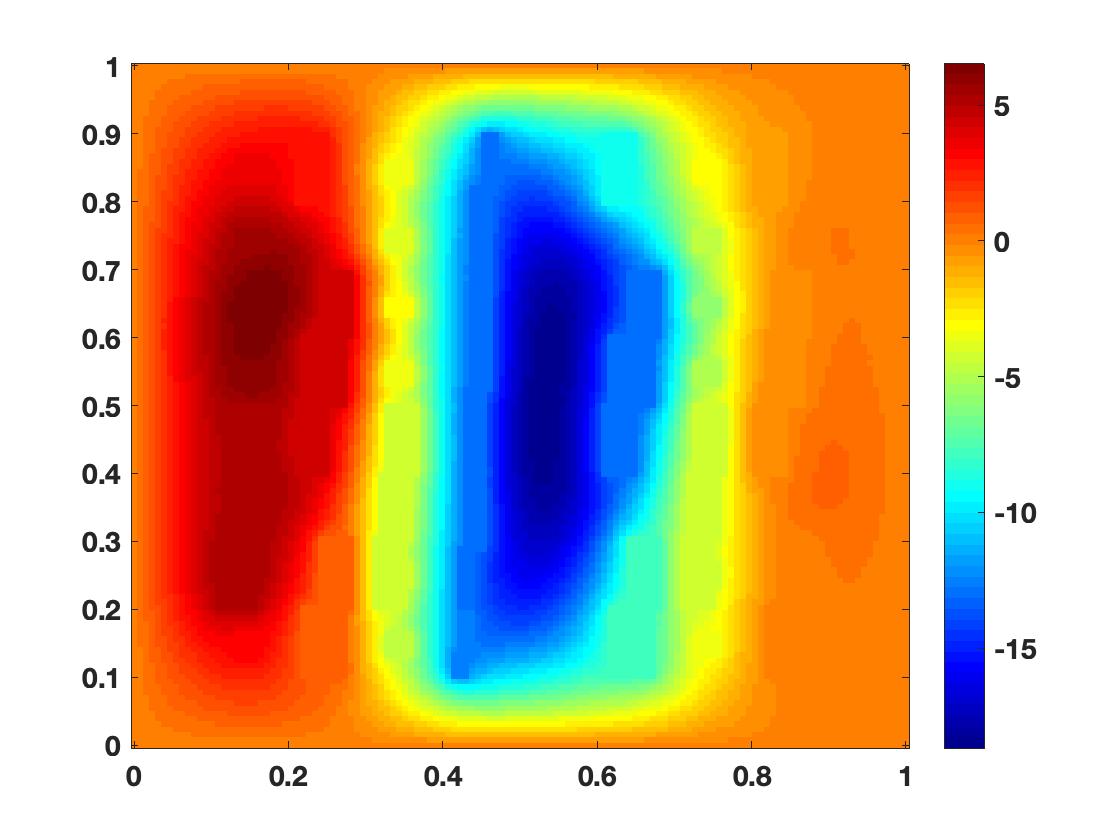}
		\includegraphics[trim={3cm 1.8cm 2.8cm 2cm},clip,width=0.24 \textwidth]{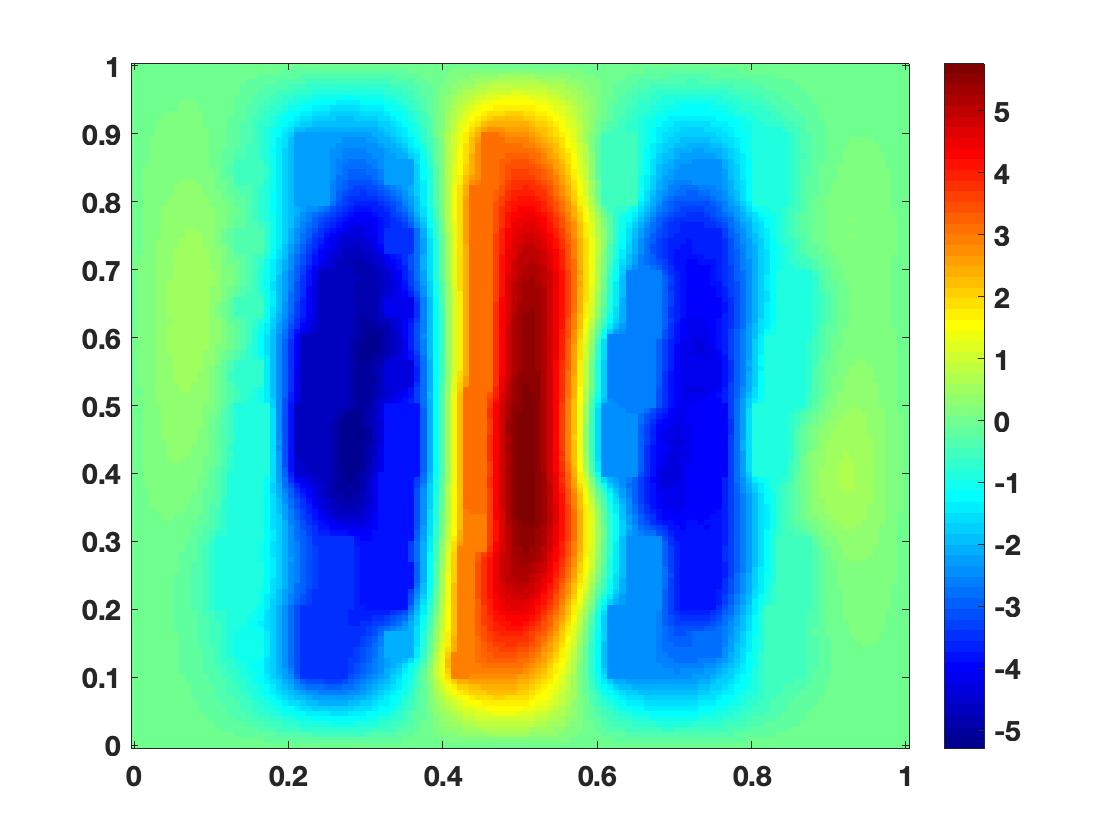}
		\includegraphics[trim={3cm 1.8cm 2.8cm 2cm},clip,width=0.24 \textwidth]{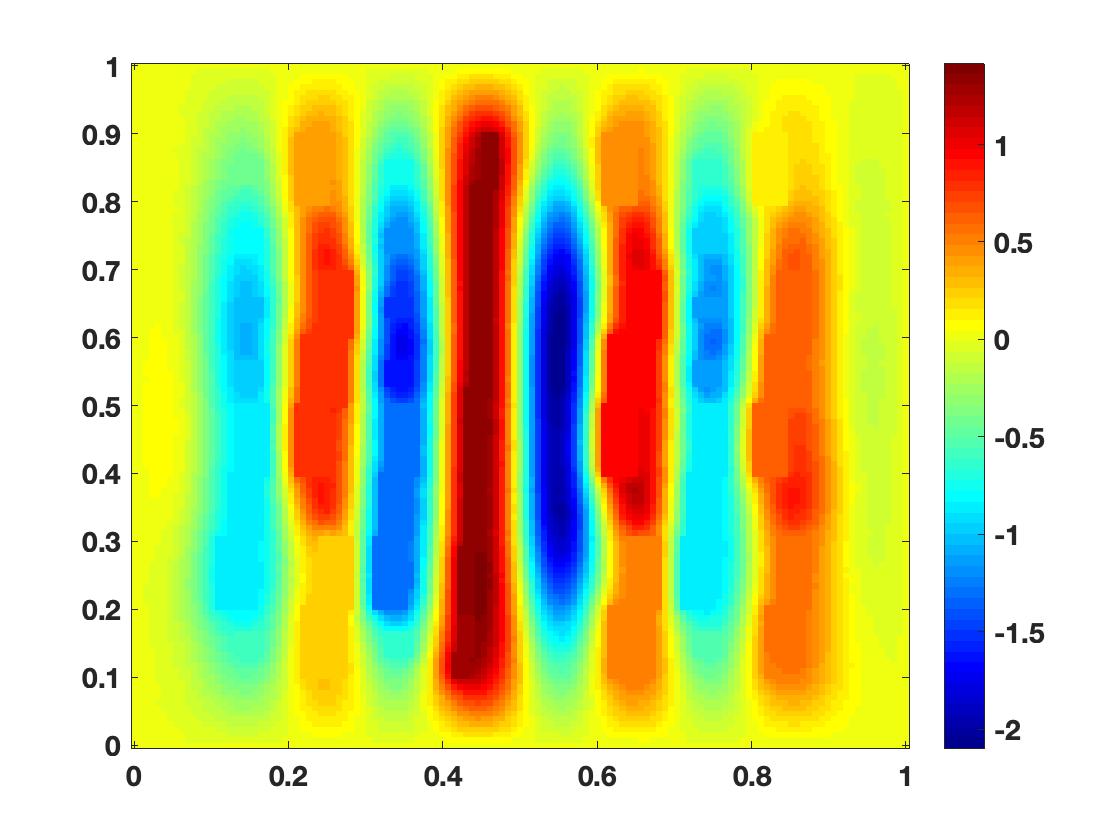}
		\caption{Numerical solutions $ U_k^n$ for $n=1, 3, 5, 10$ from Algorithm \ref{algorithm:wavelet+parareal} with $\Delta T=0.1$ and $\delta t=10^{-3}$, backward Euler scheme: iteration number $k=0$ (top), $k=1$ (middle) and $k=2$ (bottom).}
		\label{fig:PararealSol_NonzeroSource_level2_BE_coarser}
	\end{figure}	
We present the convergence history of Algorithm \ref{algorithm:wavelet+parareal} in relative $L^2(D)$ error and relative $H^1_{\kappa}(D)$ error in Tables \ref{error:L2_NonzeroSource_l2_BE_coarser} and \ref{error:H1_NonzeroSource_l2_BE_coarser}.
In each table, the second column displays the relative error between numerical solutions $u_{\text{ms},\ell}^{\text{EW,n}} $ from Algorithm \ref{algorithm:wavelet} and reference solutions $u^n_h$. The relative errors between numerical solutions $U_k^n$ from Algorithm \ref{algorithm:wavelet+parareal} and reference solutions $u^m_h$ are displayed from the third column towards the last column.

\begin{table}[H]
	\begin{center}
	\begin{tabular}{|c|c|c|c|c|c|c|}
	\hline
	$T^n$ & $\text{Rel}^{\text{EW}}_{L^2} (T^n)$  &$\text{Rel}^0_{L^2}(T^n)$& $\text{Rel}^1_{L^2}(T^n)$&$\text{Rel}^2_{L^2}(T^n)$& $\text{Rel}^3_{L^2}(T^n)$& $\text{Rel}^4_{L^2}(T^n)$
	\\ \hline
	0.1 &  0.5671  &14.3828 & 0.5671 & 0.5671 & 0.5671 & 0.5671 \\
 0.2 & 0.8234   & 30.6711 & 4.0891 & 0.8234 & 0.8234 & 0.8234 \\
 0.3 &  0.8258   &  42.2133 & 13.6117 & 0.9514 & 0.8258 & 0.8258 \\
 0.4 &   0.5897  & 22.5947 & 8.2614 & 4.9882 & 0.5569 & 0.5897 \\
 0.5 &   0.5323  & 11.4592 & 6.4804 & 1.1830 & 2.2190 & 0.5022 \\
 0.6 &  0.7072   & 8.6259 & 1.8506 & 2.0906 & 1.1444 & 0.7087 \\
 0.7 &  0.7229   & 8.9372 & 1.0193 & 1.2827 & 0.7229 & 0.7531 \\
 0.8 &  0.9680   & 11.0146 & 1.9673 & 0.9901 & 1.0137 & 0.9766 \\
 0.9 &   1.0681  & 6.2513 & 1.9613 & 1.2086 & 1.0517 & 1.0779 \\
 1.0 &  0.9145   & 5.1982 & 1.0812 & 0.9423 & 0.9399 & 0.9096\\
	\hline
	\end{tabular}
	\end{center}
\vspace{-.4cm}
	\caption{Convergence history of Algorithm \ref{algorithm:wavelet+parareal} in relative $L^2(D)$ error for Experiment 1: backward Euler scheme with ${\Delta T}=0.1$ and ${\delta t}=10^{-3}$.}
	\label{error:L2_NonzeroSource_l2_BE_coarser}
	\end{table}

	\begin{table}[H]
	\begin{center}
	\begin{tabular}{|c|c|c|c|c|c|c|}
	\hline
	$T^n$ & $\text{Rel}^{\text{EW}}_{H_{\kappa}^1} (T^n)$&$\text{Rel}^0_{H_{\kappa}^1}(T^n)$& $\text{Rel}^1_{H_{\kappa}^1}(T^n)$& $\text{Rel}^2_{H_{\kappa}^1}(T^n)$& $\text{Rel}^3_{H_{\kappa}^1}(T^n)$& $\text{Rel}^4_{H_{\kappa}^1}(T^n)$
	\\ \hline
	0.1 &  6.9437     & 16.5638 & 6.9437 & 6.9437 & 6.9437 & 6.9437 \\
 0.2 &   5.6489     &25.0878 & 6.6353 & 5.6489 & 5.6489 & 5.6489 \\
 0.3 &     4.9158   &27.6524 & 9.5624 & 4.9600 & 4.9158 & 4.9158 \\
 0.4 &     4.7240   &16.6013 & 6.2652 & 5.2435 & 4.7283 & 4.7240 \\
 0.5 &     4.8984  & 9.3734 & 5.6879 & 4.9166 & 4.9356 & 4.8997 \\
 0.6 &     5.3109   &7.9758 & 5.3640 & 5.3383 & 5.3148 & 5.3130 \\
 0.7 &     5.3064   &6.9107 & 5.3173 & 5.3128 & 5.3066 & 5.3067 \\
 0.8 &   6.2666     &7.6015 & 6.2857 & 6.2666 & 6.2671 & 6.2667 \\
 0.9 &   6.4270     &7.0167 & 6.4435 & 6.4279 & 6.4270 & 6.4270 \\
 1.0 &  4.9341      &5.2526 & 4.9371 & 4.9343 & 4.9341 & 4.9341 \\
	\hline
	\end{tabular}
	\end{center}
	\vspace{-.4cm}
	\caption{Convergence history of Algorithm \ref{algorithm:wavelet+parareal}  in relative $H_{\kappa}^1(D)  $ error  for Experiment 1: backward Euler scheme with ${\Delta T}=0.1$ and ${\delta t}=10^{-3}$.}
	\label{error:H1_NonzeroSource_l2_BE_coarser}
	\end{table}
One can observe from Tables \ref{error:L2_NonzeroSource_l2_BE_coarser} and \ref{error:H1_NonzeroSource_l2_BE_coarser} that 4 iterations is sufficient for Algorithm \ref{algorithm:wavelet+parareal} to attain the same accuracy as Algortithm \ref{algorithm:wavelet} for all discrete time levels under the $L^2(D)$-norm, while 2 iterations under $H_{\kappa}^1(D)$-norm. For each iteration, it involves solving the original system for $10$ times using coarse solver and using fine solvers for $100$ times on each time interval in parallel. However, to obtain multiscale solutions from Algorithm \ref{algorithm:wavelet}, it involves solving original system for about $1000$ times. So, with the aid of large number of  processors, it could save a lot of time to obtain numerical solutions from Algorithm \ref{algorithm:wavelet+parareal}.

	
\subsubsection*{Experiment 2: Crank-Nicolson with $\frac{\Delta T}{\delta t}=100$}
Since the backward Euler scheme is only first order accurate, a higher order accurate scheme can improve the performance of
Algorithm \ref{algorithm:wavelet} and Algorithm \ref{algorithm:wavelet+parareal}. This can be seen from the proof to Theorem \ref{prop:wavelet-based}. In this section, we will present the numerical tests with Crank-Nicolson scheme for both algorithms.

 A direct application of Crank-Nicolson scheme as a time discretization fails to maintain second order accuracy due to the possible blow up of the eigenvalues of the elliptic operator $-\nabla\cdot(\kappa\nabla \cdot)$ when $\eta_{\text{max}}\to\infty$. To improve its performance and maintain second order convergence rate, we use 3 steps of backward Euler scheme before Crank-Nicolson scheme kicks in \cite{smoothing_CN,thomee1984galerkin}.

 The multiscale solutions from Algorithm \ref{algorithm:wavelet} with Crank-Nicolson scheme are presented in Figure \ref{fig:EWsol_NonzeroSource_level2_CN_coarser}. We present  the numerical solutions $U^n_k$ for $n=1,3,5,10$ from Algorithm \ref{algorithm:wavelet+parareal} with iteration $k=0,1,2$ in Figure \ref{fig:PararealSol_NonzeroSource_level2_CN_coarser}. One can observe the same convergence behavior as in Experiment 1.
\begin{figure}[H]
		\centering
		\includegraphics[trim={3cm 1.8cm 2.8cm 2cm},clip,width=0.24 \textwidth]{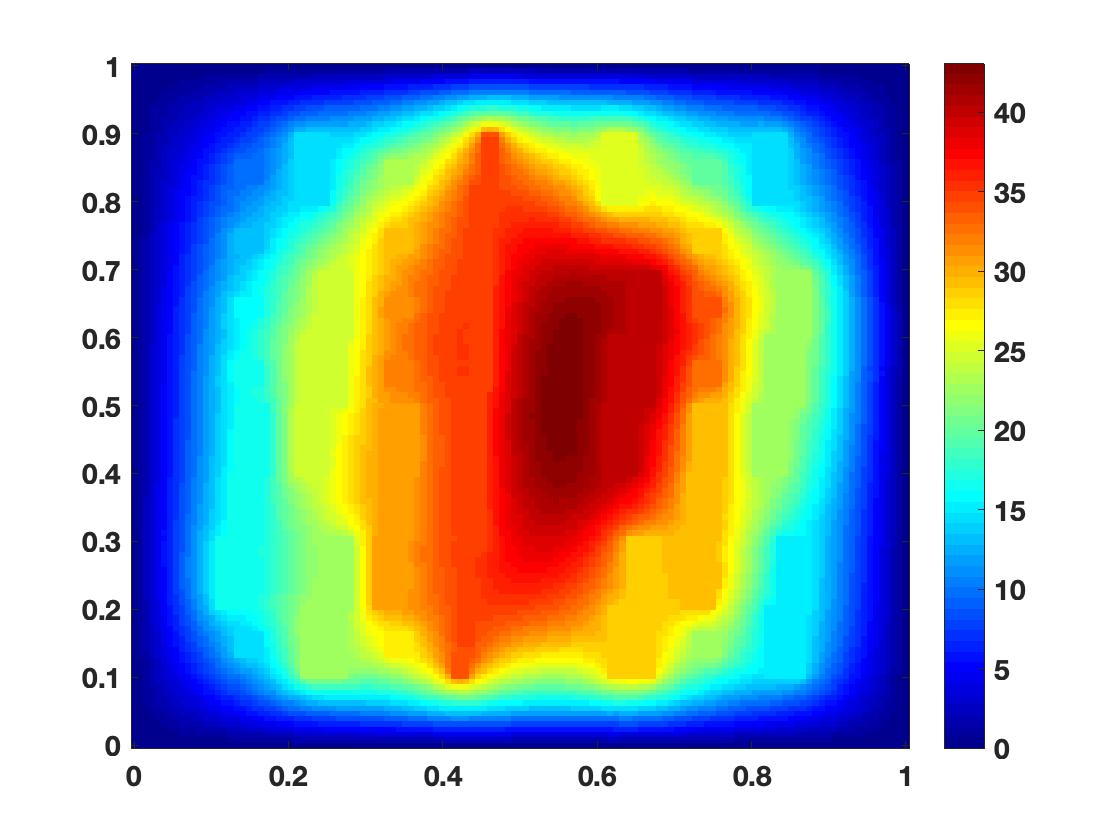}
		\includegraphics[trim={3cm 1.8cm 2.8cm 2cm},clip,width=0.24 \textwidth]{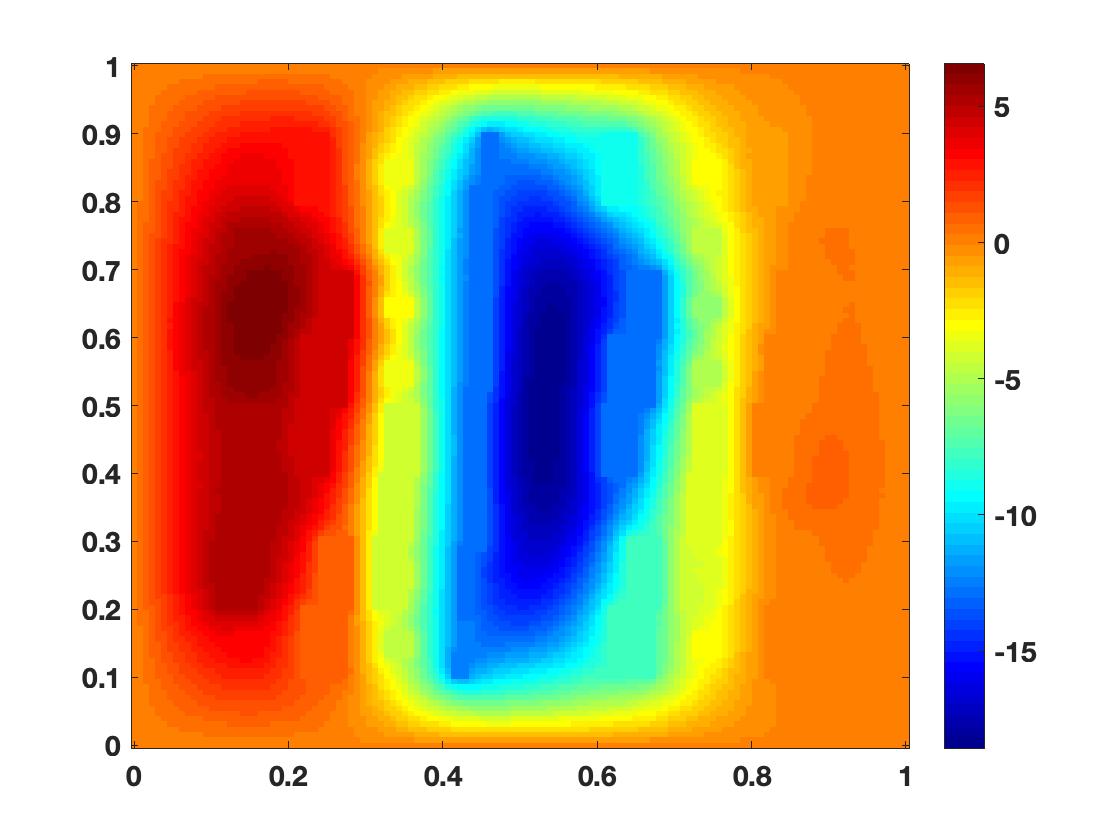}
		\includegraphics[trim={3cm 1.8cm 2.8cm 2cm},clip,width=0.24 \textwidth]{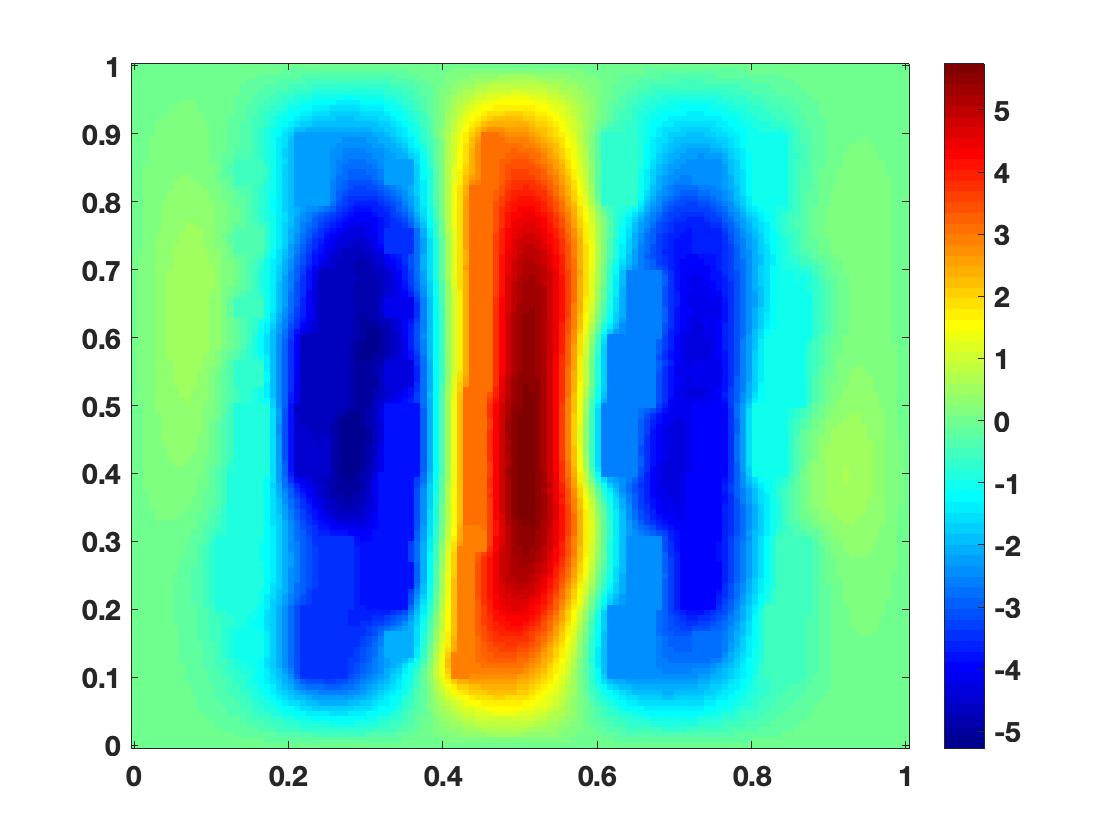}
		\includegraphics[trim={3cm 1.8cm 2.8cm 2cm},clip,width=0.24 \textwidth]{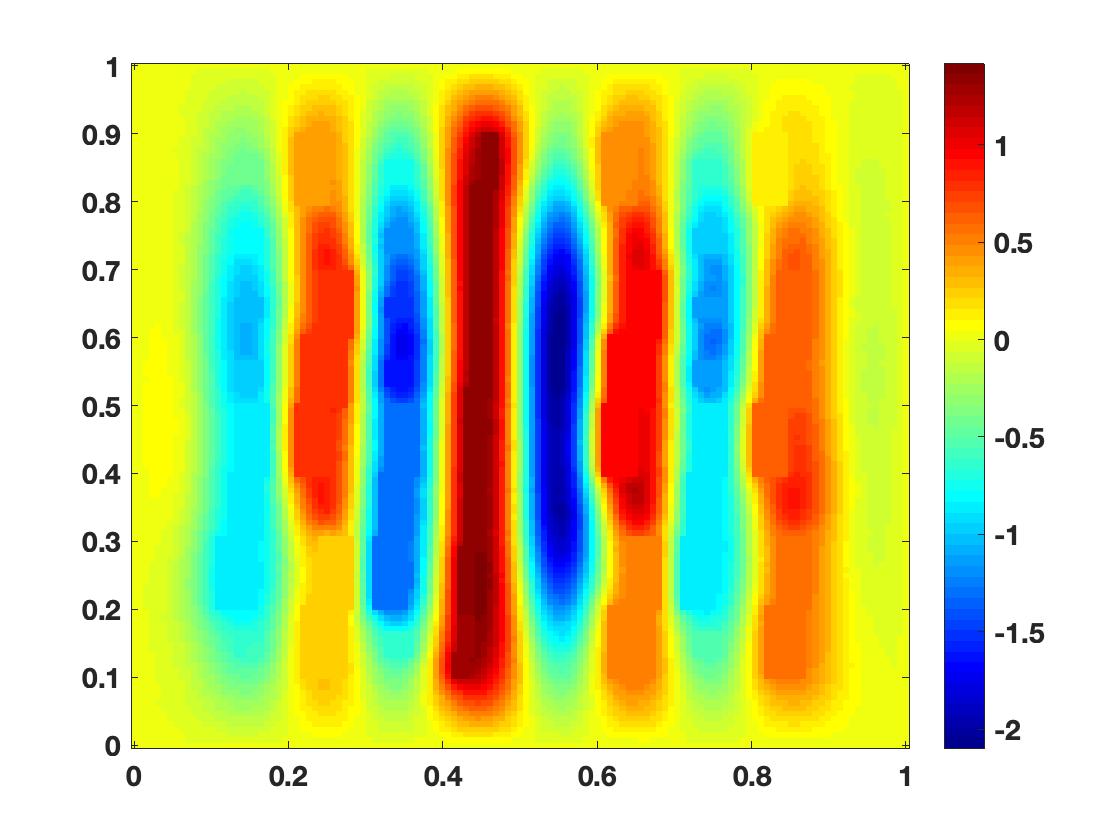}		
		\caption{Multiscale solution from Algorithm \ref{algorithm:wavelet} with $\delta t=10^{-3}$ and $\ell=2$, Crank-Nicolson scheme: $u_{\text{ms},\ell}^{\text{EW,100}}$, $u_{\text{ms},\ell}^{\text{EW,300}}$,  $u_{\text{ms},\ell}^{\text{EW,500}}$ and $u_{\text{ms},\ell}^{\text{EW,1000}}$.}
		\label{fig:EWsol_NonzeroSource_level2_CN_coarser}
	\end{figure}	


\begin{figure}[H]
		\centering
		\includegraphics[trim={3cm 1.8cm 2.8cm 2cm},clip,width=0.24 \textwidth]{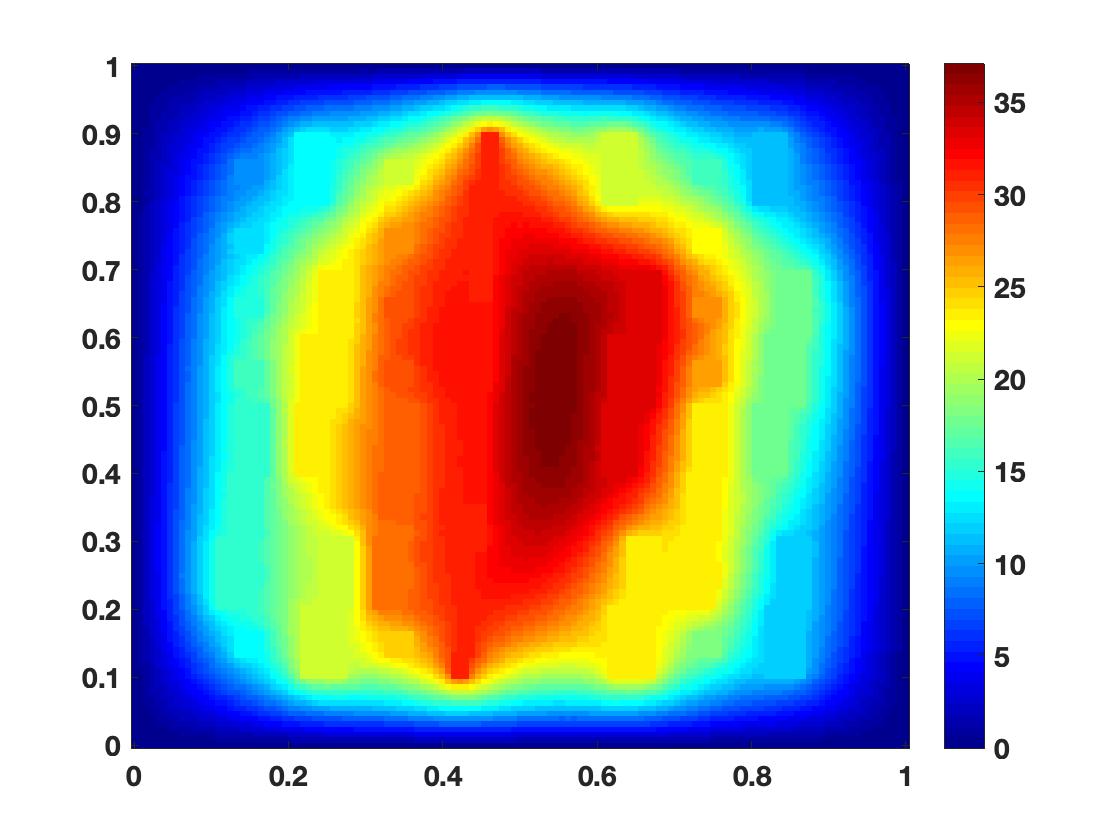}
		\includegraphics[trim={3cm 1.8cm 2.8cm 2cm},clip,width=0.24 \textwidth]{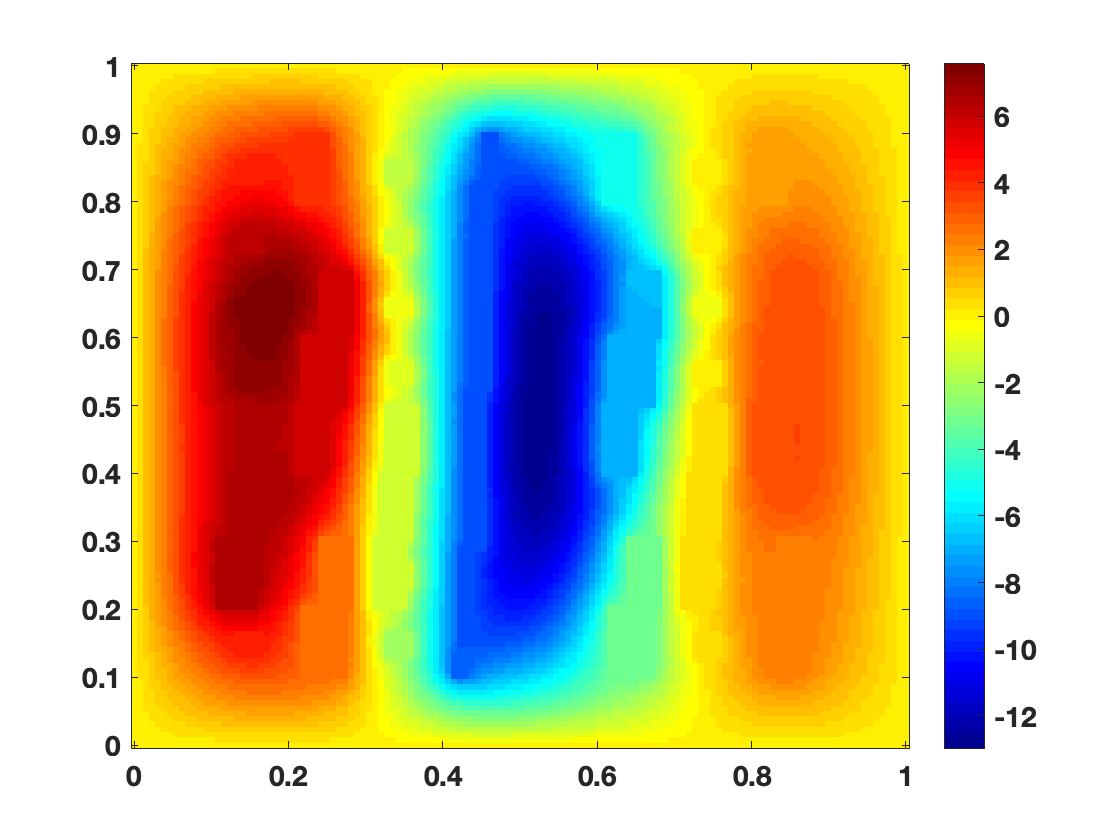}
		\includegraphics[trim={3cm 1.8cm 2.8cm 2cm},clip,width=0.24 \textwidth]{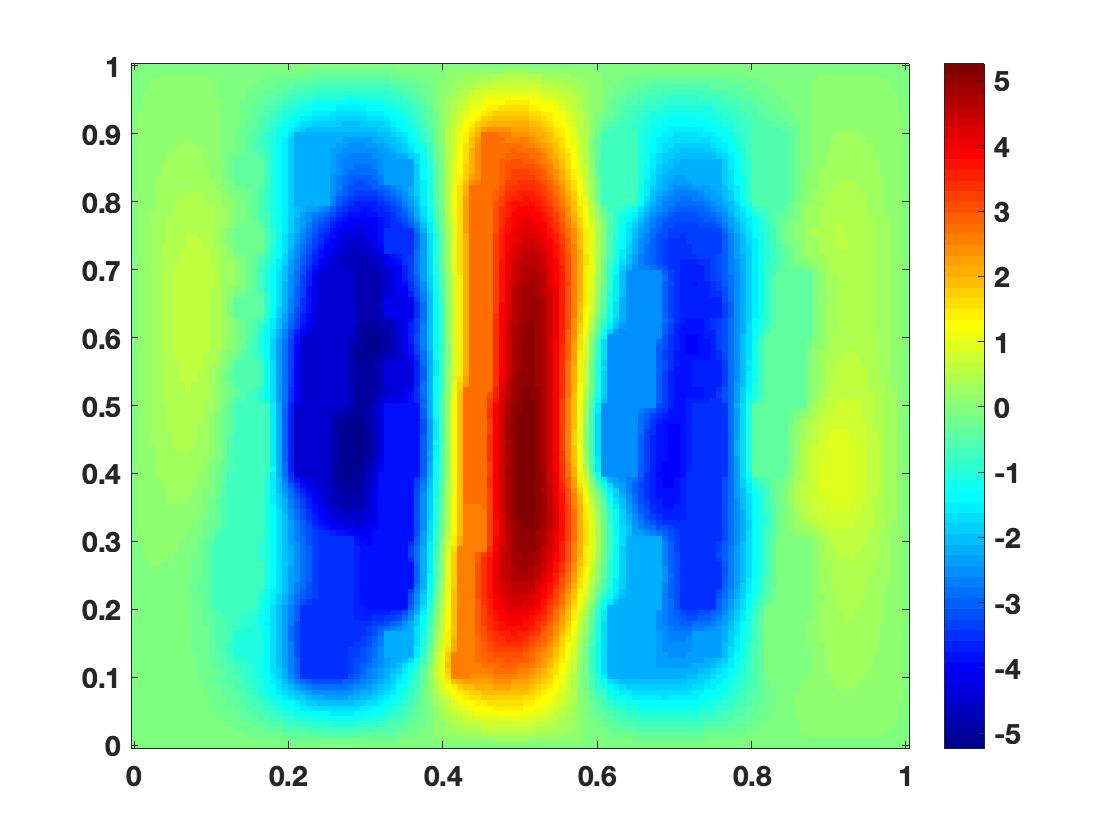}
		\includegraphics[trim={3cm 1.8cm 2.8cm 2cm},clip,width=0.24 \textwidth]{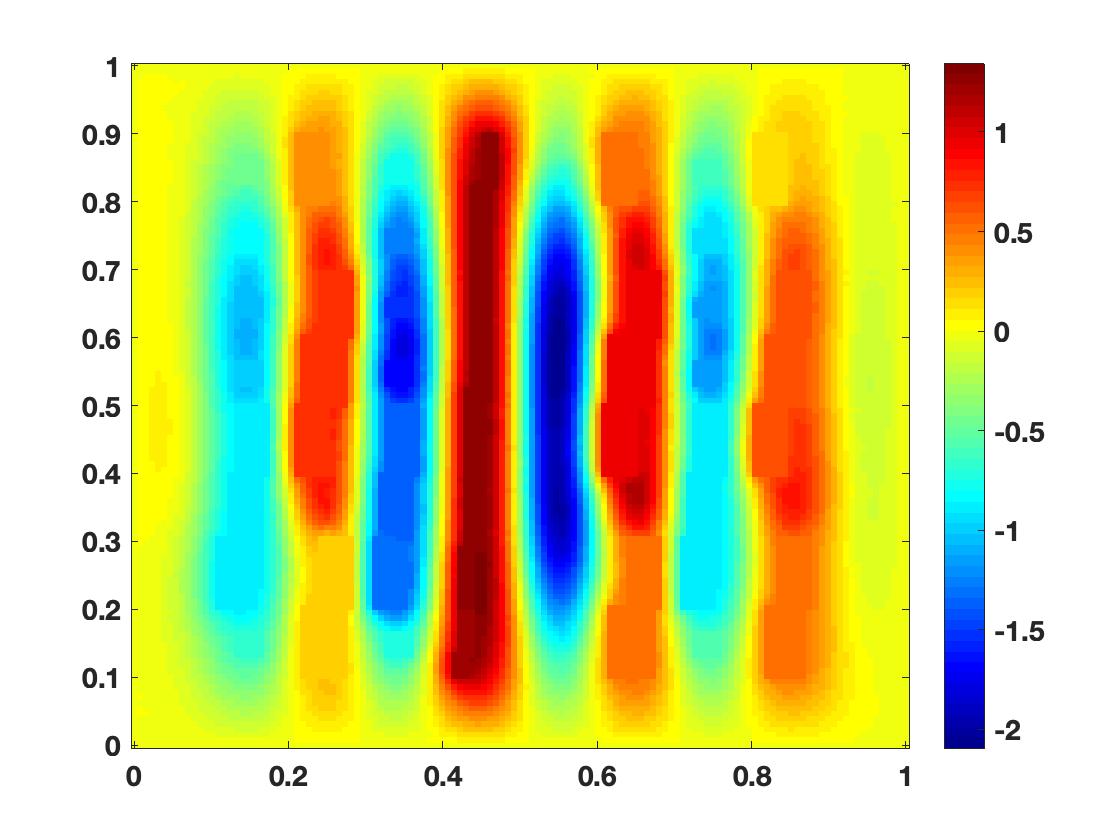}\\
		\includegraphics[trim={3cm 1.8cm 2.8cm 2cm},clip,width=0.24 \textwidth]{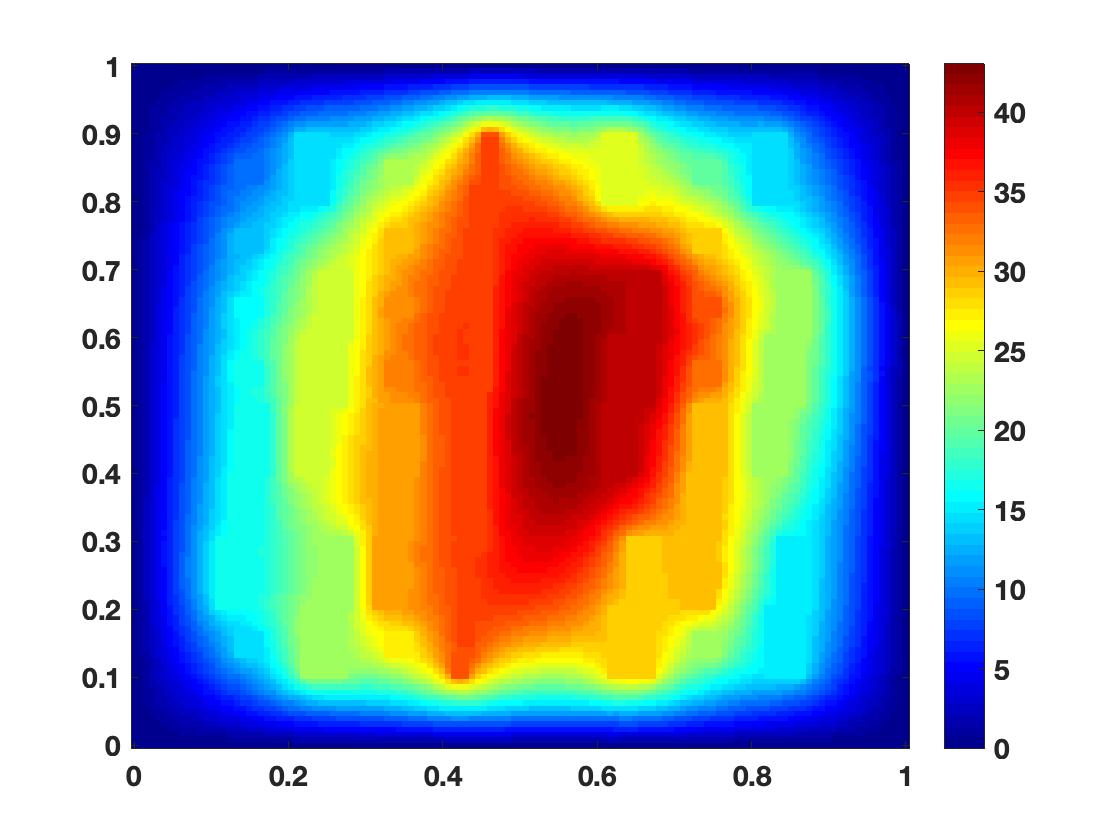}
		\includegraphics[trim={3cm 1.8cm 2.8cm 2cm},clip,width=0.24 \textwidth]{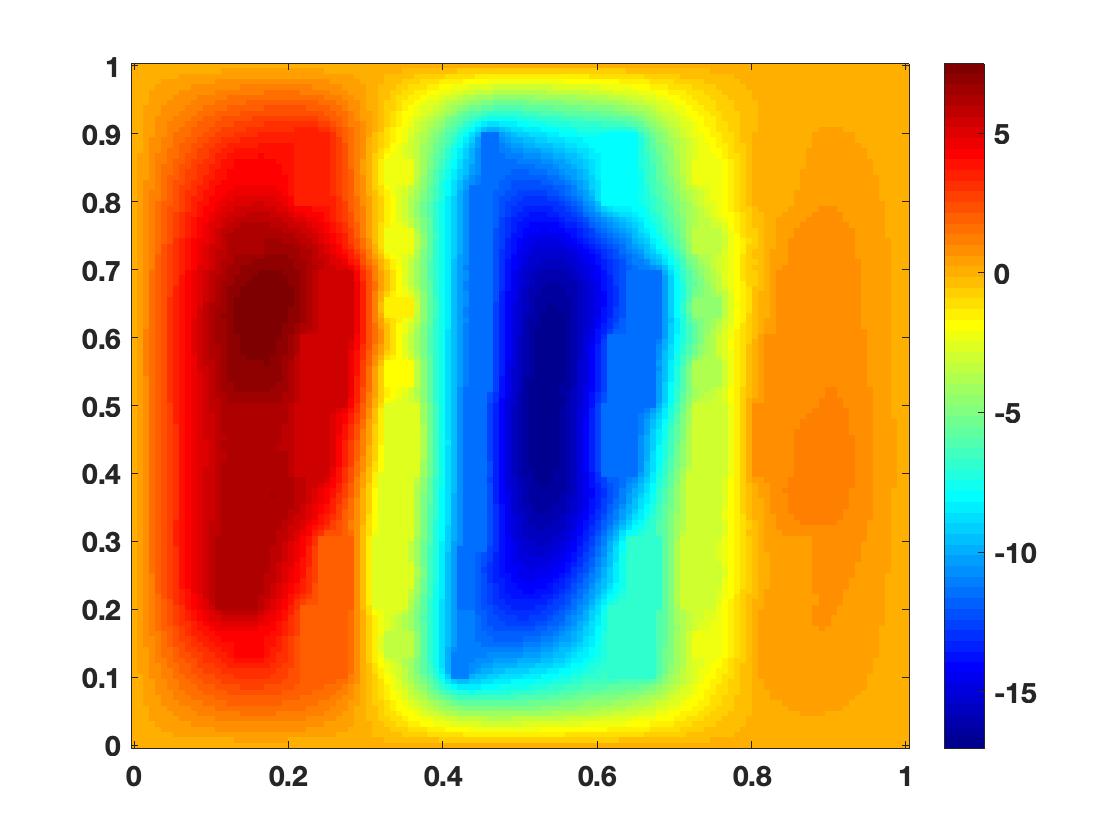}
		\includegraphics[trim={3cm 1.8cm 2.8cm 2cm},clip,width=0.24 \textwidth]{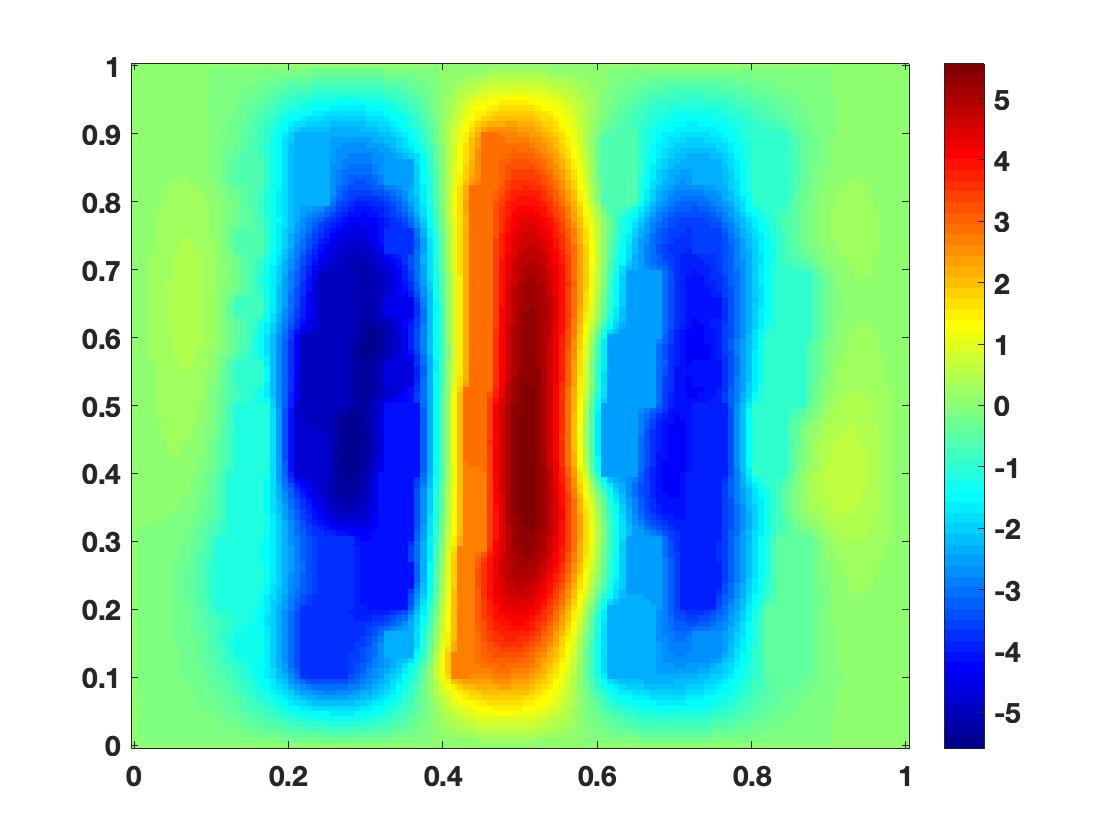}
		\includegraphics[trim={3cm 1.8cm 2.8cm 2cm},clip,width=0.24 \textwidth]{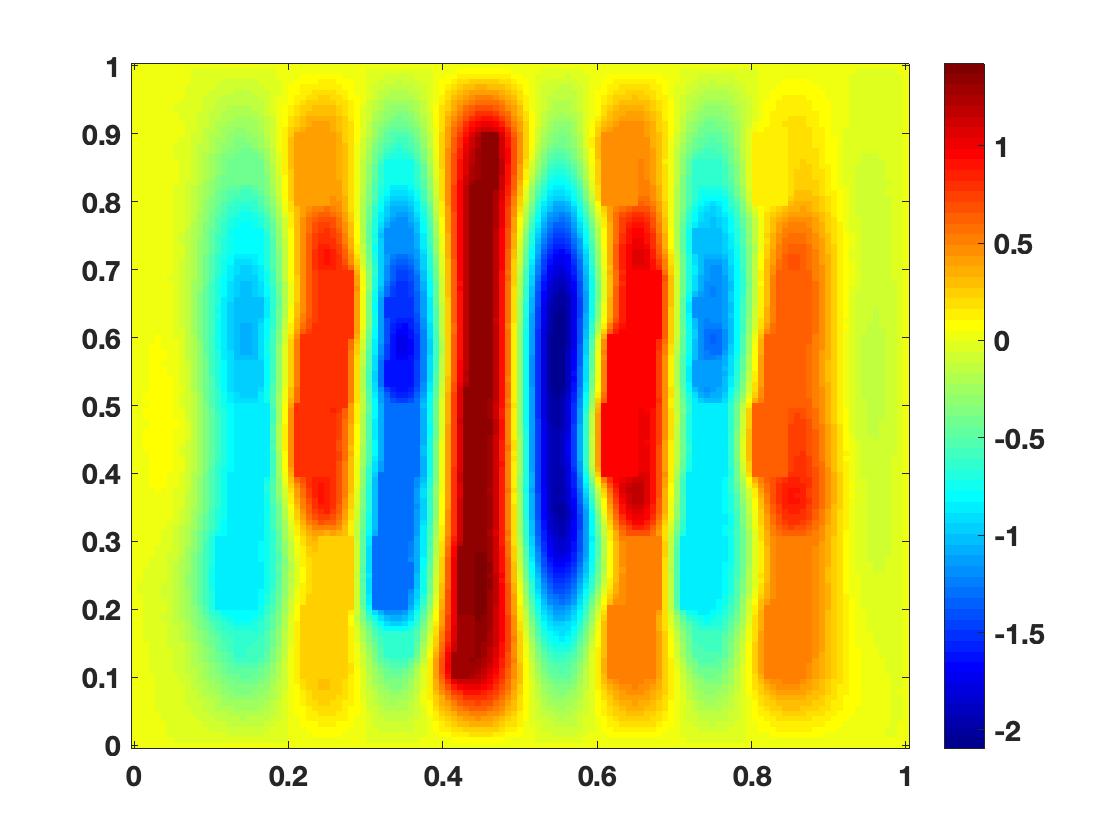}\\
		\includegraphics[trim={3cm 1.8cm 2.8cm 2cm},clip,width=0.24 \textwidth]{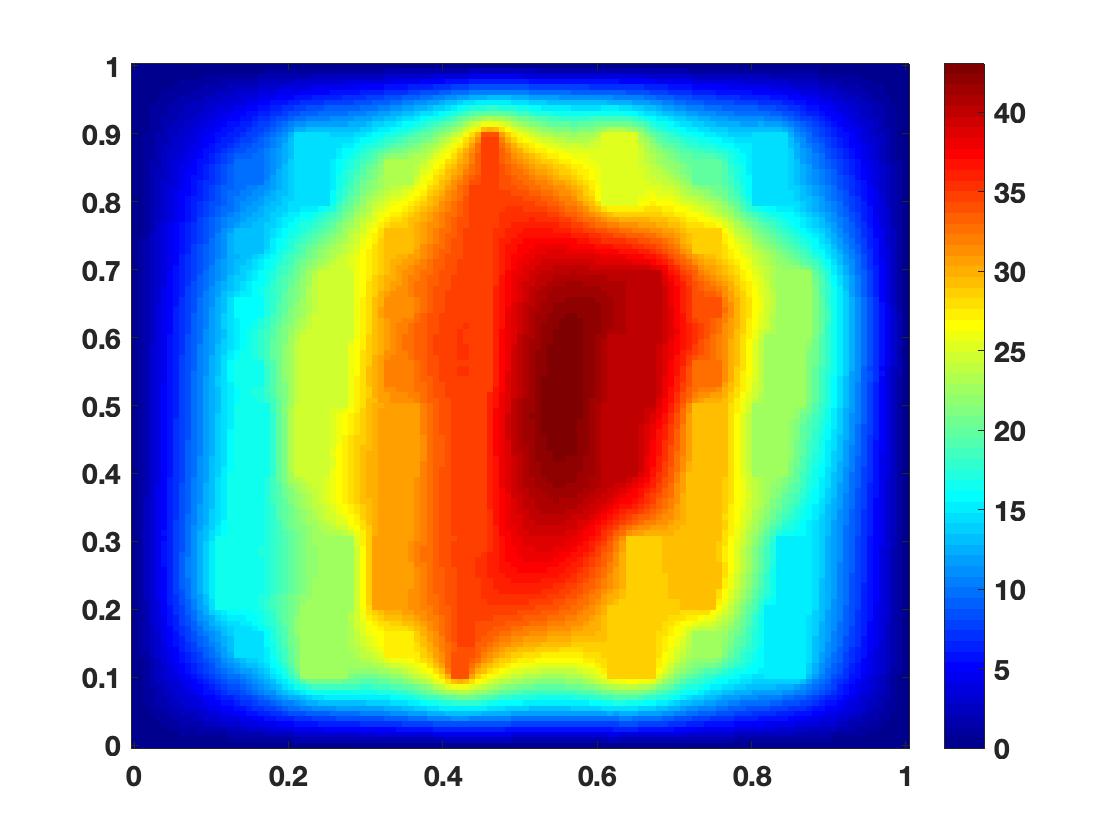}
		\includegraphics[trim={3cm 1.8cm 2.8cm 2cm},clip,width=0.24 \textwidth]{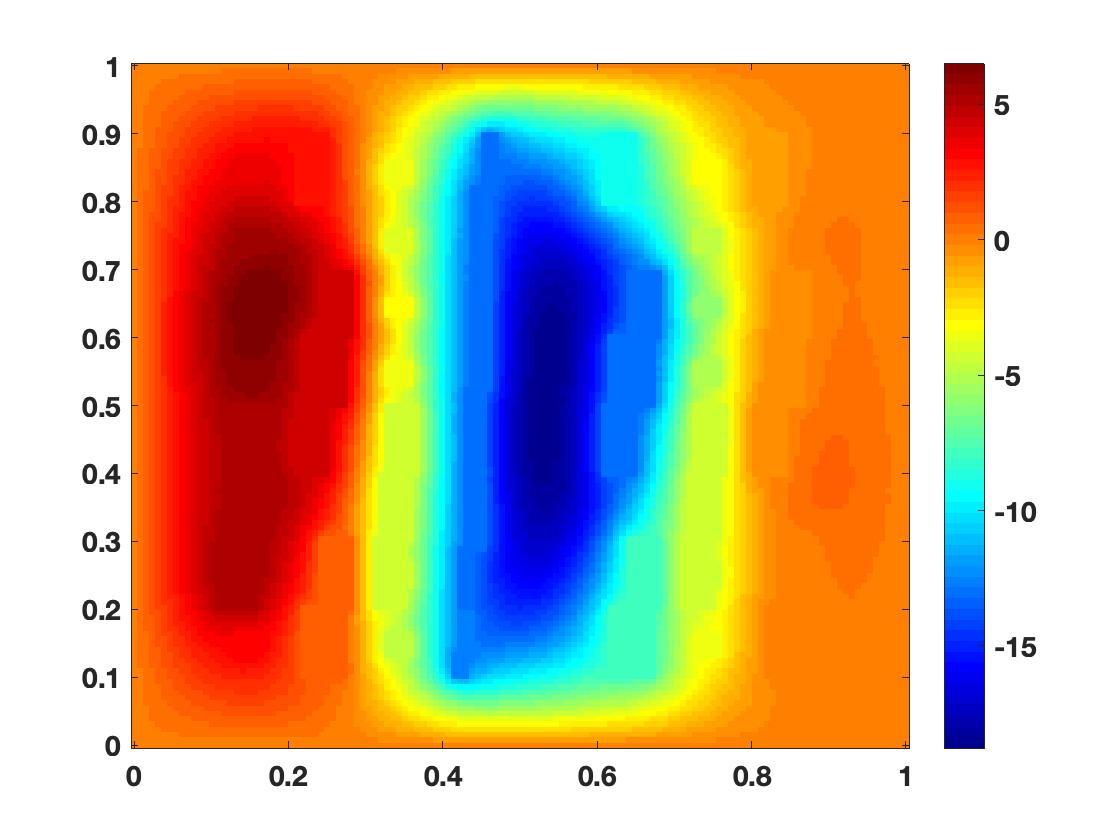}
		\includegraphics[trim={3cm 1.8cm 2.8cm 2cm},clip,width=0.24 \textwidth]{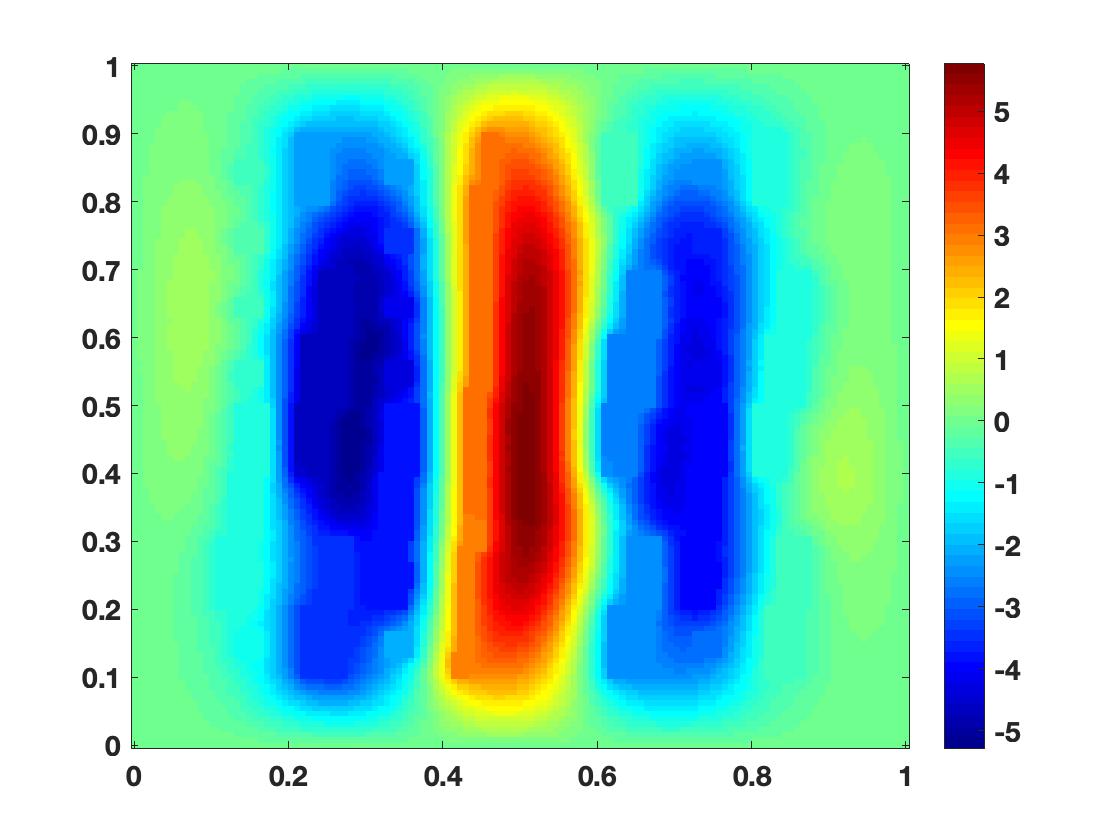}
		\includegraphics[trim={3cm 1.8cm 2.8cm 2cm},clip,width=0.24 \textwidth]{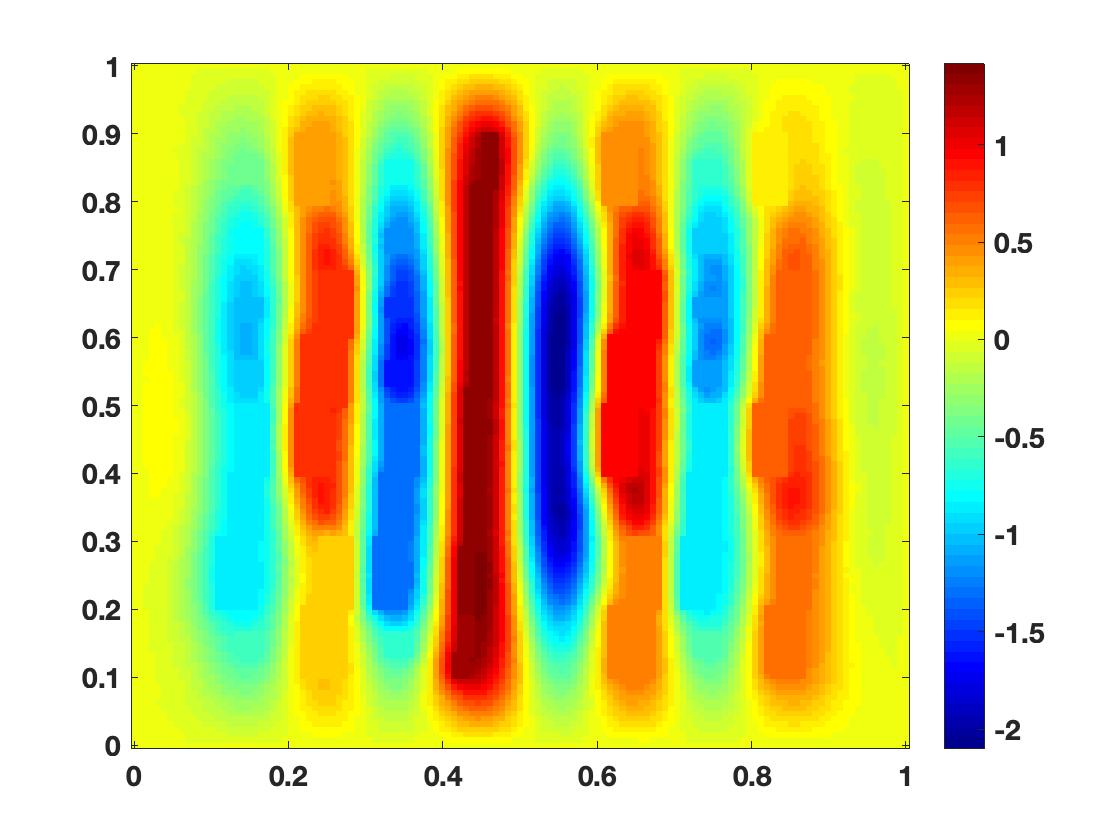}
		\caption{Numerical solutions $ U_k^n$ for $n=1, 3, 5, 10$ from Algorithm \ref{algorithm:wavelet+parareal} with $\Delta T=0.1$ and $\delta t=10^{-3}$, Crank-Nicolson scheme: iteration number $k=0$ (top), $k=1$ (middle) and $k=2$ (bottom).}
		\label{fig:PararealSol_NonzeroSource_level2_CN_coarser}
	\end{figure}	
The convergence history of Algorithm \ref{algorithm:wavelet+parareal} in $L^2(D)$-norm and $H^1_{\kappa}(D)$-norm is presented in Tables \ref{error:L2_NonzeroSource_l2_CN_coarser} and \ref{error:H1_NonzeroSource_l2_CN_coarser}, respectively. Similar to Experiment 1, we observe that 4 iterations is sufficient for Algorithm \ref{algorithm:wavelet+parareal} to reach the same accuracy as Algorithm \ref{algorithm:wavelet} at all discrete time levels under the $L^2(D)$-norm, while 2 iterations under the $H^1_{\kappa}(D)$-norm. Comparing Table \ref{error:L2_NonzeroSource_l2_BE_coarser} with Table \ref{error:L2_NonzeroSource_l2_CN_coarser}, one observes that Algorithm \ref{algorithm:wavelet+parareal} with Crank-Nicolson scheme outperforms that with backward Euler scheme under $L^2(D)$-norm.
	\begin{table}[H]
	\begin{center}
	\begin{tabular}{|c|c|c|c|c|c|c|}
	\hline
	$T^n$ & $\text{Rel}^{\text{EW}}_{L^2} (T^n)$  &$\text{Rel}^0_{L^2}(T^n)$& $\text{Rel}^1_{L^2}(T^n)$&$\text{Rel}^2_{L^2}(T^n)$& $\text{Rel}^3_{L^2}(T^n)$& $\text{Rel}^4_{L^2}(T^n)$
	\\ \hline
	0.1 &  0.3527    & 14.3828 & 0.3527 & 0.3527 & 0.3527 & 0.3527 \\
 0.2 &  0.5438     &30.6711 & 4.3959 & 0.5485 & 0.5485 & 0.5485 \\
 0.3 &   0.3510    &42.2133 & 13.2555 & 1.6259 & 0.3515 & 0.3515 \\
 0.4 &   0.3443    &22.5947 & 8.2409 & 4.9804 & 0.5513 & 0.3431 \\
 0.5 &   0.4847    &11.4592 & 6.4455 & 1.0726 & 2.2873 & 0.4642 \\
 0.6 &   0.7121    &8.6259 & 1.7874 & 2.1688 & 1.1929 & 0.6720 \\
 0.7 &    0.7049   &8.9372 & 0.9898 & 1.2377 & 0.7162 & 0.7834 \\
 0.8 &   0.9467    &11.0146 & 1.8508 & 0.9946 & 0.9773 & 0.9465 \\
 0.9 &   1.0636    &6.2513 & 1.8811 & 1.1818 & 1.0505 & 1.0737 \\
 1.0 &   0.9072    &5.1982 & 1.0806 & 0.9286 & 0.9252 & 0.9054 \\
	\hline
	\end{tabular}
	\end{center}
	\vspace{-.4cm}
	\caption{Convergence history of Algorithm \ref{algorithm:wavelet+parareal}  in relative $L^2(D)$ error for Experiment 2: Crank-Nicolson scheme with ${\Delta T}=0.1$ and ${\delta t}=10^{-3}$.}
	\label{error:L2_NonzeroSource_l2_CN_coarser}
	\end{table}
	
	\begin{table}[H]
	\begin{center}
	\begin{tabular}{|c|c|c|c|c|c|c|}
	\hline
	$T^n$ & $\text{Rel}^{\text{EW}}_{H_{\kappa}^1} (T^n)$&$\text{Rel}^0_{H_{\kappa}^1}(T^n)$& $\text{Rel}^1_{H_{\kappa}^1}(T^n)$& $\text{Rel}^2_{H_{\kappa}^1}(T^n)$& $\text{Rel}^3_{H_{\kappa}^1}(T^n)$& $\text{Rel}^4_{H_{\kappa}^1}(T^n)$
	\\ \hline
	 0.1 & 6.9448      &16.5638 & 6.9448 & 6.9448 & 6.9448 & 6.9448 \\
 0.2 &  5.6381     &25.0878 & 6.7485 & 5.6383 & 5.6383 & 5.6383 \\
 0.3 &  4.9062     &27.6524 & 9.3773 & 5.0384 & 4.9062 & 4.9062 \\
 0.4 &  4.7218     &16.6013 & 6.2318 & 5.2453 & 4.7328 & 4.7218 \\
 0.5 &  4.8979     &9.3734 & 5.6700 & 4.9104 & 4.9386 & 4.8996 \\
 0.6 &   5.3108    &7.9758 & 5.3592 & 5.3401 & 5.3153 & 5.3126 \\
 0.7 &   5.3062    &6.9107 & 5.3156 & 5.3121 & 5.3064 & 5.3069 \\
 0.8 &   6.2664    &7.6015 & 6.2824 & 6.2667 & 6.2669 & 6.2664 \\
 0.9 &   6.4269    &7.0167 & 6.4412 & 6.4274 & 6.4270 & 6.4269 \\
 1.0 &   4.9341    &5.2526 & 4.9369 & 4.9343 & 4.9340 & 4.9341 \\
	\hline
	\end{tabular}
	\end{center}
	\vspace{-.4cm}
	\caption{Convergence history of Algorithm \ref{algorithm:wavelet+parareal}  in relative $H_{\kappa}^1(D)  $ error for Experiment 2: Crank-Nicolson scheme with ${\Delta T}=0.1$ and ${\delta t}=10^{-3}$.}
	\label{error:H1_NonzeroSource_l2_CN_coarser}
	\end{table}

\subsubsection*{Experiment 3: backward Euler with $\frac{\Delta T}{\delta t}=10$}
We are also interested in studying how the coarse solver and fine solver affect the performance of our proposed WEMP algorithm. To this end,  we choose $\Delta T=10^{-2}$, $\delta t=10^{-3}$ and utilize backward Euler scheme in time discretization. Note that the ratio between the coarse time step and fine time step is smaller than that in Experiment 1.

 The multiscale solutions from Algorithm \ref{algorithm:wavelet} with backward Euler scheme are presented in Figure \ref{fig:EWsol_NonzeroSource_level2_BE}. We present the numerical solutions $U^n_k$ for $n=10, 30, 50, 100$ from Algorithm \ref{algorithm:wavelet+parareal} with iteration number $k=0,1,2$ in Figure \ref{fig:PararealSol_NonzeroSource_level2_Backward}.	One can observe the same convergence behavior as in Experiment 1.
\begin{figure}[H]
		\centering
		\includegraphics[trim={3cm 1.8cm 2.8cm 2cm},clip,width=0.24 \textwidth]{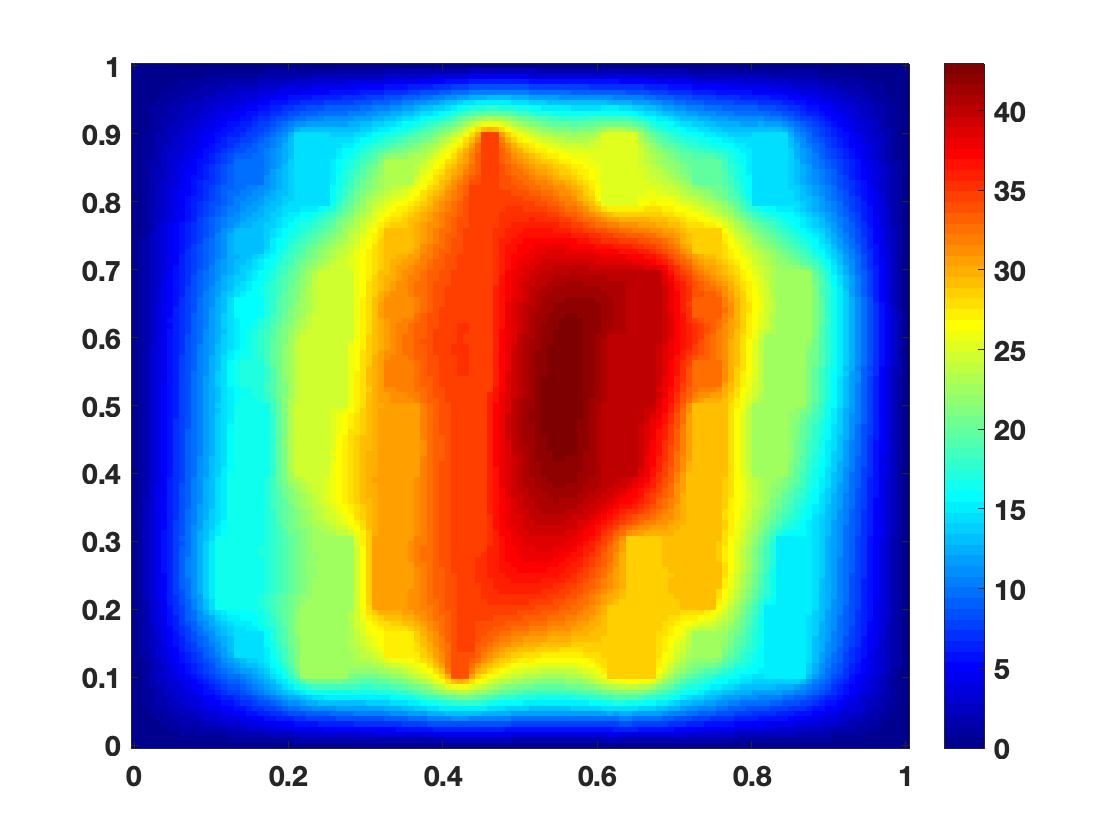}
		\includegraphics[trim={3cm 1.8cm 2.8cm 2cm},clip,width=0.24 \textwidth]{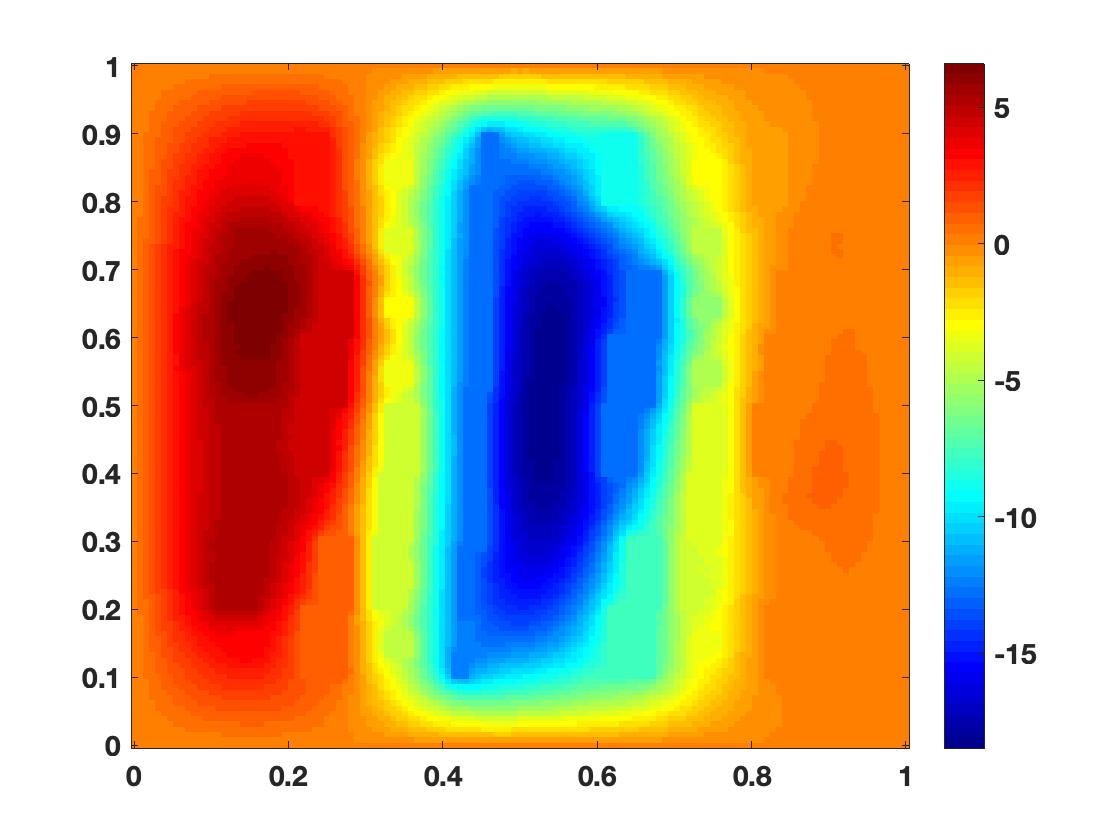}
		\includegraphics[trim={3cm 1.8cm 2.8cm 2cm},clip,width=0.24 \textwidth]{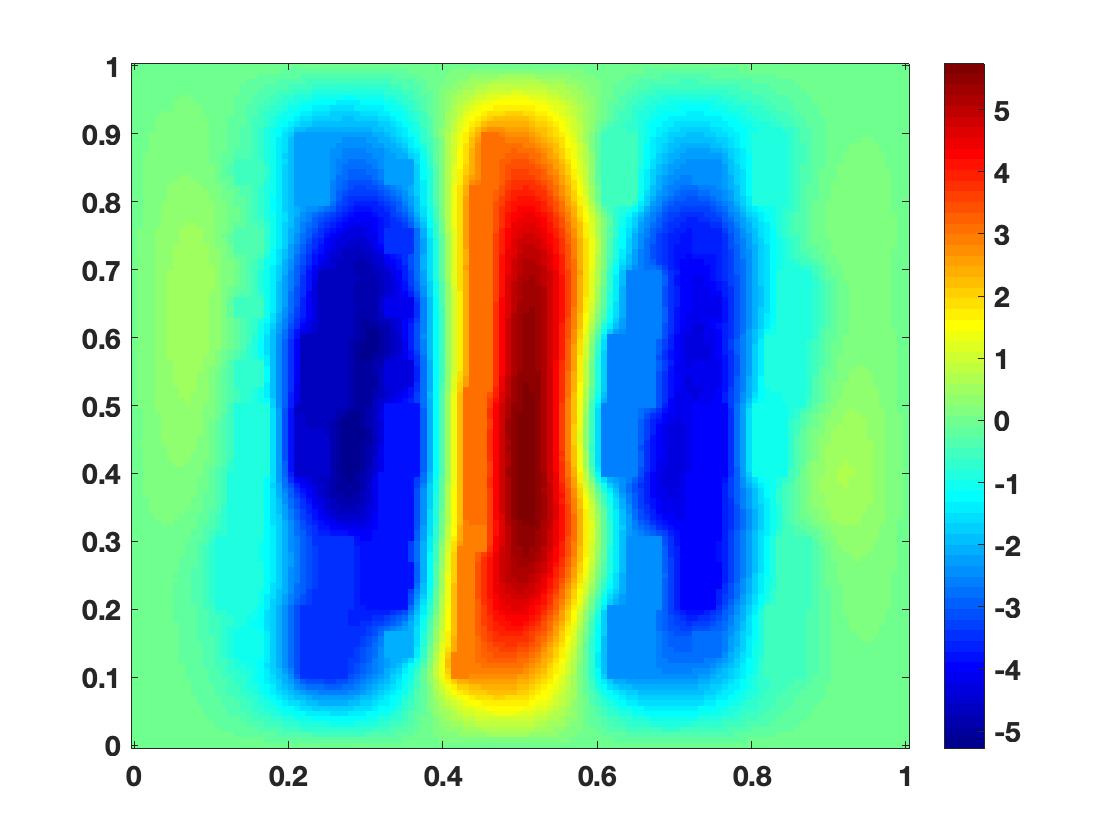}
		\includegraphics[trim={3cm 1.8cm 2.8cm 2cm},clip,width=0.24 \textwidth]{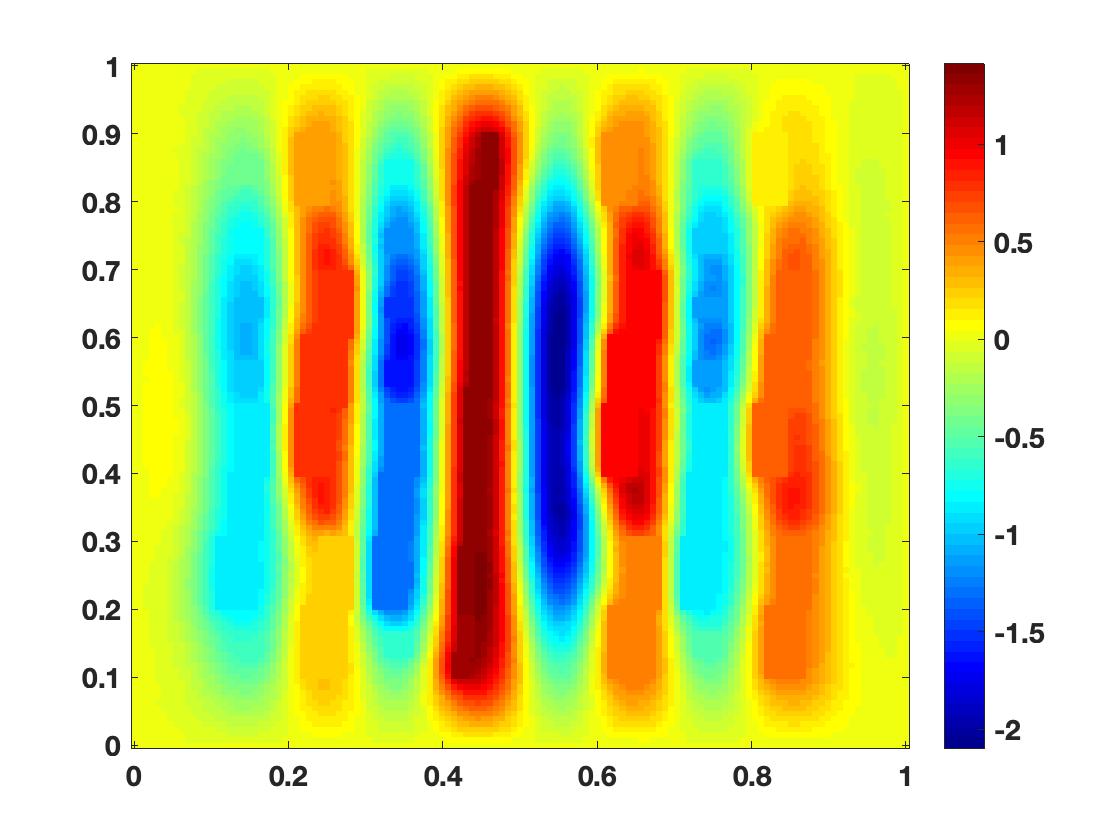}
		
		\caption{Multiscale solution from Algorithm \ref{algorithm:wavelet} with $\delta t=10^{-3}$ and $\ell=2$, backward Euler scheme: $u_{\text{ms},\ell}^{\text{EW,100}}$,  $u_{\text{ms},\ell}^{\text{EW,300}}$, $u_{\text{ms},\ell}^{\text{EW,500}}$ and $u_{\text{ms},\ell}^{\text{EW,1000}}$.}
		\label{fig:EWsol_NonzeroSource_level2_BE}
	\end{figure}	
	
	\begin{figure}[H]
		\centering
		\includegraphics[trim={3cm 1.8cm 2.8cm 2cm},clip,width=0.24 \textwidth]{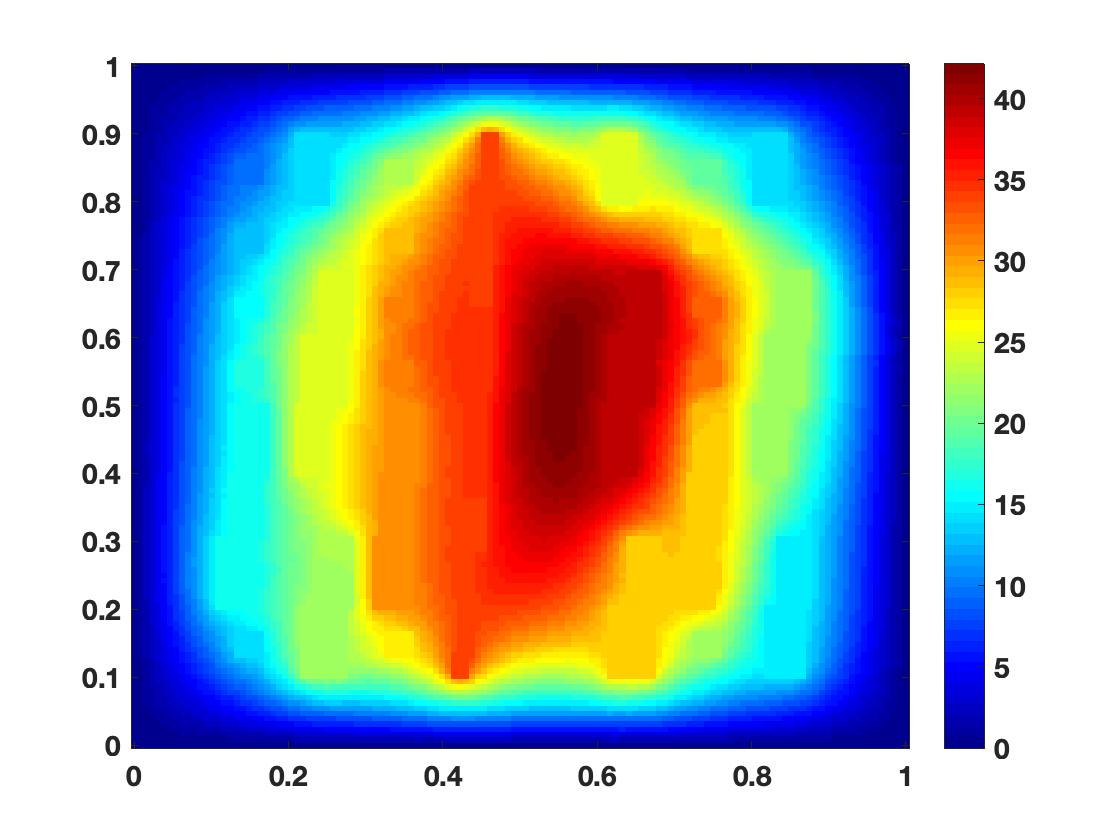}
		\includegraphics[trim={3cm 1.8cm 2.8cm 2cm},clip,width=0.24 \textwidth]{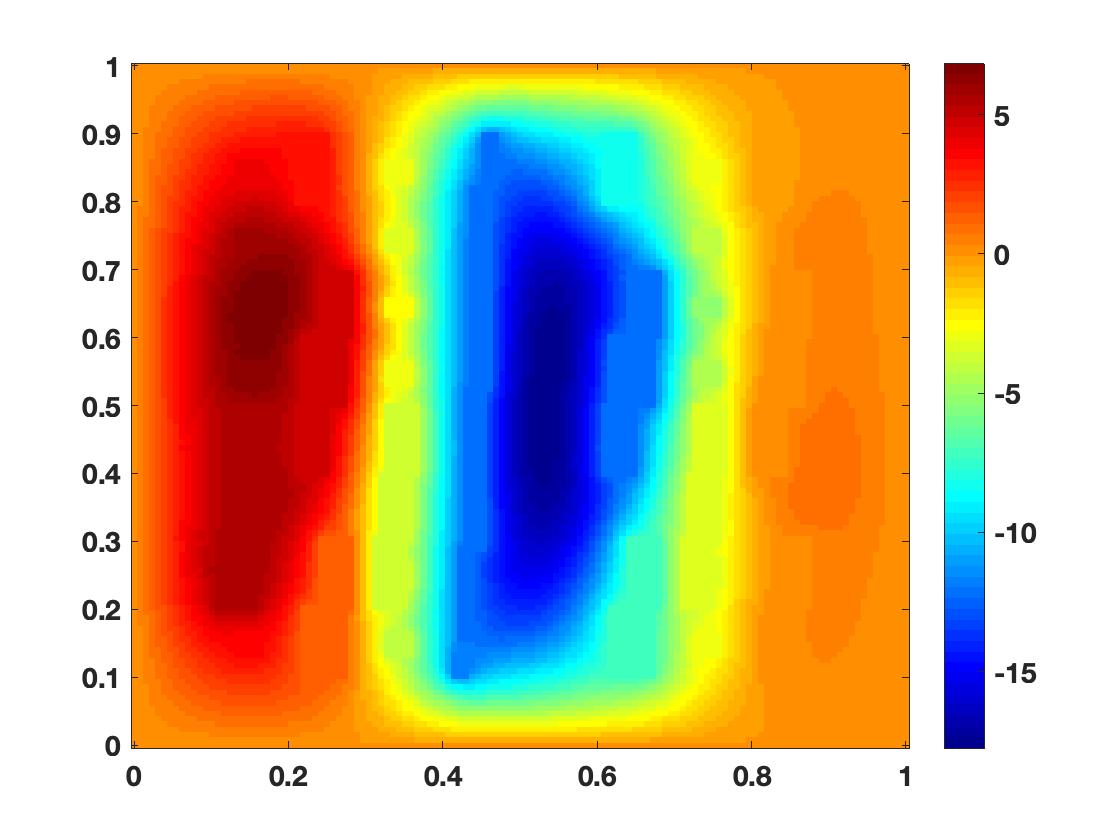}
		\includegraphics[trim={3cm 1.8cm 2.8cm 2cm},clip,width=0.24 \textwidth]{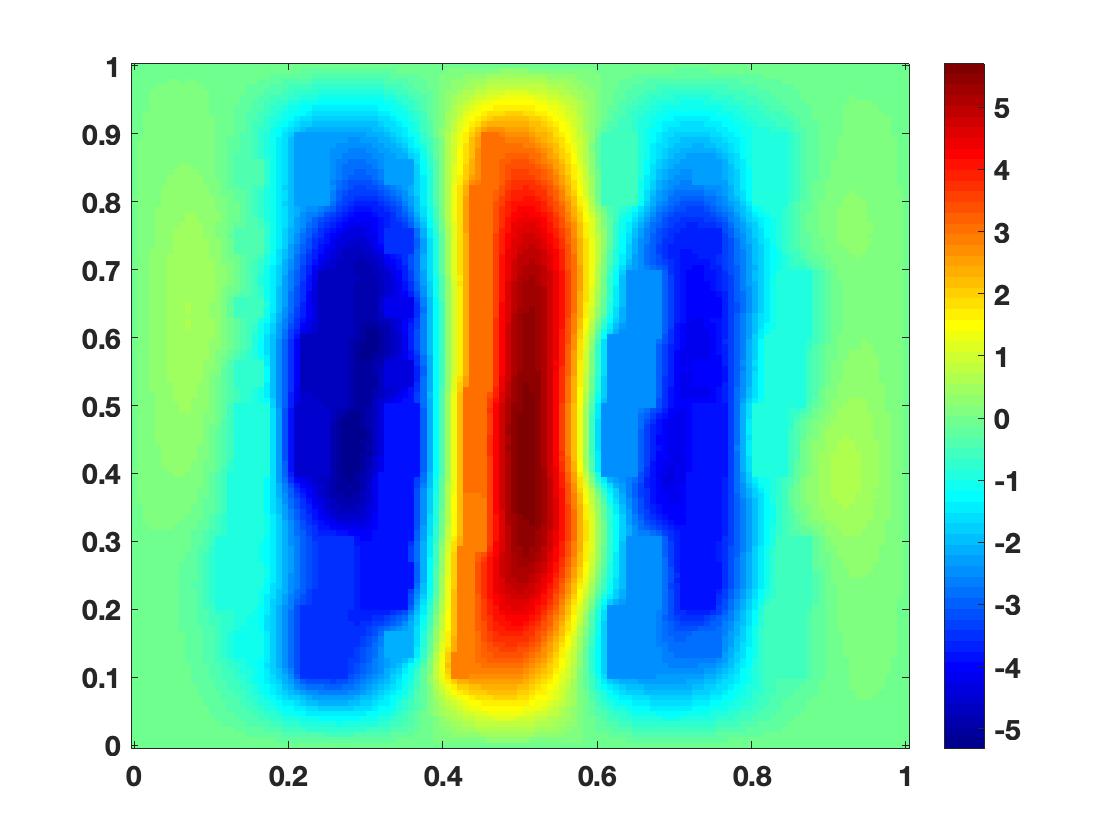}
		\includegraphics[trim={3cm 1.8cm 2.8cm 2cm},clip,width=0.24 \textwidth]{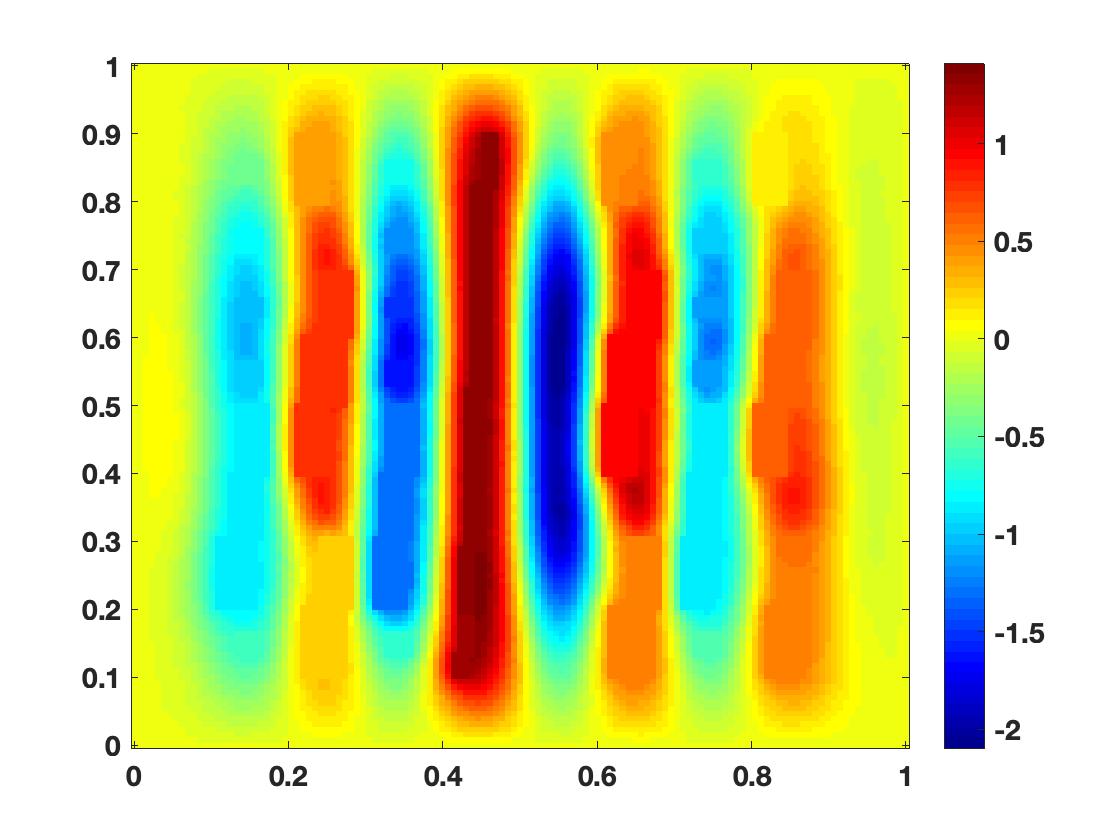}\\
		\includegraphics[trim={3cm 1.8cm 2.8cm 2cm},clip,width=0.24 \textwidth]{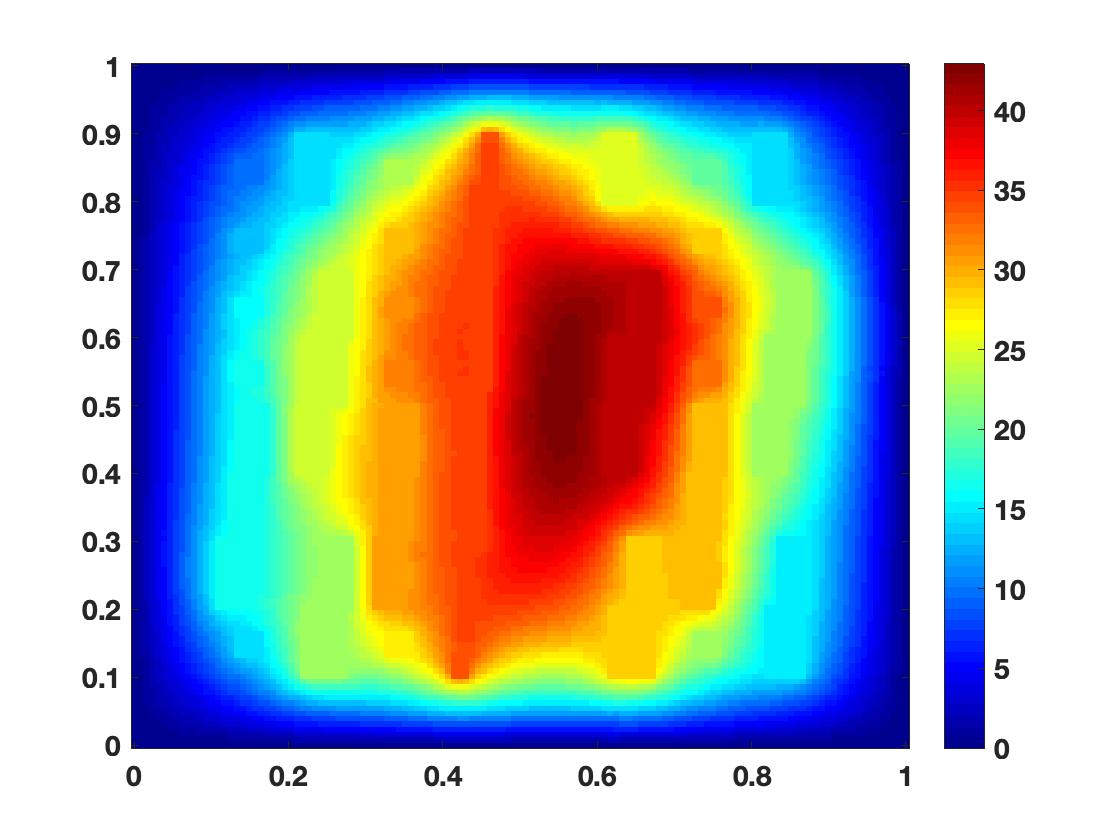}
		\includegraphics[trim={3cm 1.8cm 2.8cm 2cm},clip,width=0.24 \textwidth]{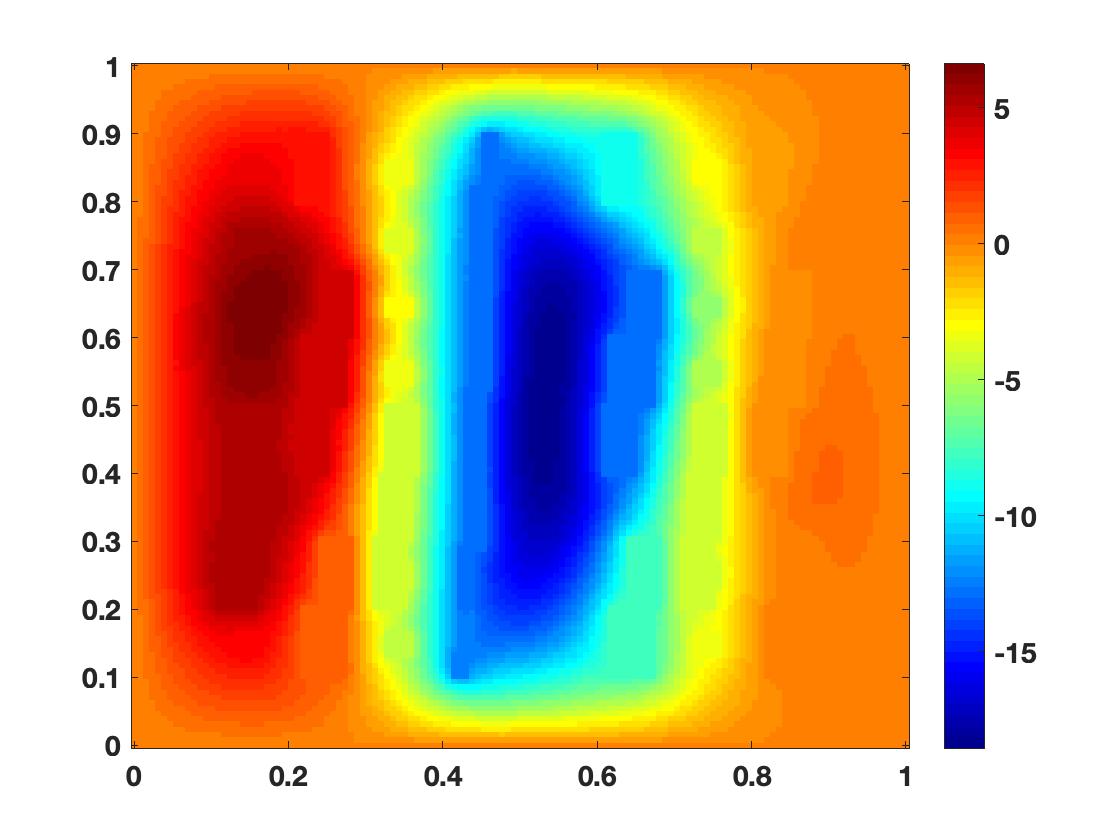}
		\includegraphics[trim={3cm 1.8cm 2.8cm 2cm},clip,width=0.24 \textwidth]{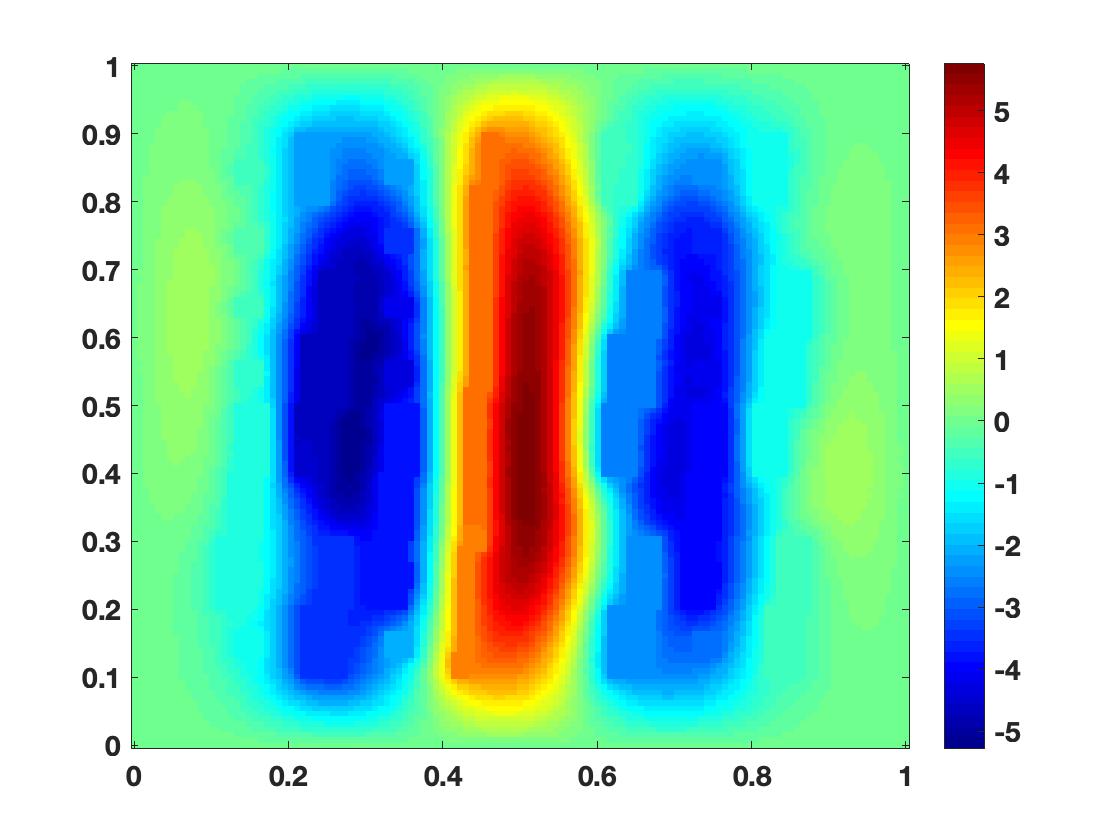}
		\includegraphics[trim={3cm 1.8cm 2.8cm 2cm},clip,width=0.24 \textwidth]{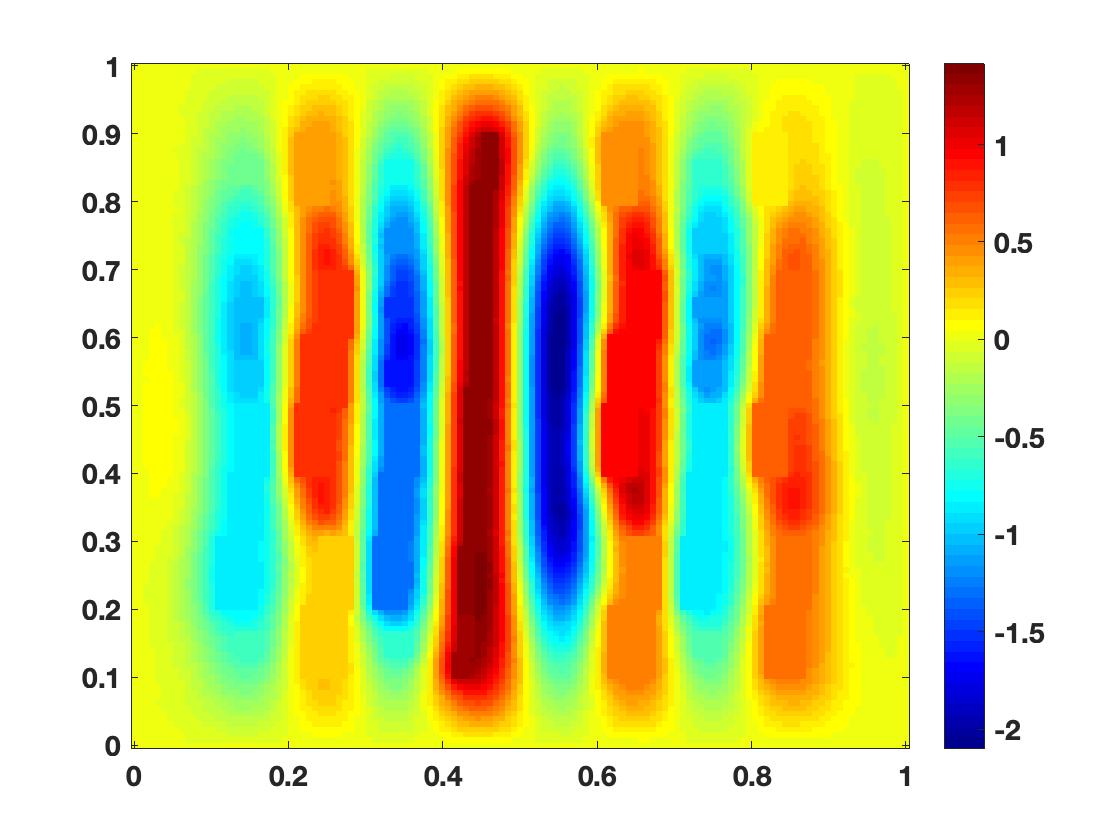}\\
		\includegraphics[trim={3cm 1.8cm 2.8cm 2cm},clip,width=0.24 \textwidth]{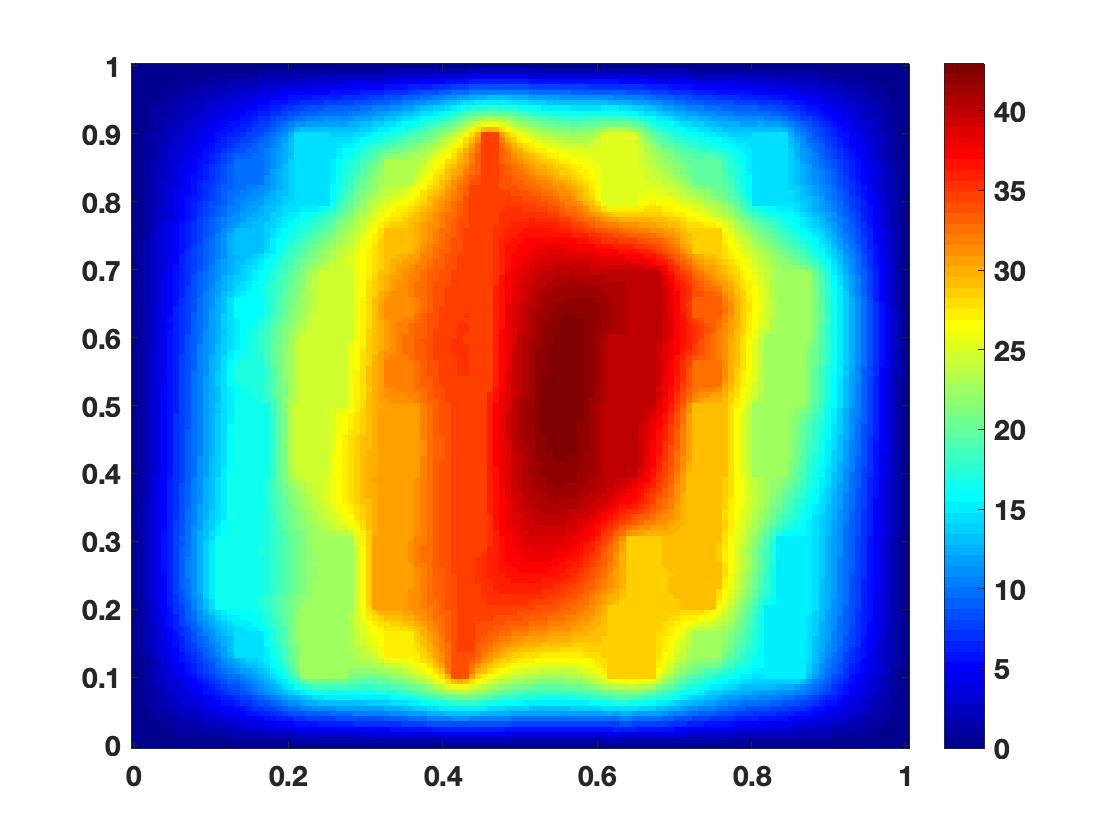}
		\includegraphics[trim={3cm 1.8cm 2.8cm 2cm},clip,width=0.24 \textwidth]{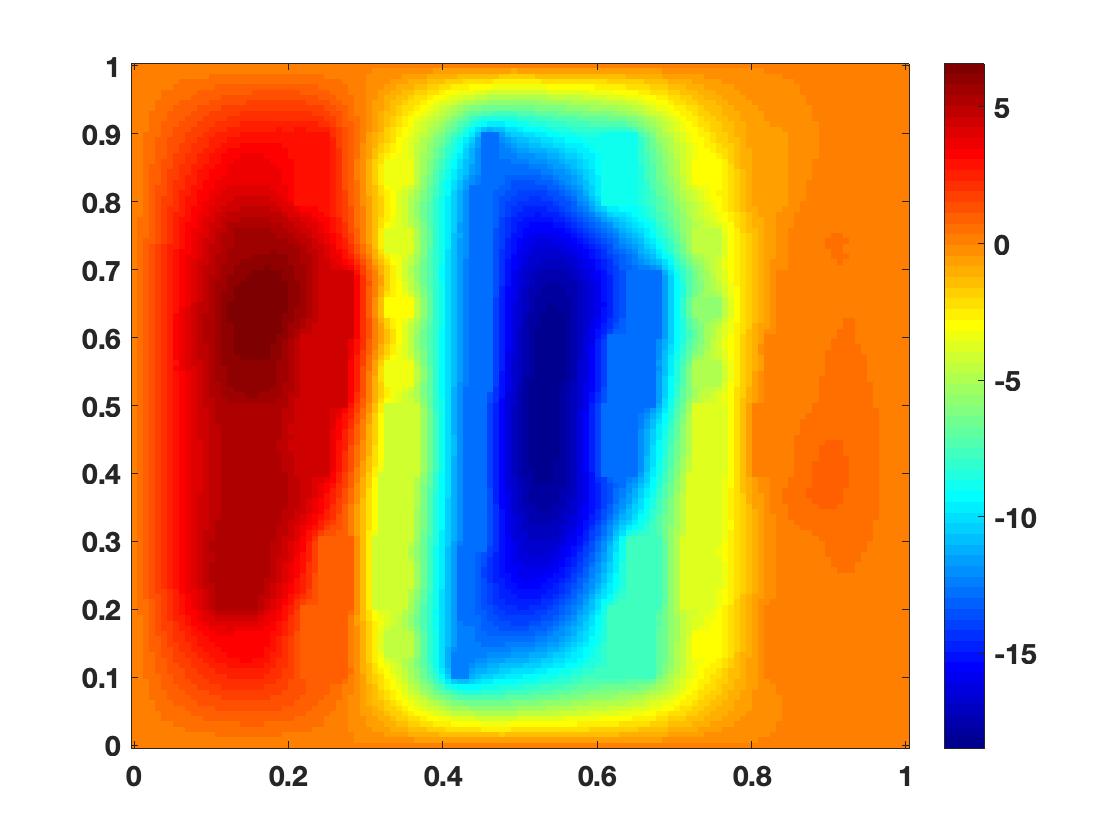}
		\includegraphics[trim={3cm 1.8cm 2.8cm 2cm},clip,width=0.24 \textwidth]{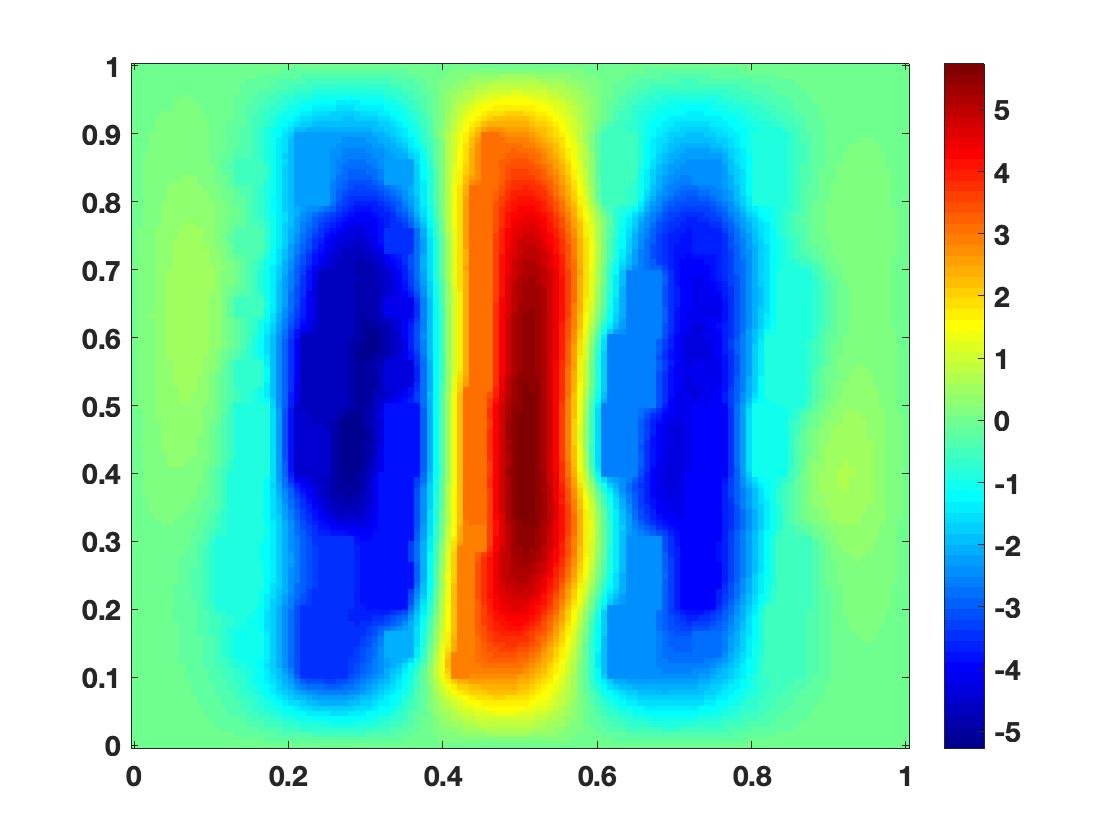}
		\includegraphics[trim={3cm 1.8cm 2.8cm 2cm},clip,width=0.24 \textwidth]{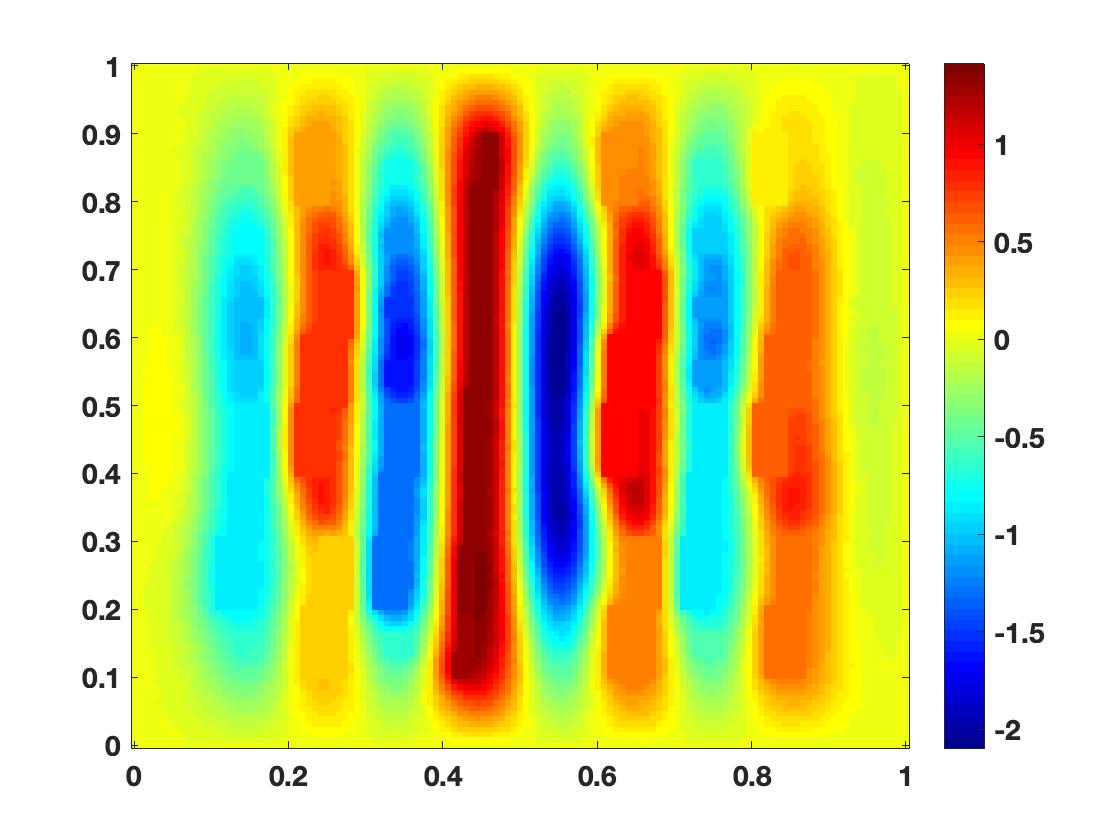}\\
		
		\caption{Numerical solutions $ U_k^n$ for $n=10, 30, 50, 100$ from Algorithm \ref{algorithm:wavelet+parareal} with $\Delta T=10^{-2}$ and $\delta t=10^{-3}$, backward Euler scheme: iteration $k=0$ (top), $k=1$ (middle) and $k=2$ (bottom).}
		\label{fig:PararealSol_NonzeroSource_level2_Backward}
	\end{figure}	
The convergence history of Algorithm \ref{algorithm:wavelet+parareal} in $L^2(D)$-norm and $H^1_{\kappa}(D)$-norm is presented in Tables \ref{error:L2_parareal_NonzeroSource_l2_BN} and \ref{error:H1_parareal_NonzeroSource_l2_BN}.	
Comparing Table \ref{error:L2_NonzeroSource_l2_BE_coarser} with Table \ref{error:L2_parareal_NonzeroSource_l2_BN}, one can see 1 iteration is sufficient for the numerical solutions from Algorithm \ref{algorithm:wavelet+parareal} to reach the same accuracy as multiscale solutions from Algorithm \ref{algorithm:wavelet} under $L^2(D)$-norm and $H^1_{\kappa}(D)$-norm when the coarse time step $\Delta T=10^{-2}$ becomes smaller. However, this involves more coarse solvers for each iteration.  Furthermore, a decreased coarse time step is only practical when sufficient processors are available. 
\begin{table}[H]
	\begin{center}
	\begin{tabular}{|c|c|c|c|c|c|c|}
	\hline
	$T^n$ & $\text{Rel}^{\text{EW}}_{L^2} (T^n)$  &$\text{Rel}^0_{L^2}(T^n)$& $\text{Rel}^1_{L^2}(T^n)$&$\text{Rel}^2_{L^2}(T^n)$& $\text{Rel}^3_{L^2}(T^n)$& $\text{Rel}^4_{L^2}(T^n)$
	\\ \hline
	0.1 & 0.5671 & 2.4095 & 0.6072 & 0.5563 & 0.5680 & 0.5670 \\
 0.2 & 0.8234 & 3.4560 & 0.4291 & 0.8527 & 0.8237 & 0.8230 \\
 0.3 & 0.8258 & 6.2194 & 0.4501 & 0.8474 & 0.8253 & 0.8258 \\
 0.4 & 0.5897 & 3.0769 & 0.5520 & 0.6288 & 0.5874 & 0.5891 \\
 0.5 & 0.5323 & 1.9183 & 0.5216 & 0.5621 & 0.5273 & 0.5332 \\
 0.6 & 0.7072 & 1.1525 & 0.6757 & 0.7204 & 0.7052 & 0.7078 \\
 0.7 & 0.7229 & 1.4333 & 0.7154 & 0.7287 & 0.7222 & 0.7229 \\
 0.8 & 0.9680 & 1.9263 & 0.9493 & 0.9720 & 0.9680 & 0.9679 \\
 0.9 & 1.0681 & 1.6490 & 1.0817 & 1.0673 & 1.0684 & 1.0681 \\
 1.0 & 0.9145 & 1.1618 & 0.9204 & 0.9146 & 0.9145 & 0.9146 \\
	\hline
	\end{tabular}
	\end{center}
	\vspace{-.4cm}
	\caption{Convergence history of Algorithm \ref{algorithm:wavelet+parareal} in relative $L^2(D)$ error for Experiment 3: backward Euler scheme with ${\Delta T}=10^{-2}$ and ${\delta t}=10^{-3}$.}
	\label{error:L2_parareal_NonzeroSource_l2_BN}
	\end{table}

	\begin{table}[H]
	\begin{center}
	\begin{tabular}{|c|c|c|c|c|c|c|}
	\hline
	$T^n$ & $\text{Rel}^{\text{EW}}_{H_{\kappa}^1} (T^n)$&$\text{Rel}^0_{H_{\kappa}^1}(T^n)$& $\text{Rel}^1_{H_{\kappa}^1}(T^n)$& $\text{Rel}^2_{H_{\kappa}^1}(T^n)$& $\text{Rel}^3_{H_{\kappa}^1}(T^n)$& $\text{Rel}^4_{H_{\kappa}^1}(T^n)$
	\\ \hline
	0.1 & 6.9437 & 7.2482 & 6.9435 & 6.9438 & 6.9437 & 6.9437 \\
 0.2 & 5.6489 & 6.2363 & 5.6351 & 5.6509 & 5.6488 & 5.6489 \\
 0.3 & 4.9158 & 6.1653 & 4.9083 & 4.9177 & 4.9157 & 4.9158 \\
 0.4 & 4.7240 & 5.0851 & 4.7270 & 4.7260 & 4.7239 & 4.7240 \\
 0.5 & 4.8984 & 5.0041 & 4.8988 & 4.8993 & 4.8984 & 4.8984 \\
 0.6 & 5.3109 & 5.3409 & 5.3116 & 5.3114 & 5.3108 & 5.3109 \\
 0.7 & 5.3064 & 5.3324 & 5.3066 & 5.3066 & 5.3064 & 5.3064 \\
 0.8 & 6.2666 & 6.2916 & 6.2667 & 6.2667 & 6.2666 & 6.2666 \\
 0.9 & 6.4270 & 6.4420 & 6.4270 & 6.4270 & 6.4270 & 6.4270 \\
 1.0 & 4.9341 & 4.9390 & 4.9341 & 4.9341 & 4.9341 & 4.9341 \\
	\hline
	\end{tabular}
	\end{center}
	\vspace{-.4cm}
	\caption{Convergence history of Algorithm \ref{algorithm:wavelet+parareal} in relative $H_{\kappa}^1(D)$ error for Experiment 3: backward Euler scheme with ${\Delta T}=10^{-2}$ and ${\delta t}=10^{-3}$.}
	\label{error:H1_parareal_NonzeroSource_l2_BN}
	\end{table}
\subsection{Numerical tests with zero source term}	\label{subsection: Numerical tests with zero source term}
In this section, we test the performance of Algorithm \ref{algorithm:wavelet+parareal} for Problem (\ref{eqn:pde}) with backward Euler scheme and Crank-Nicolson scheme. The source term $f:=0$. Consequently, the solution decays rapidly to 0. To generate solutions with reasonable size, we set the final time $T=0.1$, the coarse time step $\Delta T:=10^{-2}$ and the fine time step $\delta t=10^{-3}$. The initial data and permeability are the same as in the previous section.  We use backward Euler scheme with time step $10^{-3}$ to solve for the reference solutions $u^n_h$. The reference solutions $u^n_h$ for $n=10,30,50,100$ are plotted in Figure \ref{fig:fine_solution_ZeroSource}. 
	\begin{figure}[H]
		\centering
		\includegraphics[trim={3cm 1.8cm 2.8cm 2cm},clip,width=0.24 \textwidth]{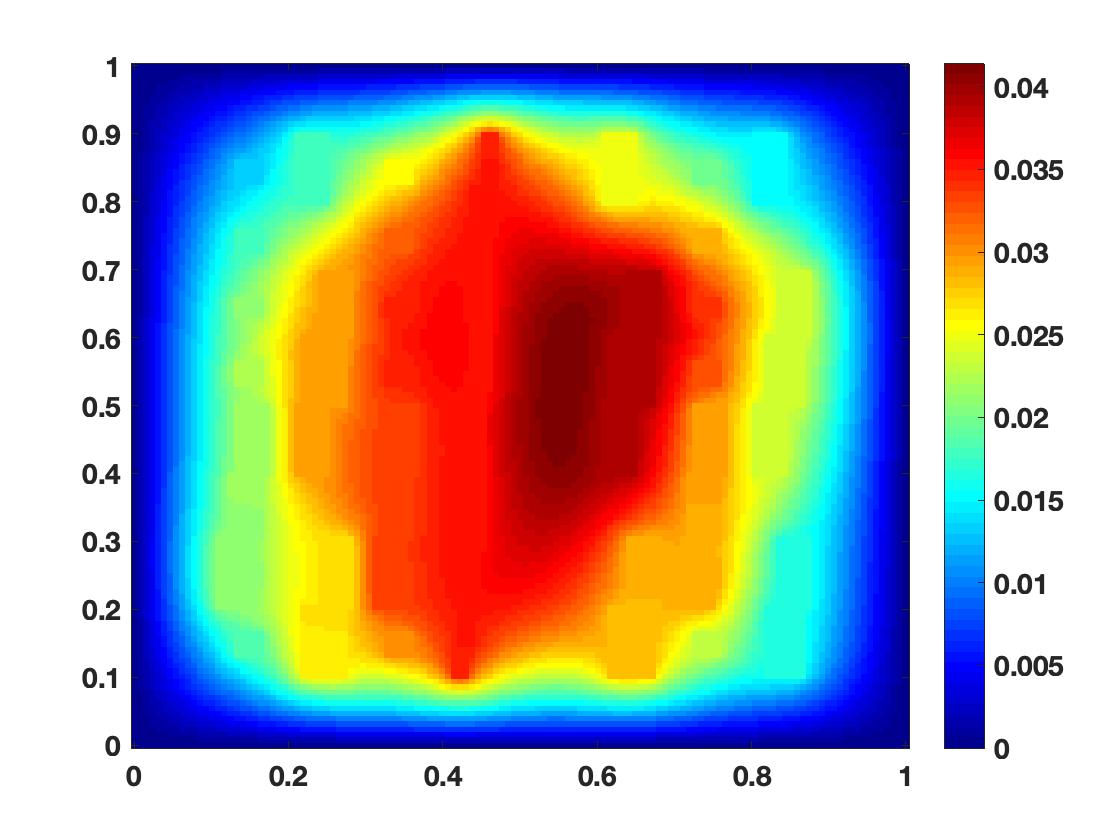}
		\includegraphics[trim={3cm 1.8cm 2.8cm 2cm},clip,width=0.24 \textwidth]{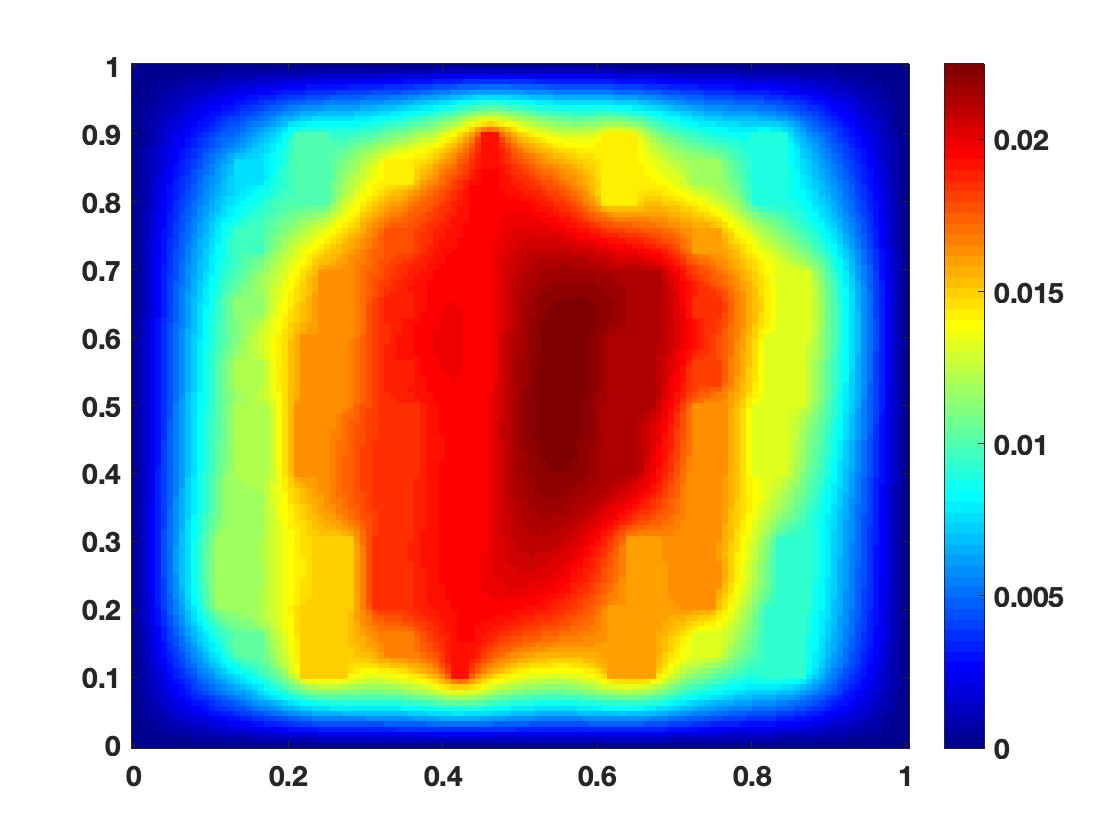}
		\includegraphics[trim={3cm 1.8cm 2.8cm 2cm},clip,width=0.24 \textwidth]{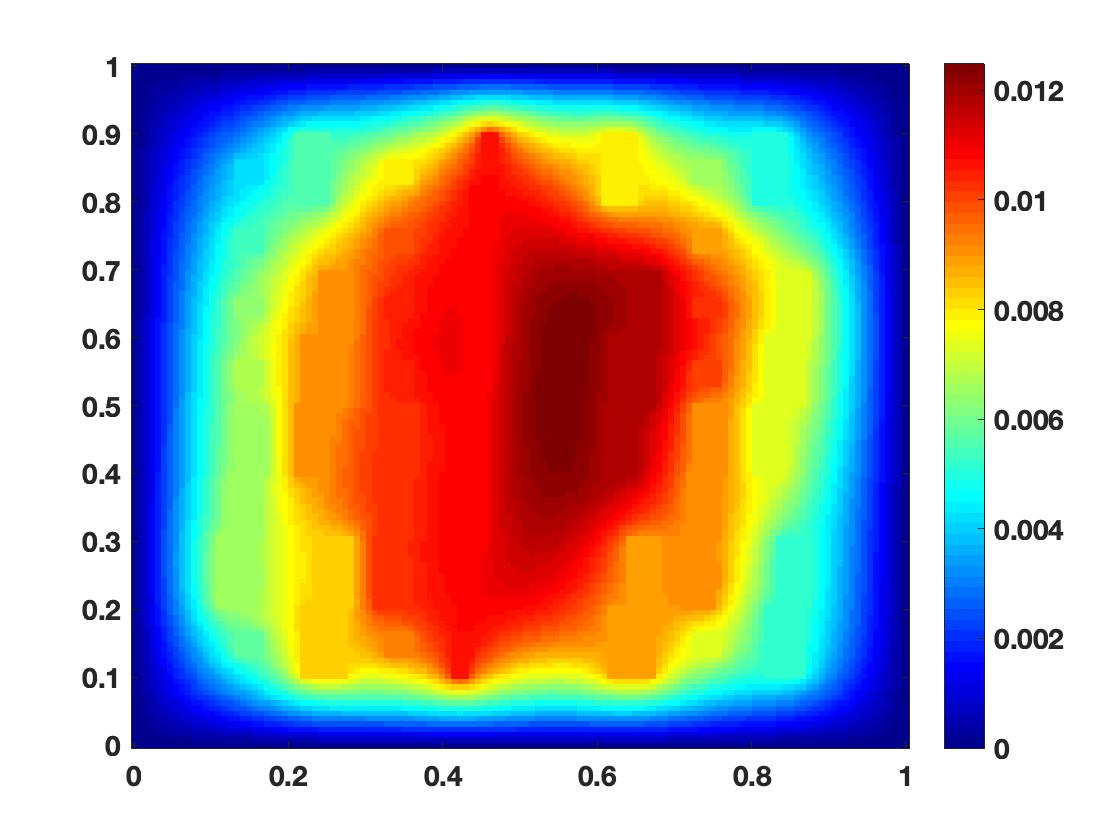}
		\includegraphics[trim={3cm 1.8cm 2.8cm 2cm},clip,width=0.24 \textwidth]{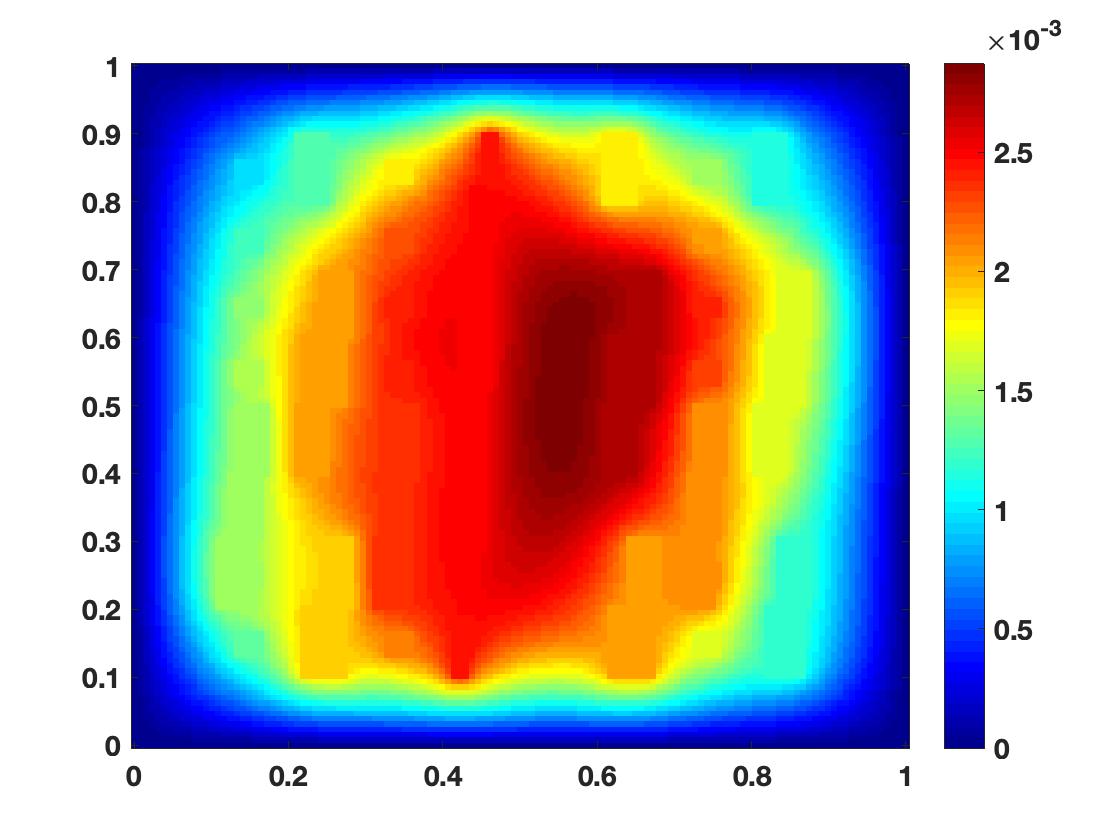}
		\caption{Numerical solutions $u^n_h$ to (\ref{eqn:weakform_h}) with $f=0$ for $n=10,30,50$ and $100$ with $\delta t=10^{-3}$.}\label{fig:fine_solution_ZeroSource}
	\end{figure}
The multiscale solutions $u_{\text{ms},\ell}^{\text{EW,n}}$ for $n=10,30,50,100$ from Algorithm \ref{algorithm:wavelet} with backward Euler scheme and time step size $\delta t=10^{-3}$ are presented in Figure \ref{fig:WEMsFEM_solution_ZeroSource}. We present the numerical solutions $U^n_k$ for $n=1,3,5,10$ from Algorithm \ref{algorithm:wavelet+parareal} with iteration number $k=0,1,2$ in Figure \ref{fig:PararealSol_ZeroSource_level2}. One can observe $U^n_k$ converges to $u_{\text{ms},\ell}^{\text{EW,n}}$ as the iteration number $k$ increases.
\begin{figure}[H]
		\centering
		\includegraphics[trim={3cm 1.8cm 2.8cm 2cm},clip,width=0.24 \textwidth]{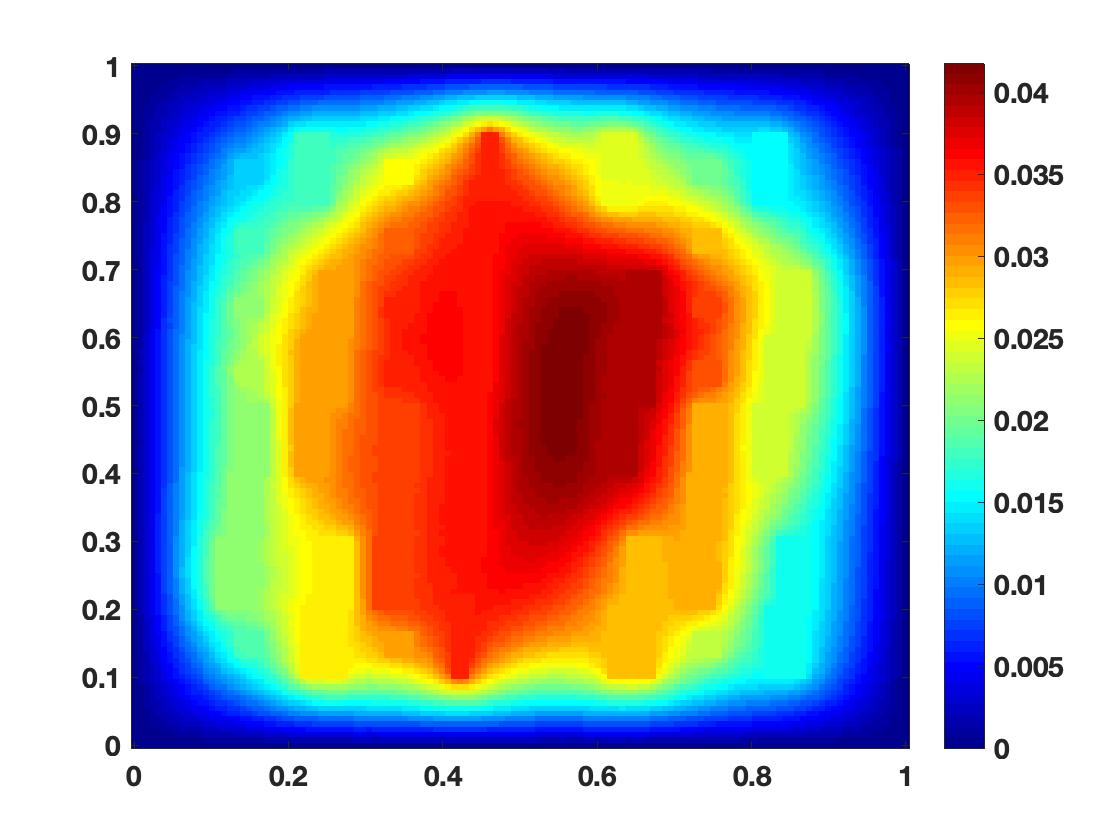}
		\includegraphics[trim={3cm 1.8cm 2.8cm 2cm},clip,width=0.24 \textwidth]{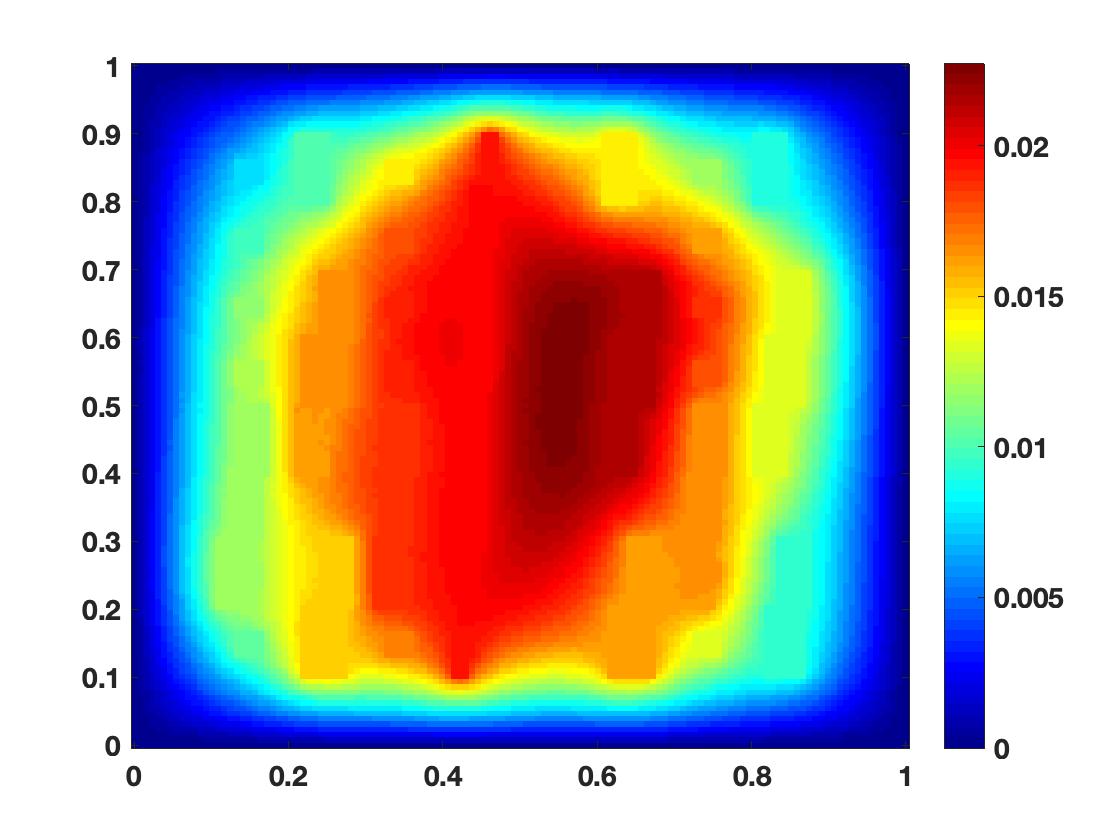}
		\includegraphics[trim={3cm 1.8cm 2.8cm 2cm},clip,width=0.24 \textwidth]{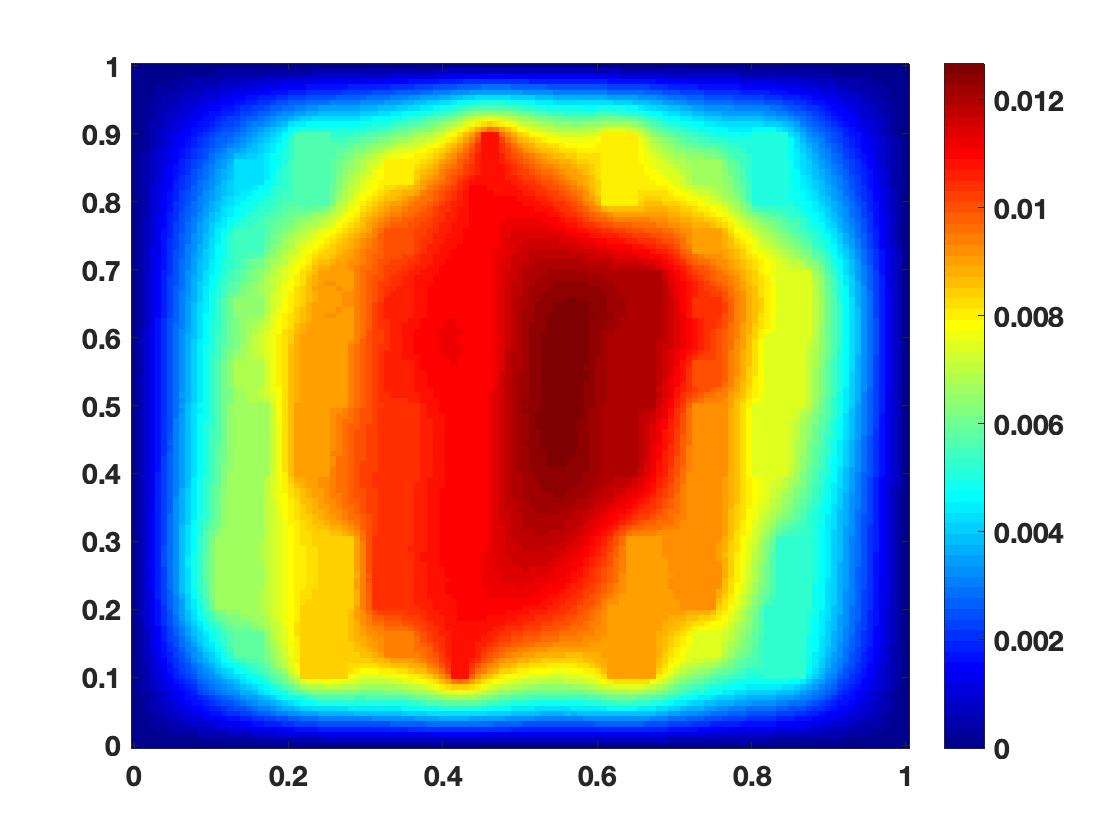}
		\includegraphics[trim={3cm 1.8cm 2.8cm 2cm},clip,width=0.24 \textwidth]{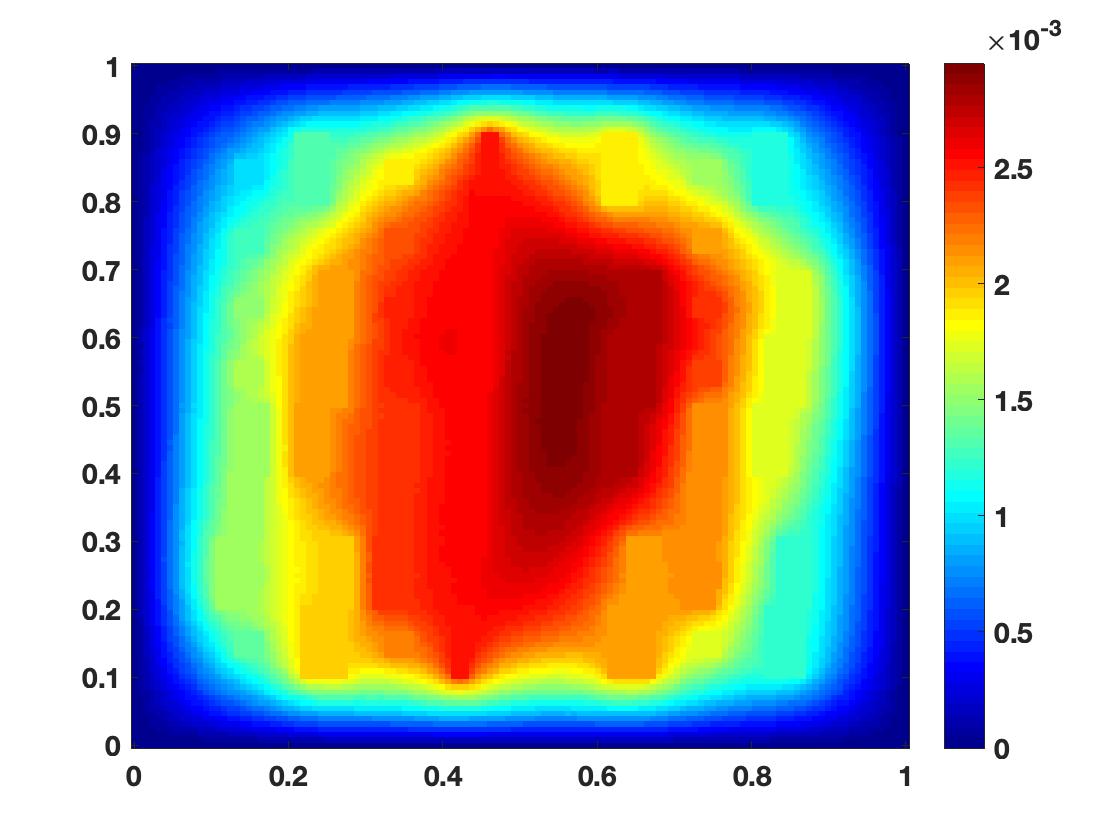}
		\caption{Multiscale solutions from Algorithm \ref{algorithm:wavelet} with $\delta t=10^{-3}$ and $\ell=2$, backward Euler scheme: $u_{\text{ms},\ell}^{\text{EW,10}}$, $u_{\text{ms},\ell}^{\text{EW,30}}$,  $u_{\text{ms},\ell}^{\text{EW,50}}$ and $u_{\text{ms},\ell}^{\text{EW,100}}$. }
		\label{fig:WEMsFEM_solution_ZeroSource}
	\end{figure}
	\begin{figure}[H]
		\centering
		\includegraphics[trim={3cm 1.8cm 2.8cm 2cm},clip,width=0.24 \textwidth]{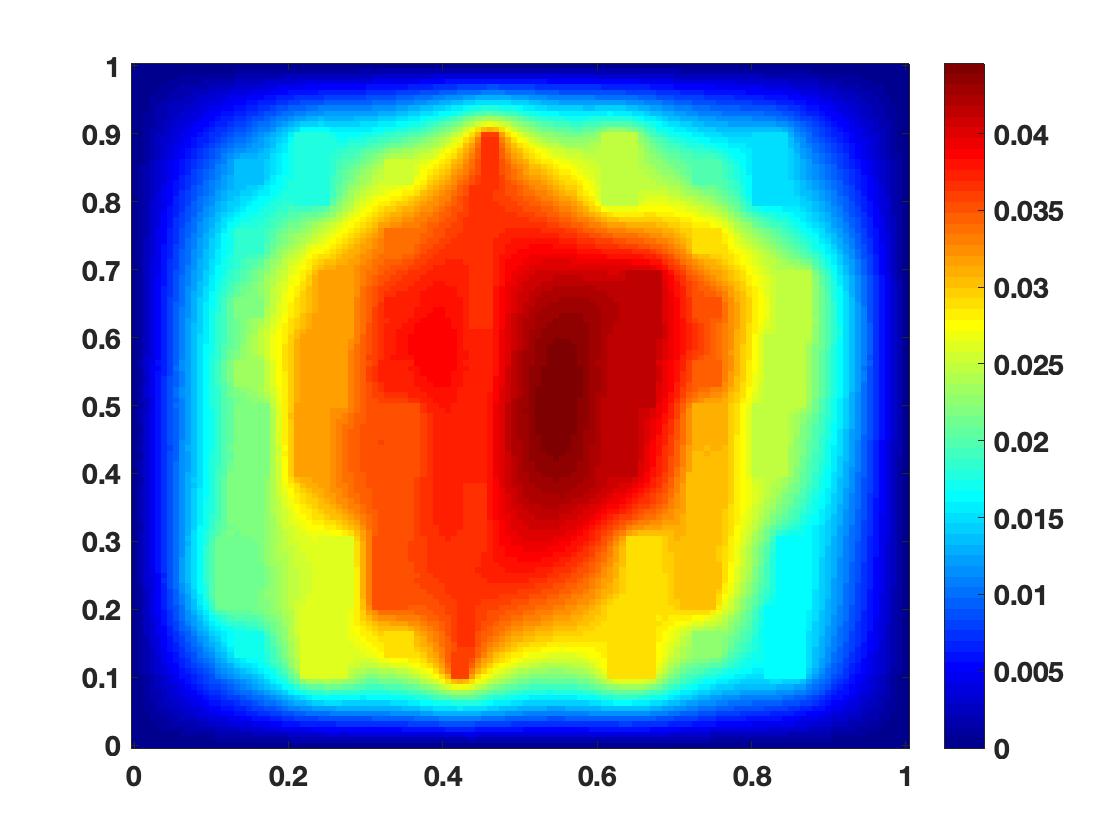}
		\includegraphics[trim={3cm 1.8cm 2.8cm 2cm},clip,width=0.24 \textwidth]{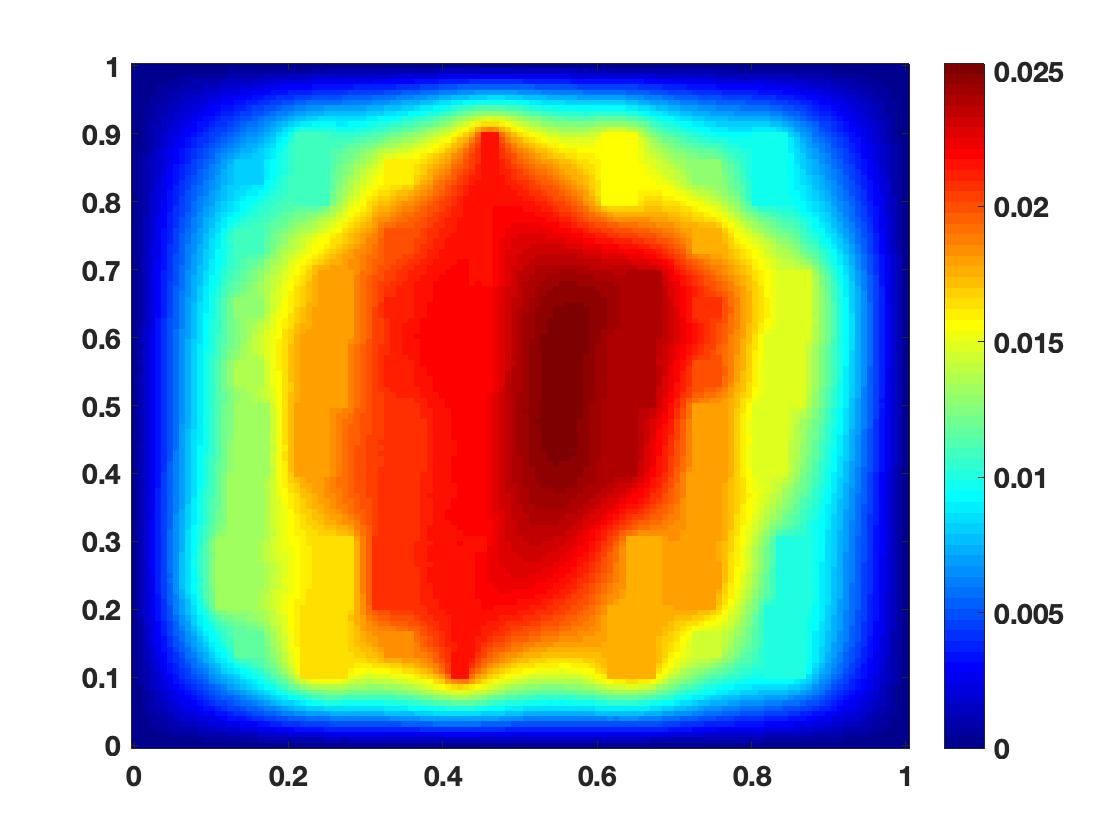}
		\includegraphics[trim={3cm 1.8cm 2.8cm 2cm},clip,width=0.24 \textwidth]{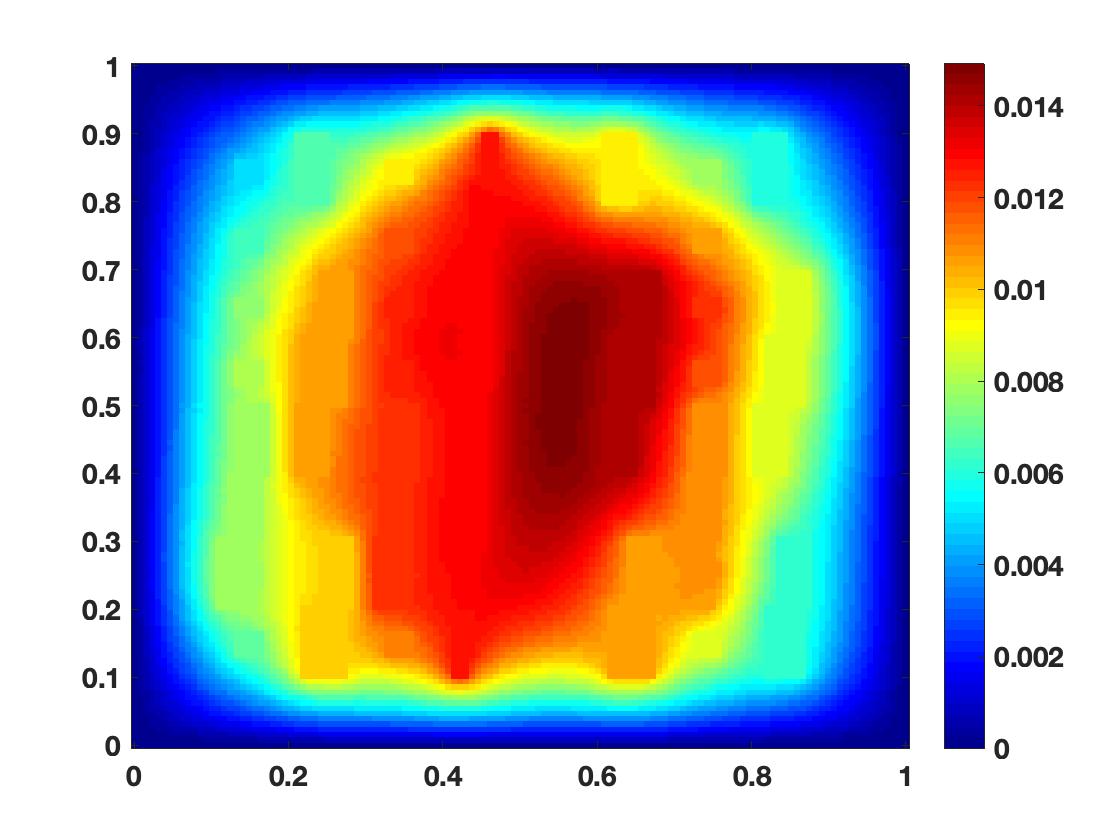}
		\includegraphics[trim={3cm 1.8cm 2.8cm 2cm},clip,width=0.24 \textwidth]{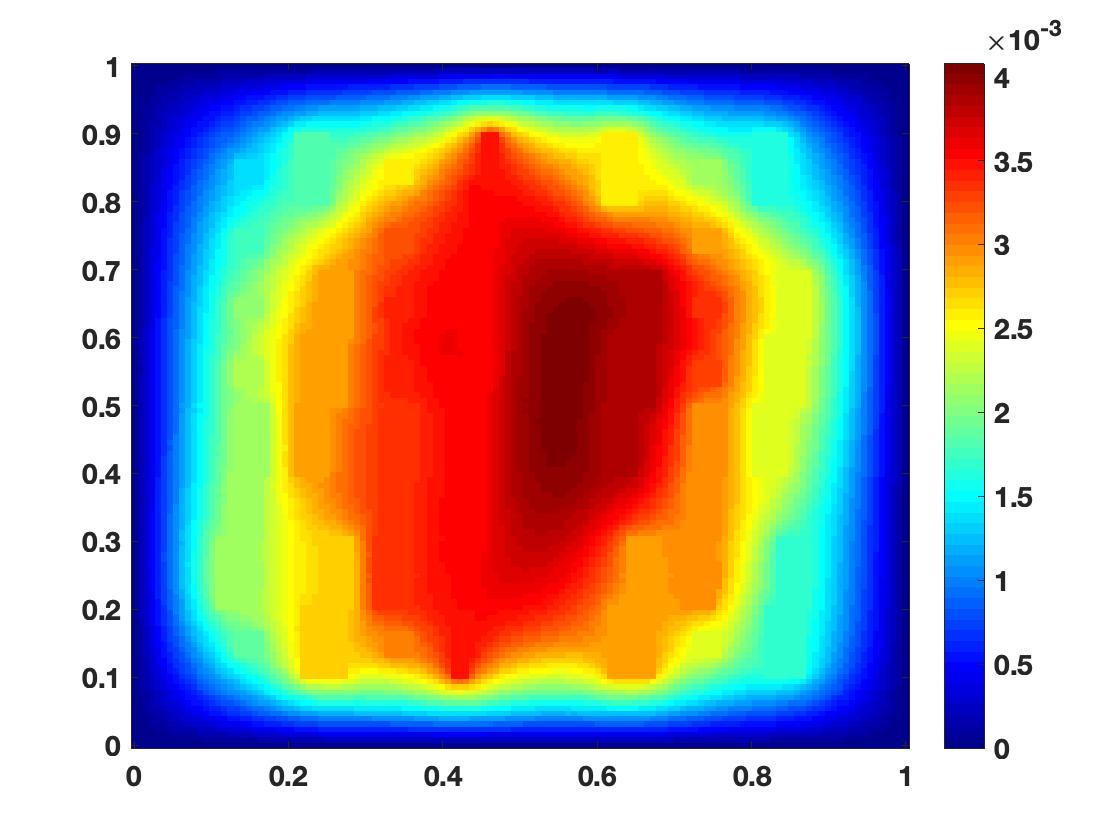}\\
		\includegraphics[trim={3cm 1.8cm 2.8cm 2cm},clip,width=0.24 \textwidth]{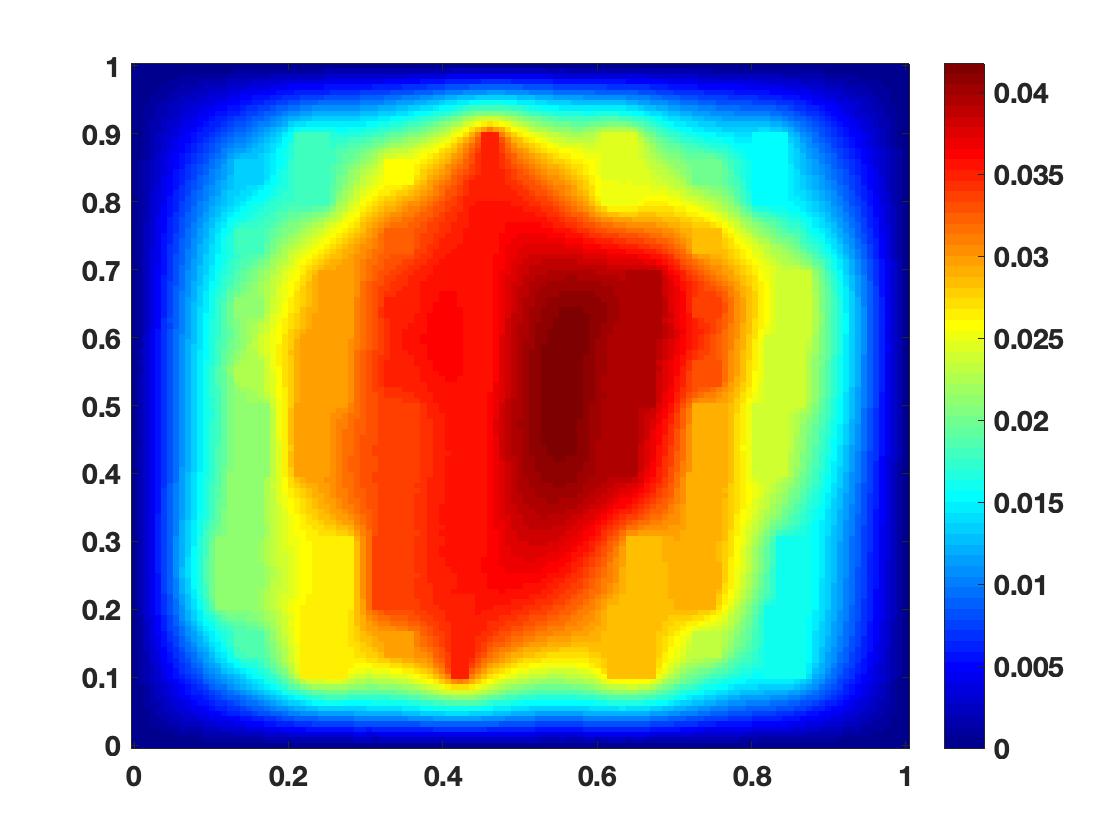}
		\includegraphics[trim={3cm 1.8cm 2.8cm 2cm},clip,width=0.24 \textwidth]{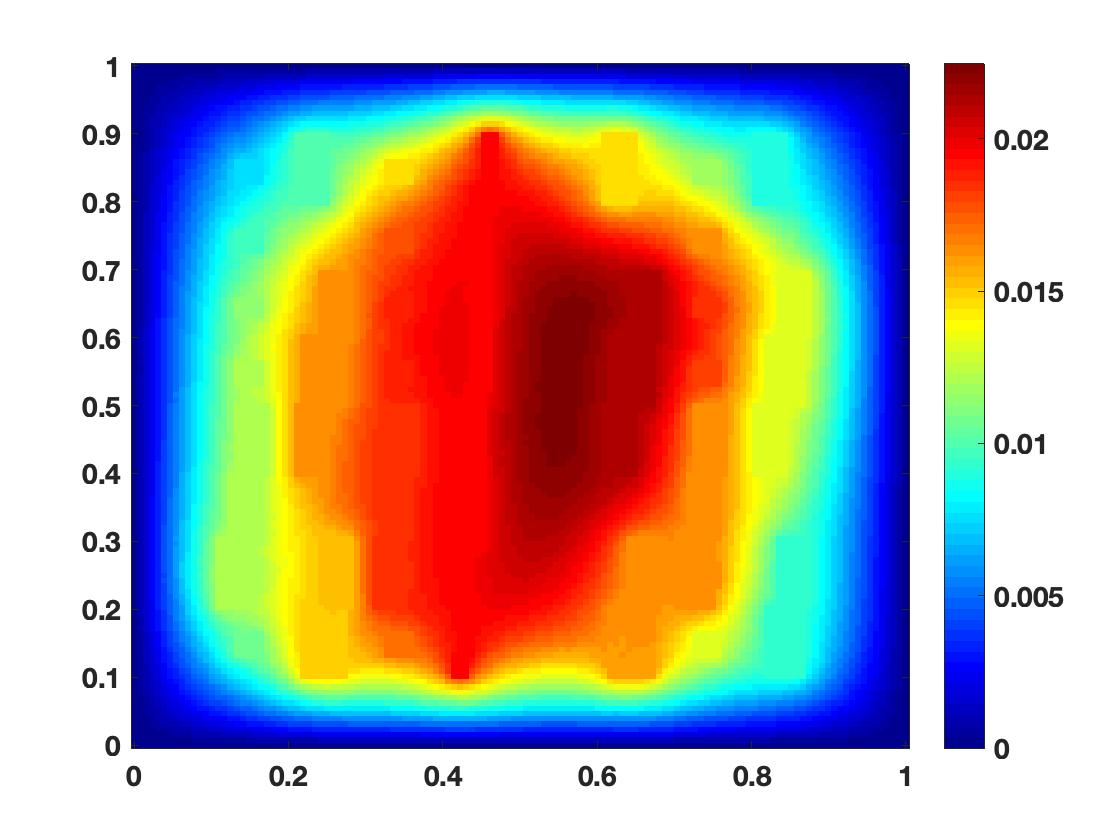}
		\includegraphics[trim={3cm 1.8cm 2.8cm 2cm},clip,width=0.24 \textwidth]{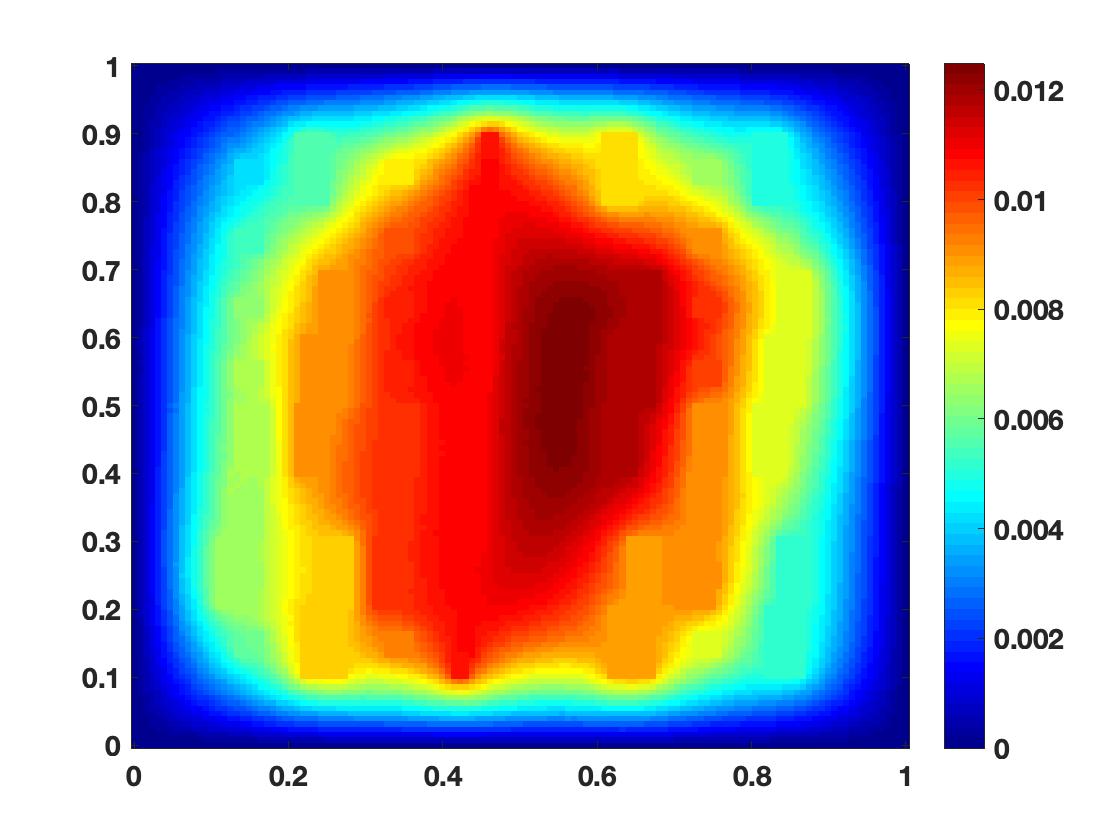}
		\includegraphics[trim={3cm 1.8cm 2.8cm 2cm},clip,width=0.24 \textwidth]{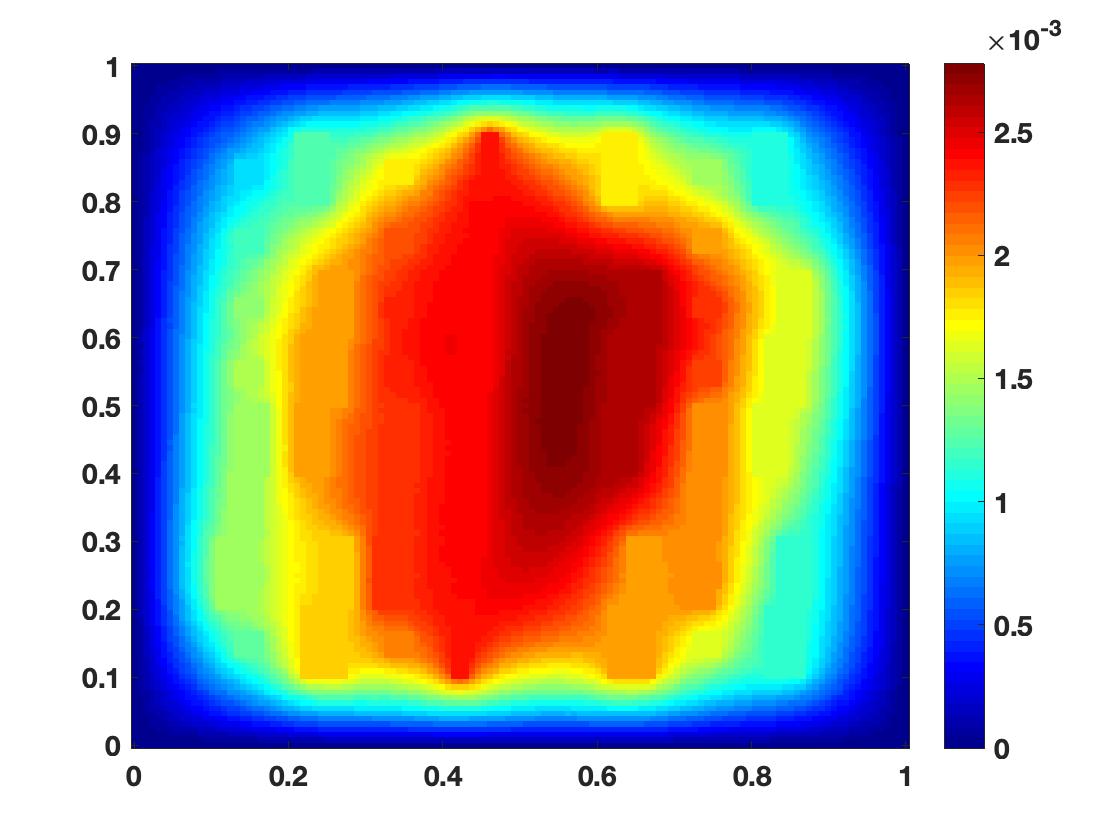}\\
		\includegraphics[trim={3cm 1.8cm 2.8cm 2cm},clip,width=0.24 \textwidth]{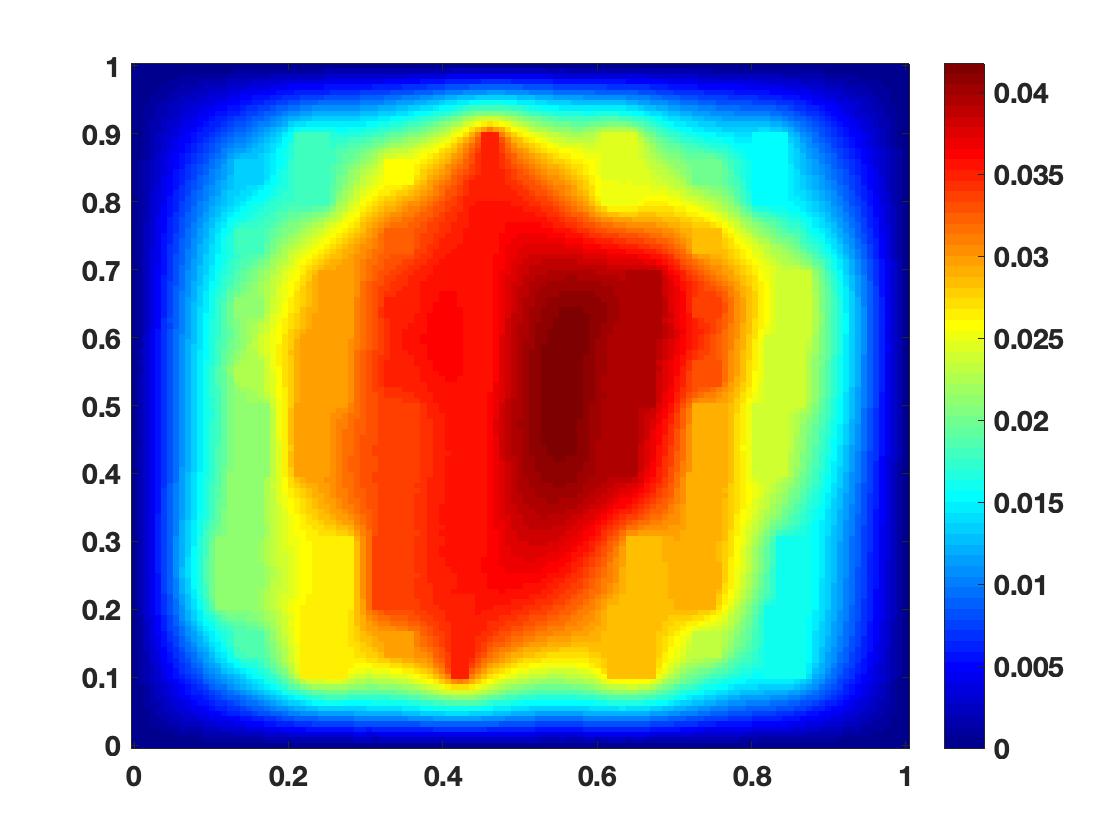}
		\includegraphics[trim={3cm 1.8cm 2.8cm 2cm},clip,width=0.24 \textwidth]{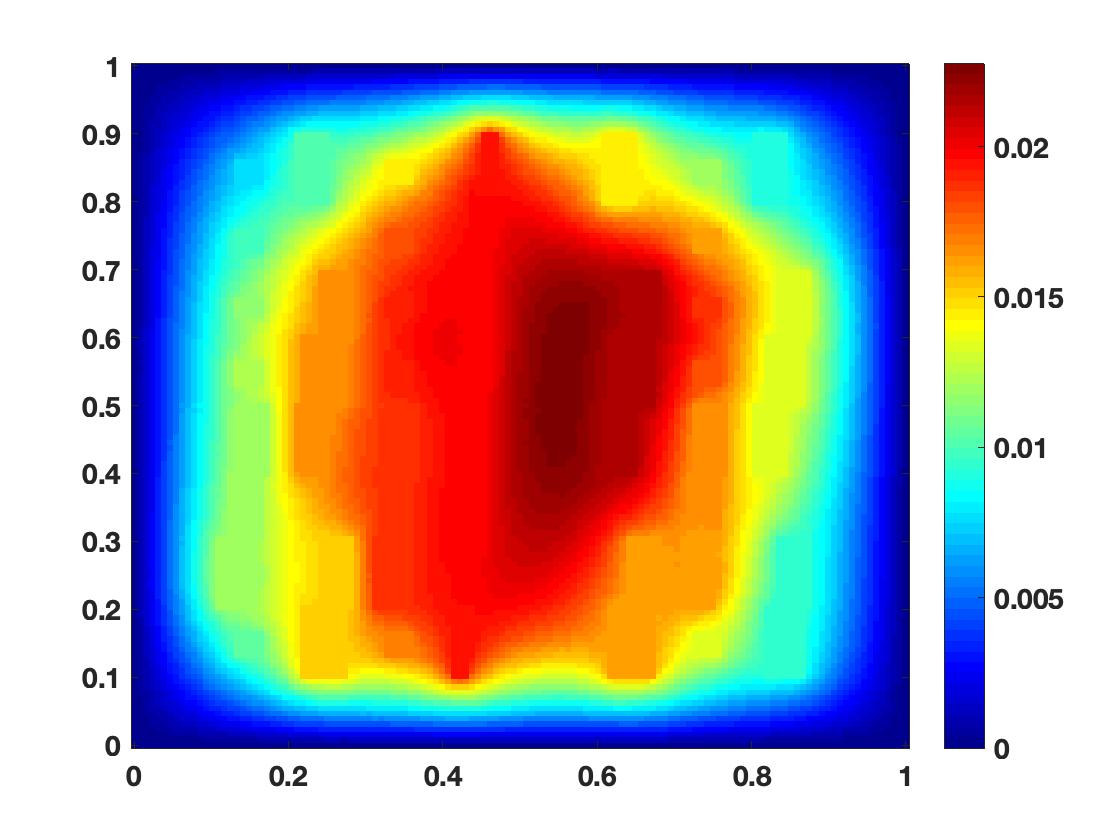}
		\includegraphics[trim={3cm 1.8cm 2.8cm 2cm},clip,width=0.24 \textwidth]{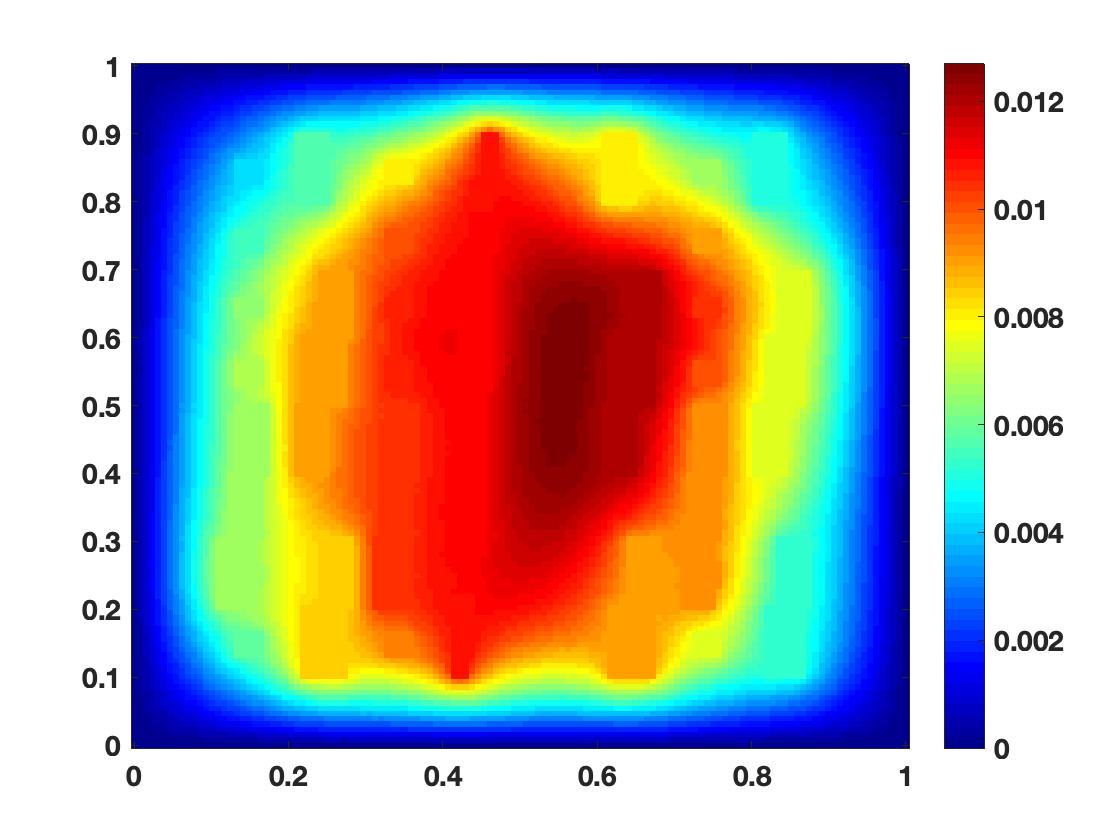}
		\includegraphics[trim={3cm 1.8cm 2.8cm 2cm},clip,width=0.24 \textwidth]{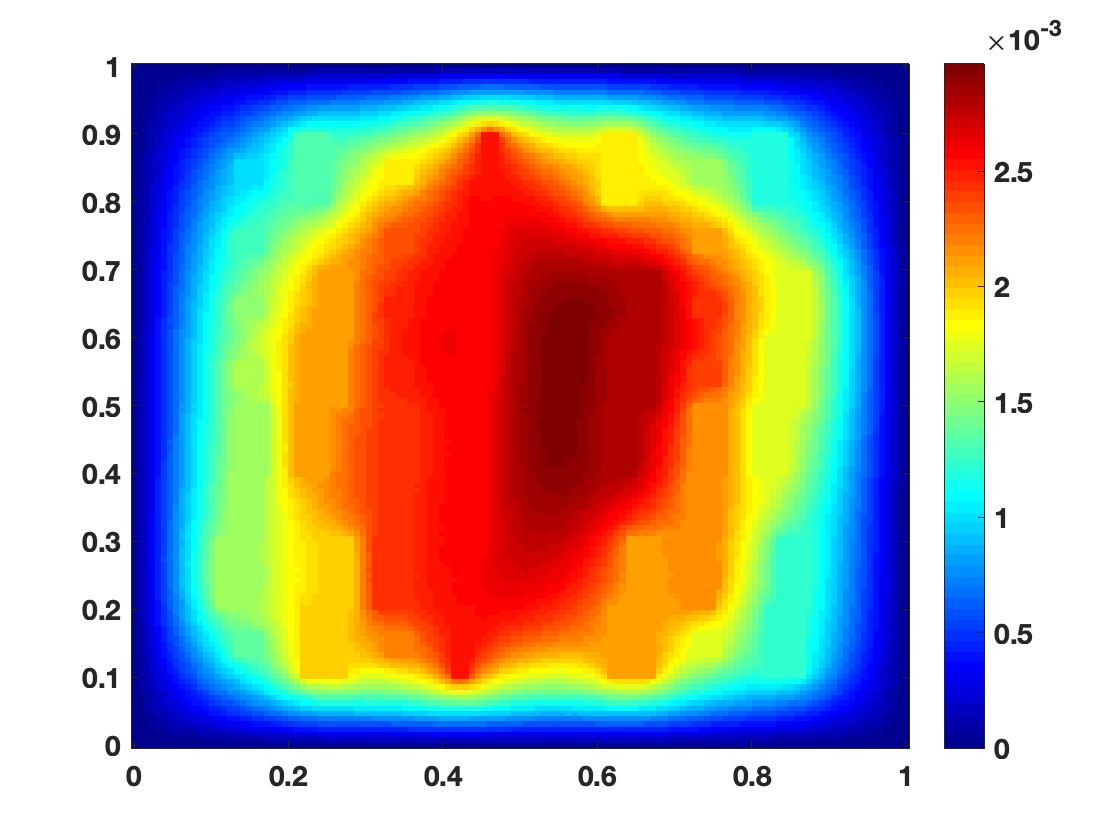}\\
		\caption{Numerical solutions $ U_k^n$ for $n=1, 3, 5, 10$ from Algorithm \ref{algorithm:wavelet+parareal} with ${\Delta T}=10^{-2}$ and ${\delta t}=10^{-3}$,  backward Euler scheme: iteration number $k=0$ (top), $k=1$ (middle) and $k=2$ (bottom).}
		\label{fig:PararealSol_ZeroSource_level2}
	\end{figure}		
The convergence history of Algorithm \ref{algorithm:wavelet+parareal} in $L^2(D)$-norm and $H^1_{\kappa}(D)$-norm is presented in Tables   \ref{error:L2_parareal_ZeroSource_l2} and  \ref{error:H1_parareal_ZeroSource_l2}. From the two tables, we see that  1 iteration is sufficent for the numerical solutions from Algorithm \ref{algorithm:wavelet+parareal} with backward Euler to converge under $L^2(D)$-norm and $H^1_{\kappa}(D)$-norm. We can conclude that our proposed algorithm with backward Euler scheme is effective in solving Problem (\ref{eqn:pde}) with zero source term.	
	\begin{table}[H]
	\begin{center}
	\begin{tabular}{|c|c|c|c|c|c|c|}
	\hline
	$T^n$ & $\text{Rel}^{\text{EW}}_{L^2} (T^n)$  &$\text{Rel}^0_{L^2}(T^n)$& $\text{Rel}^1_{L^2}(T^n)$&$\text{Rel}^2_{L^2}(T^n)$& $\text{Rel}^3_{L^2}(T^n)$& $\text{Rel}^4_{L^2}(T^n)$
	\\ \hline
	0.01 & 0.4381 & 4.8988 & 0.4381 & 0.4381 & 0.4381 & 0.4381 \\
0.02 & 0.5265 & 7.3857 & 0.7332 & 0.5265 & 0.5265 & 0.5265 \\
0.03 & 0.7540 & 11.0428 & 0.7002 & 0.7764 & 0.7540 & 0.7540 \\
0.04 & 0.9925 & 14.9558 & 0.5013 & 1.0253 & 0.9929 & 0.9925 \\
0.05 & 1.2334 & 19.0240 & 0.2844 & 1.2816 & 1.2342 & 1.2334 \\
0.06 & 1.4757 & 23.2389 & 0.3941 & 1.5582 & 1.4751 & 1.4758 \\
0.07 & 1.7191 & 27.6036 & 0.8900 & 1.8610 & 1.7154 & 1.7192 \\
0.08 & 1.9633 & 32.1231 & 1.5819 & 2.1943 & 1.9544 & 1.9636 \\
0.09 & 2.2084 & 36.8028 & 2.4535 & 2.5635 & 2.1913 & 2.2090 \\
0.10 & 2.4542 & 41.6483 & 3.5123 & 2.9751 & 2.4247 & 2.4554 \\
	\hline
	\end{tabular}
	\end{center}
	\vspace{-.4cm}
	\caption{Convergence history of Algorithm \ref{algorithm:wavelet+parareal} in relative $L^2(D)$ error for $f=0$: backward Euler scheme with ${\Delta T}=10^{-2}$ and ${\delta t}=10^{-3}$.}
	\label{error:L2_parareal_ZeroSource_l2}
	\end{table}

	\begin{table}[H]
	\begin{center}
	\begin{tabular}{|c|c|c|c|c|c|c|}
	\hline
	$T^n$ & $\text{Rel}^{\text{EW}}_{H_{\kappa}^1} (T^n)$&$\text{Rel}^0_{H_{\kappa}^1}(T^n)$& $\text{Rel}^1_{H_{\kappa}^1}(T^n)$& $\text{Rel}^2_{H_{\kappa}^1}(T^n)$& $\text{Rel}^3_{H_{\kappa}^1}(T^n)$& $\text{Rel}^4_{H_{\kappa}^1}(T^n)$
	\\ \hline
  0.01 & 7.1035 & 14.3354 & 7.1035 & 7.1035 & 7.1035 & 7.1035 \\
0.02 & 7.0711 & 11.3573 & 7.3430 & 7.0711 & 7.0711 & 7.0711 \\
0.03 & 7.1069 & 13.6291 & 7.2438 & 7.1275 & 7.1069 & 7.1069 \\
0.04 & 7.1544 & 17.0111 & 7.1216 & 7.1801 & 7.1551 & 7.1544 \\
0.05 & 7.2098 & 20.7809 & 7.0502 & 7.2335 & 7.2115 & 7.2098 \\
0.06 & 7.2727 & 24.7976 & 7.0138 & 7.3016 & 7.2741 & 7.2728 \\
0.07 & 7.3432 & 29.0209 & 7.0171 & 7.3900 & 7.3432 & 7.3434 \\
0.08 & 7.4211 & 33.4354 & 7.0887 & 7.5022 & 7.4188 & 7.4214 \\
0.09 & 7.5064 & 38.0353 & 7.2745 & 7.6424 & 7.5006 & 7.5068 \\
0.10 & 7.5990 & 42.8190 & 7.6303 & 7.8168 & 7.5877 & 7.5995 \\
	\hline
	\end{tabular}
	\end{center}
	\vspace{-.4cm}
	\caption{Convergence history of Algorithm \ref{algorithm:wavelet+parareal} in relative $H_{\kappa}^1 (D) $ error for $f=0$: backward Euler scheme with ${\Delta T}=10^{-2}$ and ${\delta t}=10^{-3}$.}
	\label{error:H1_parareal_ZeroSource_l2}
	\end{table}

Our last experiment is replacing backward Euler scheme by Crank-Nicolson scheme for the above problem. The corresponding multiscale solutions from Algorithm 1 are presented in Figure \ref{fig:WEMsFEM_solution_ZeroSource_CN}. We present the numerical solutions $U^n_k$ from Algorithm \ref{algorithm:wavelet+parareal} for $n=1,3,5,10$ with iteration number $k=0,2,4$ in Figure \ref{fig:PararealSol_ZeroSource_level2_CN}.
\begin{figure}[H]
		\centering
		\includegraphics[trim={3cm 1.8cm 2.8cm 2cm},clip,width=0.24 \textwidth]{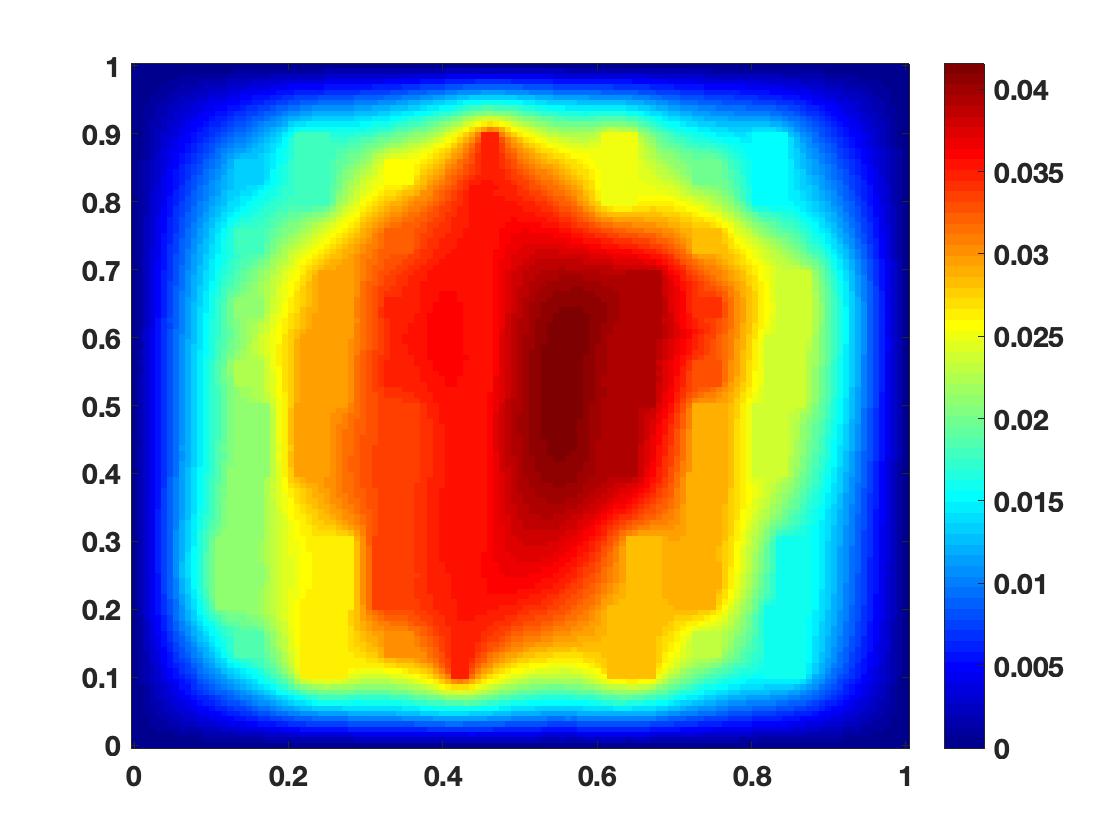}
		\includegraphics[trim={3cm 1.8cm 2.8cm 2cm},clip,width=0.24 \textwidth]{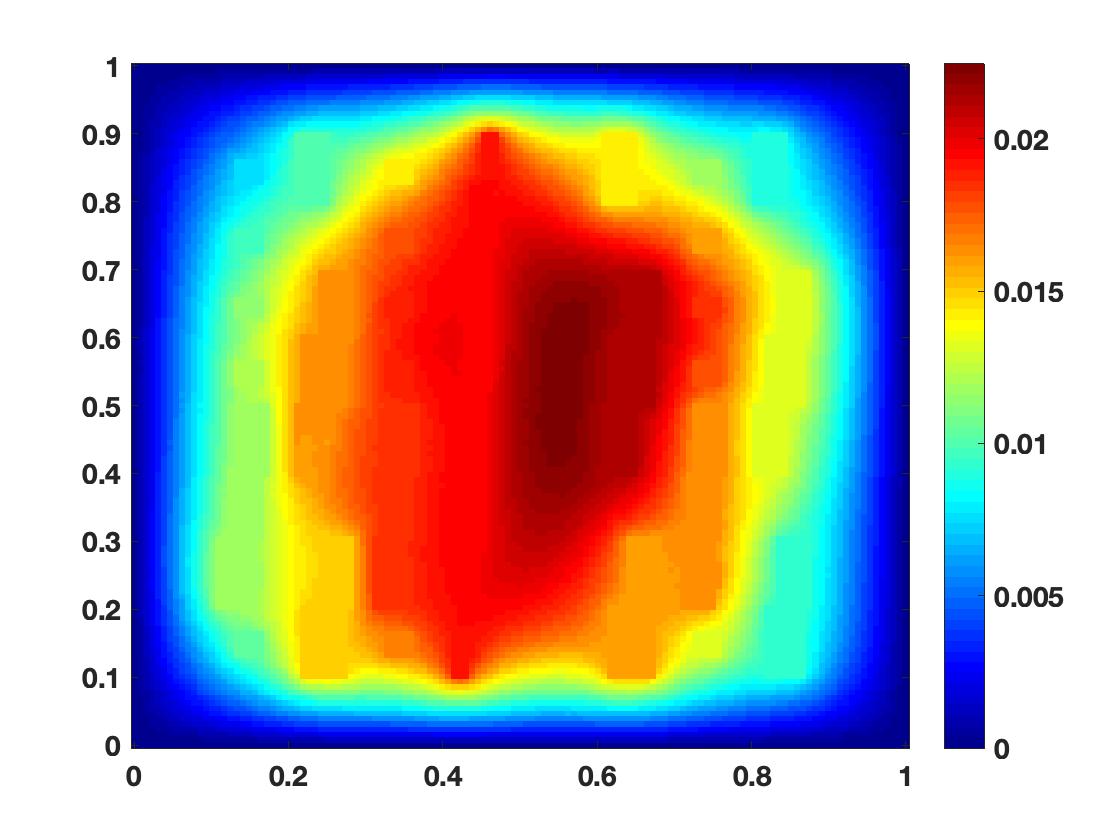}
		\includegraphics[trim={3cm 1.8cm 2.8cm 2cm},clip,width=0.24 \textwidth]{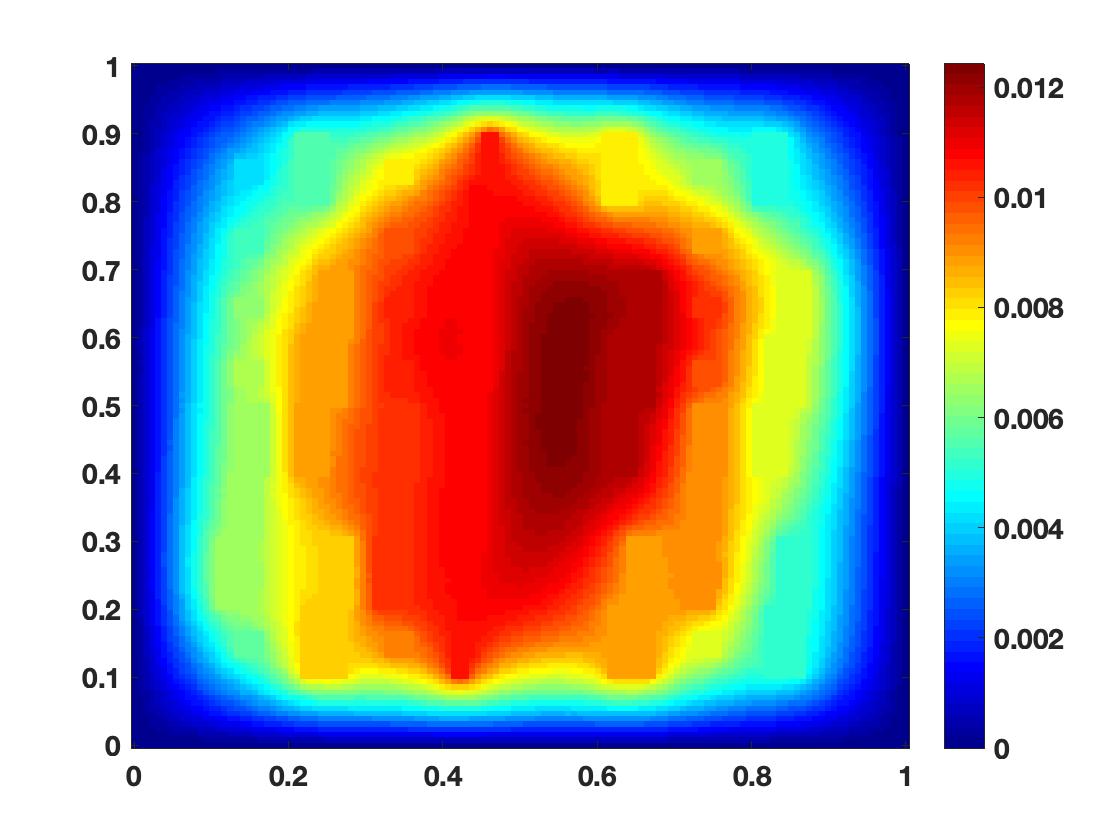}
		\includegraphics[trim={3cm 1.8cm 2.8cm 2cm},clip,width=0.24 \textwidth]{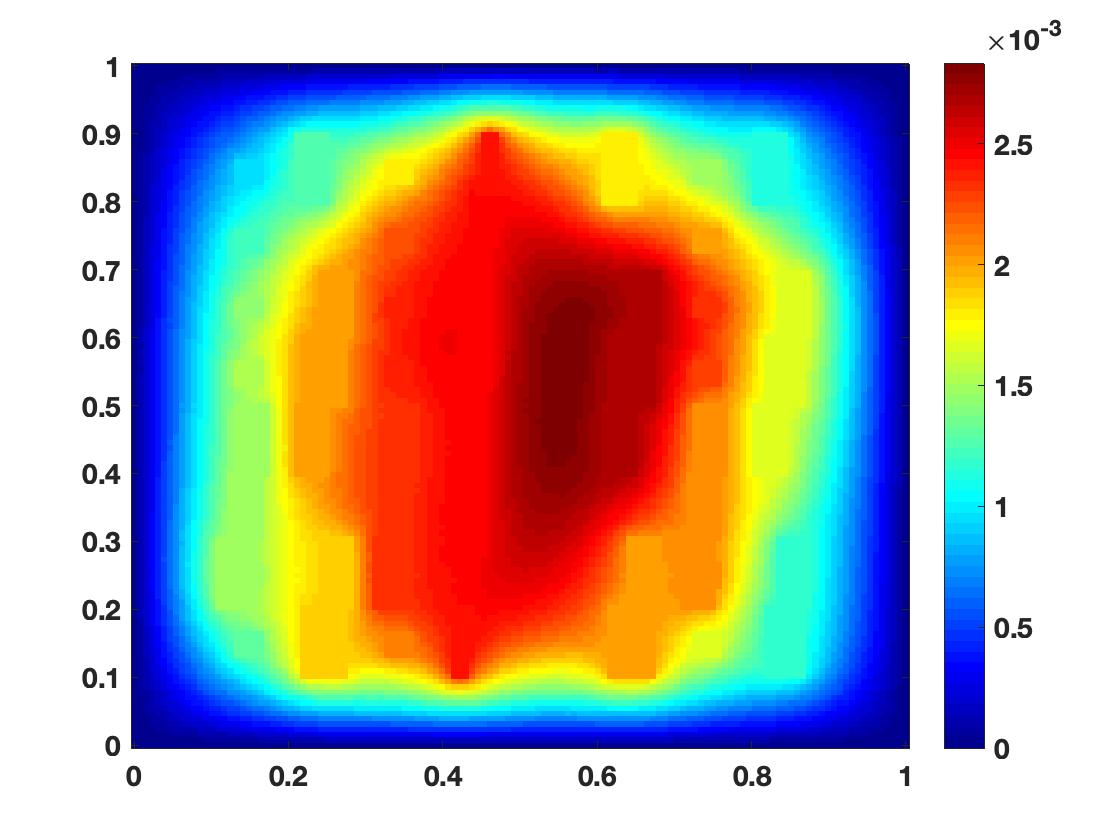}
		
		\caption{Multiscale solution from Algorithm \ref{algorithm:wavelet} with $\delta t=10^{-3}$ and $\ell=2$, Crank-Nicolson scheme: $u_{\text{ms},\ell}^{\text{EW,10}}$, $u_{\text{ms},\ell}^{\text{EW,30}}$,  $u_{\text{ms},\ell}^{\text{EW,50}}$ and $u_{\text{ms},\ell}^{\text{EW,100}}$. }
		\label{fig:WEMsFEM_solution_ZeroSource_CN}
	\end{figure}
\begin{figure}[H]
		\centering
		\includegraphics[trim={3cm 1.8cm 2.8cm 2cm},clip,width=0.24 \textwidth]{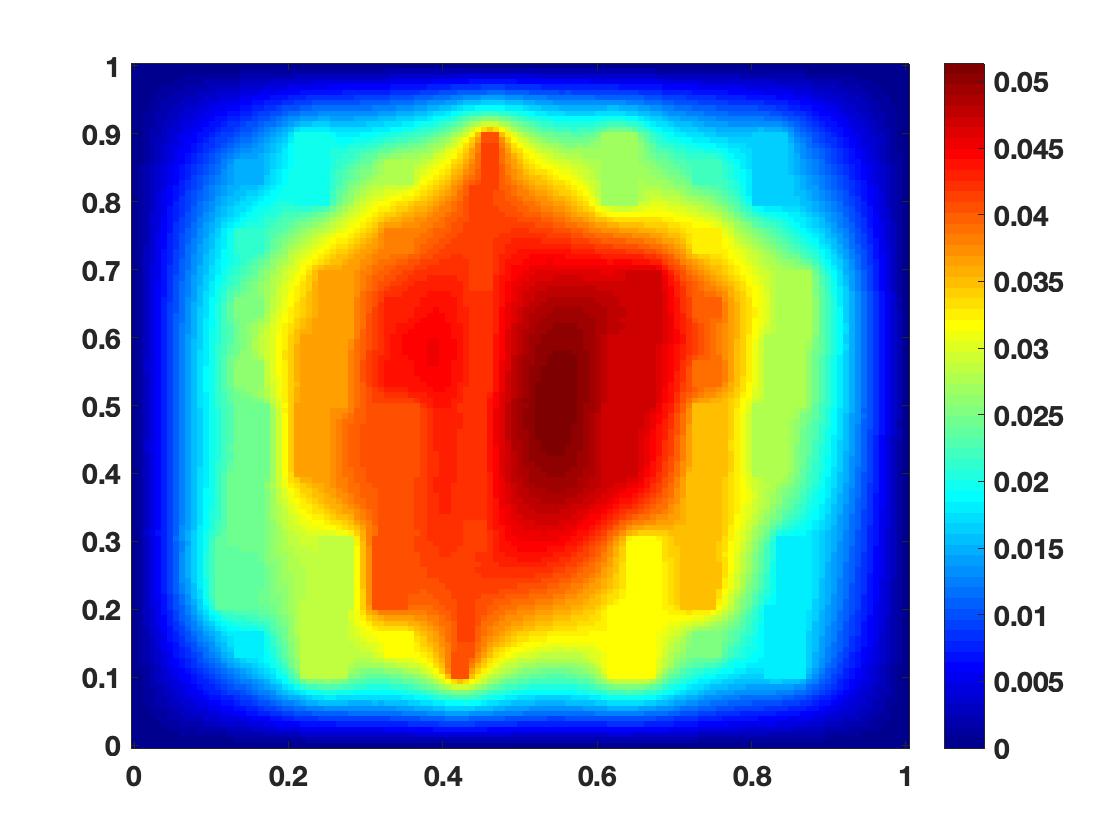}
		\includegraphics[trim={3cm 1.8cm 2.8cm 2cm},clip,width=0.24 \textwidth]{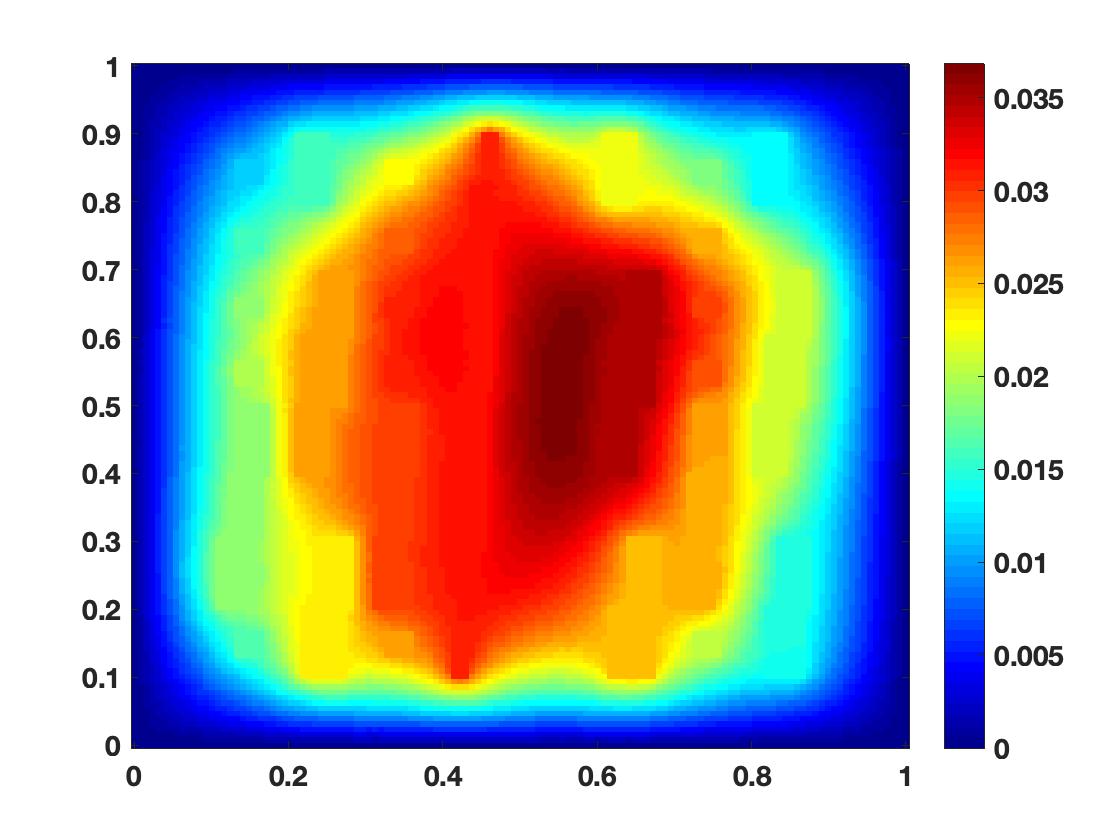}
		\includegraphics[trim={3cm 1.8cm 2.8cm 2cm},clip,width=0.24 \textwidth]{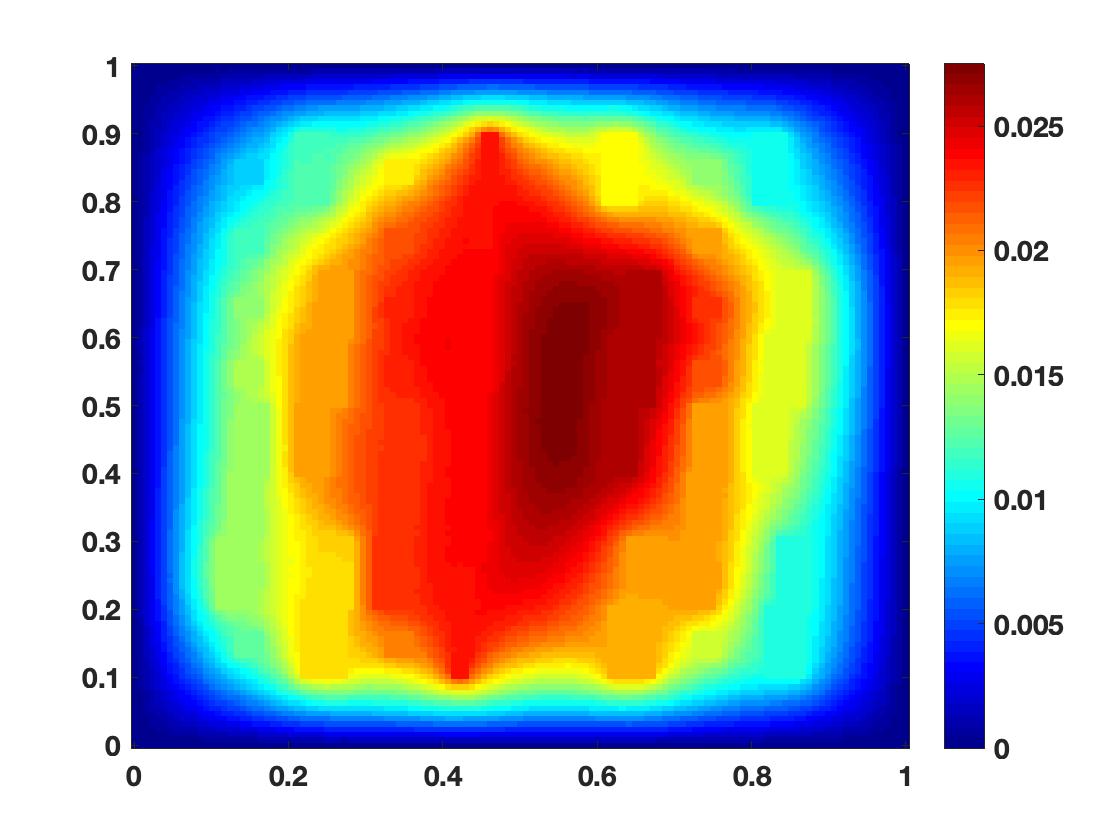}
		\includegraphics[trim={3cm 1.8cm 2.8cm 2cm},clip,width=0.24 \textwidth]{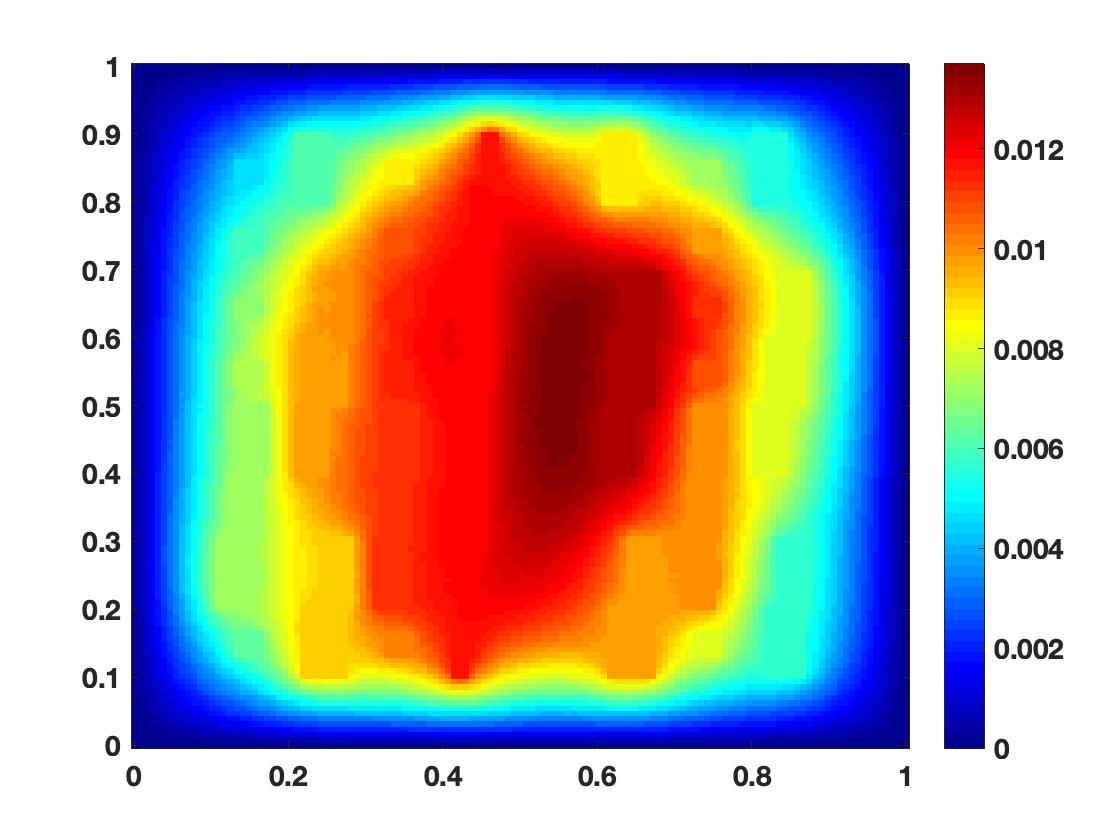}\\
		\includegraphics[trim={3cm 1.8cm 2.8cm 2cm},clip,width=0.24 \textwidth]{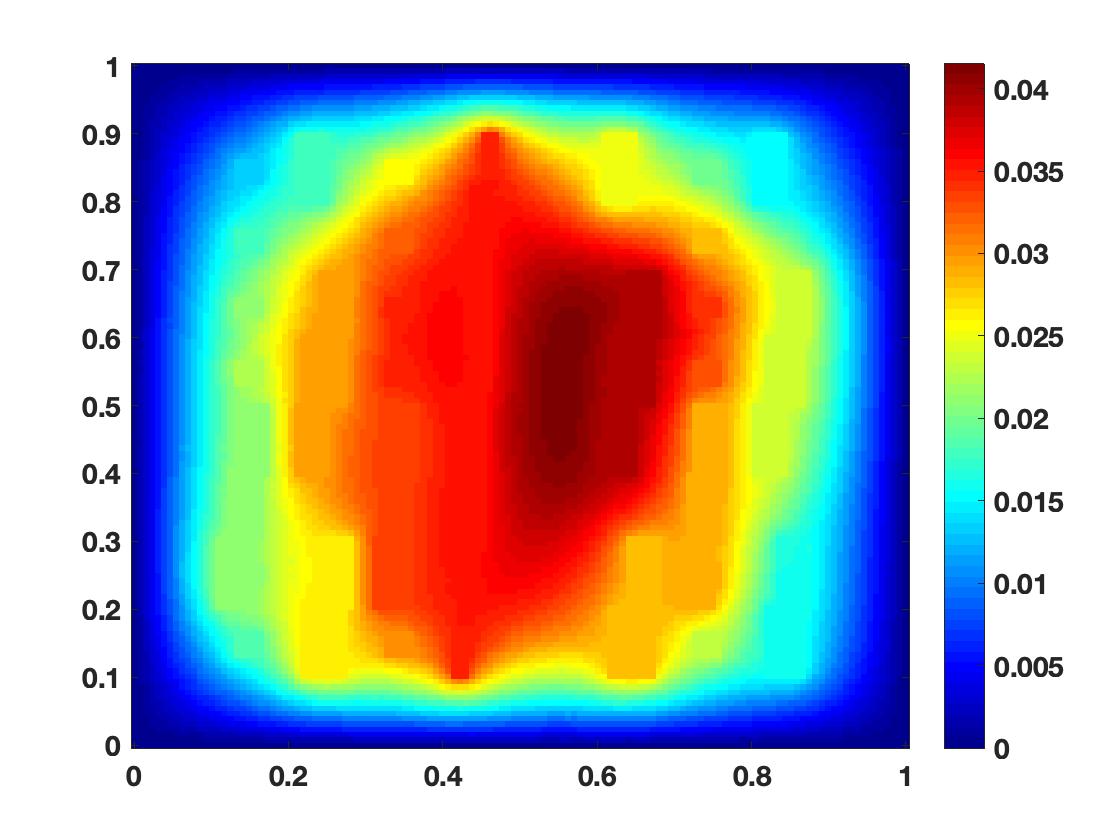}
		\includegraphics[trim={3cm 1.8cm 2.8cm 2cm},clip,width=0.24 \textwidth]{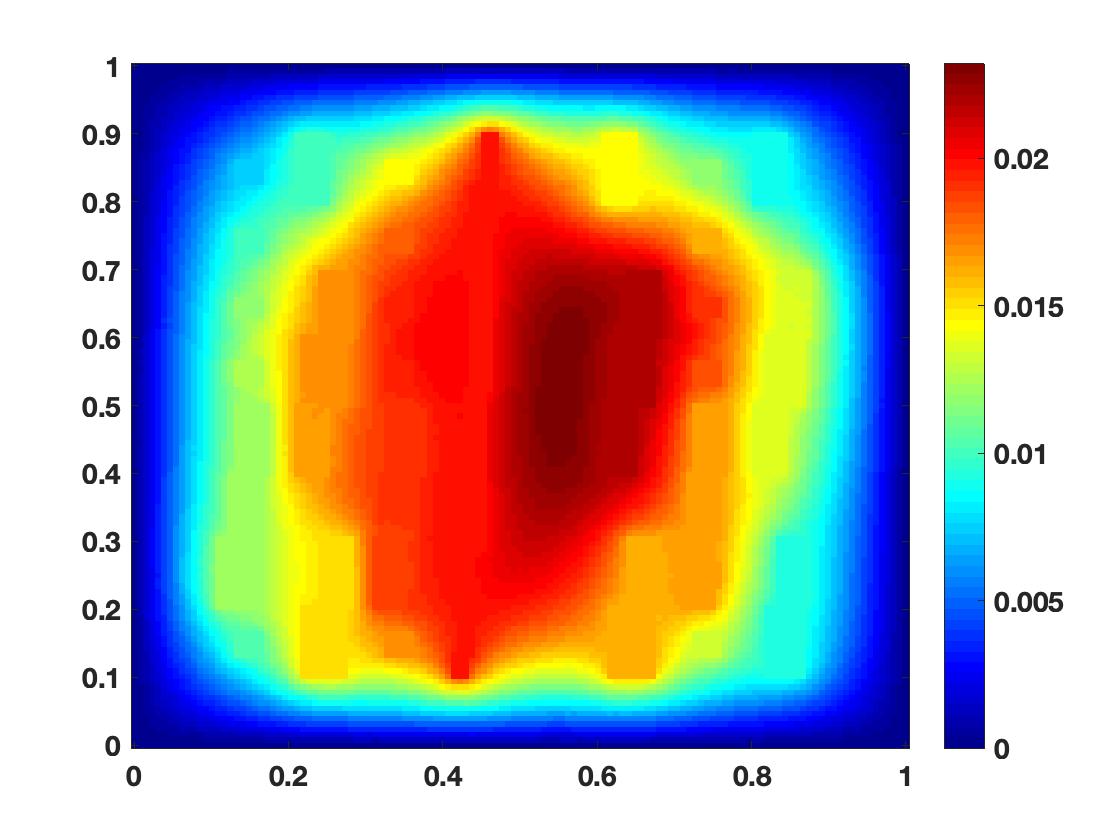}
		\includegraphics[trim={3cm 1.8cm 2.8cm 2cm},clip,width=0.24 \textwidth]{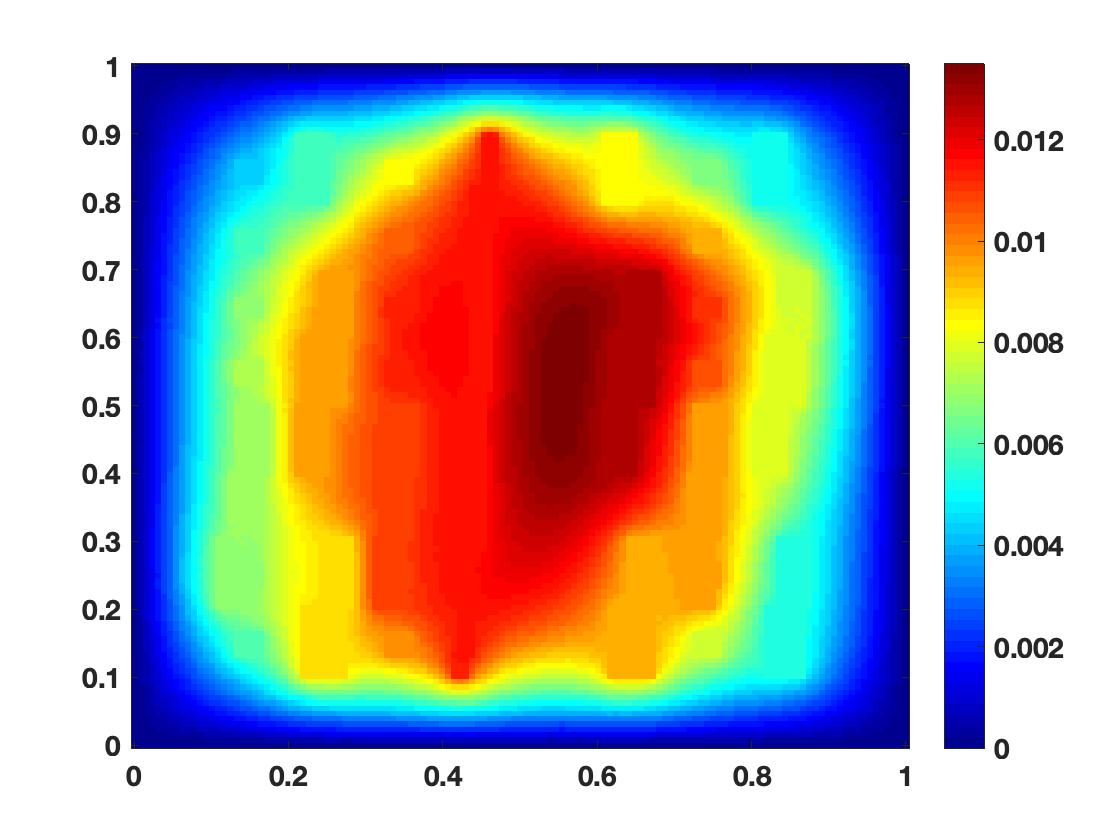}
		\includegraphics[trim={3cm 1.8cm 2.8cm 2cm},clip,width=0.24 \textwidth]{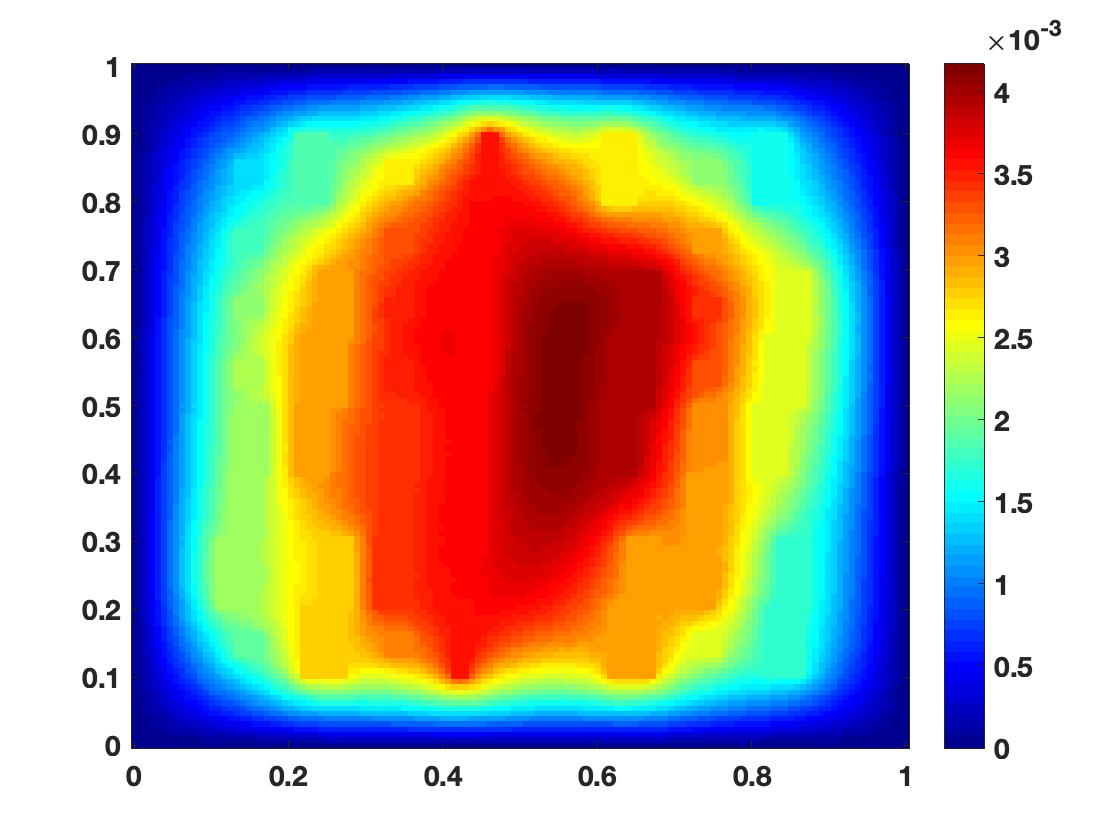}\\
		\includegraphics[trim={3cm 1.8cm 2.8cm 2cm},clip,width=0.24 \textwidth]{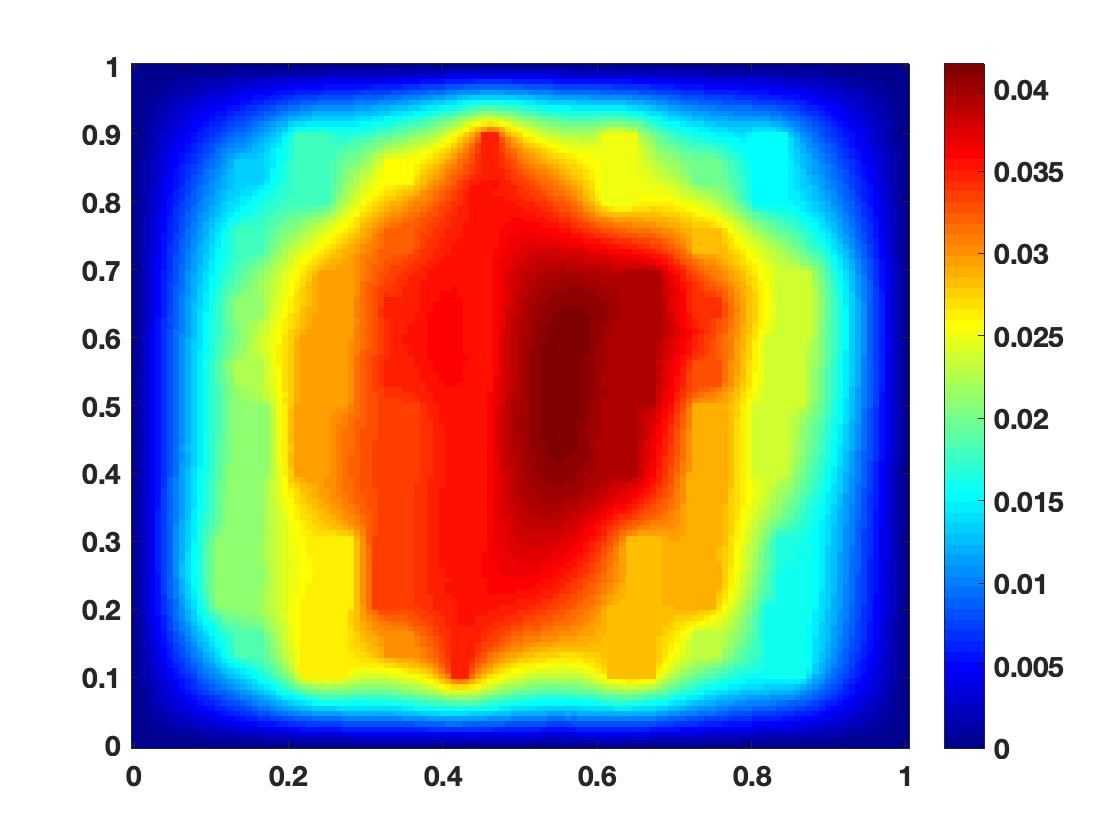}
		\includegraphics[trim={3cm 1.8cm 2.8cm 2cm},clip,width=0.24 \textwidth]{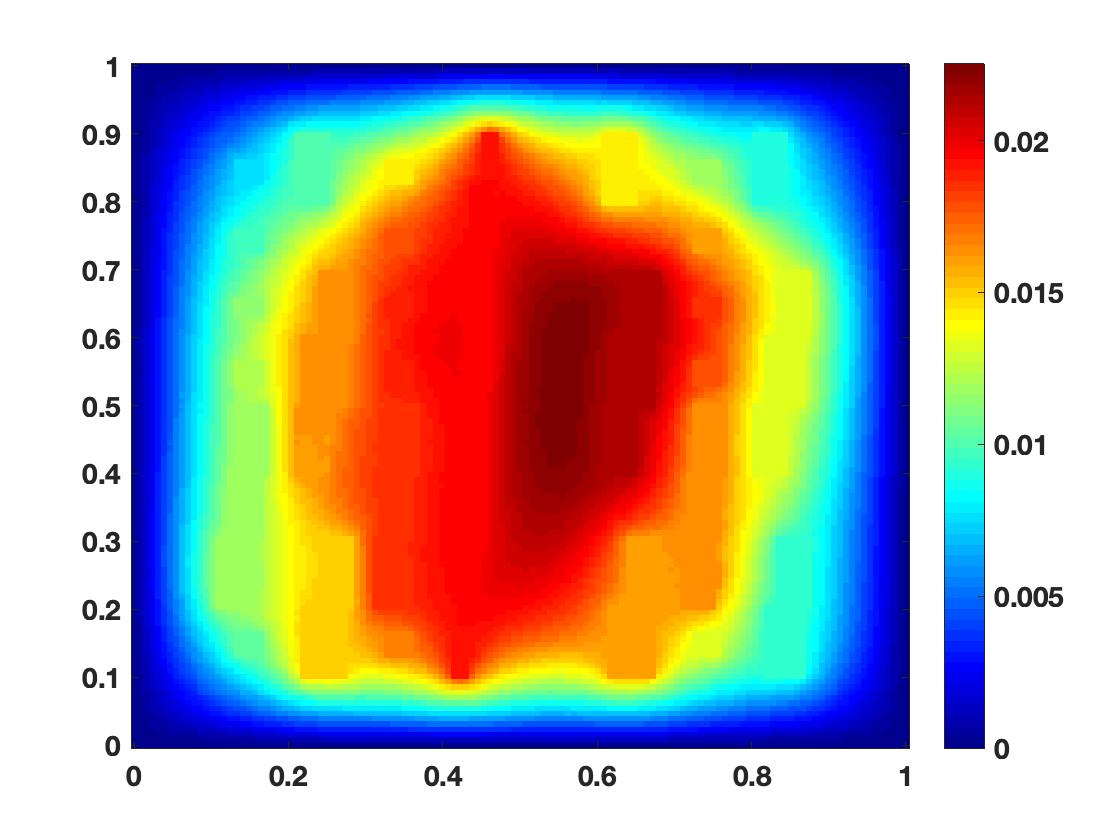}
		\includegraphics[trim={3cm 1.8cm 2.8cm 2cm},clip,width=0.24 \textwidth]{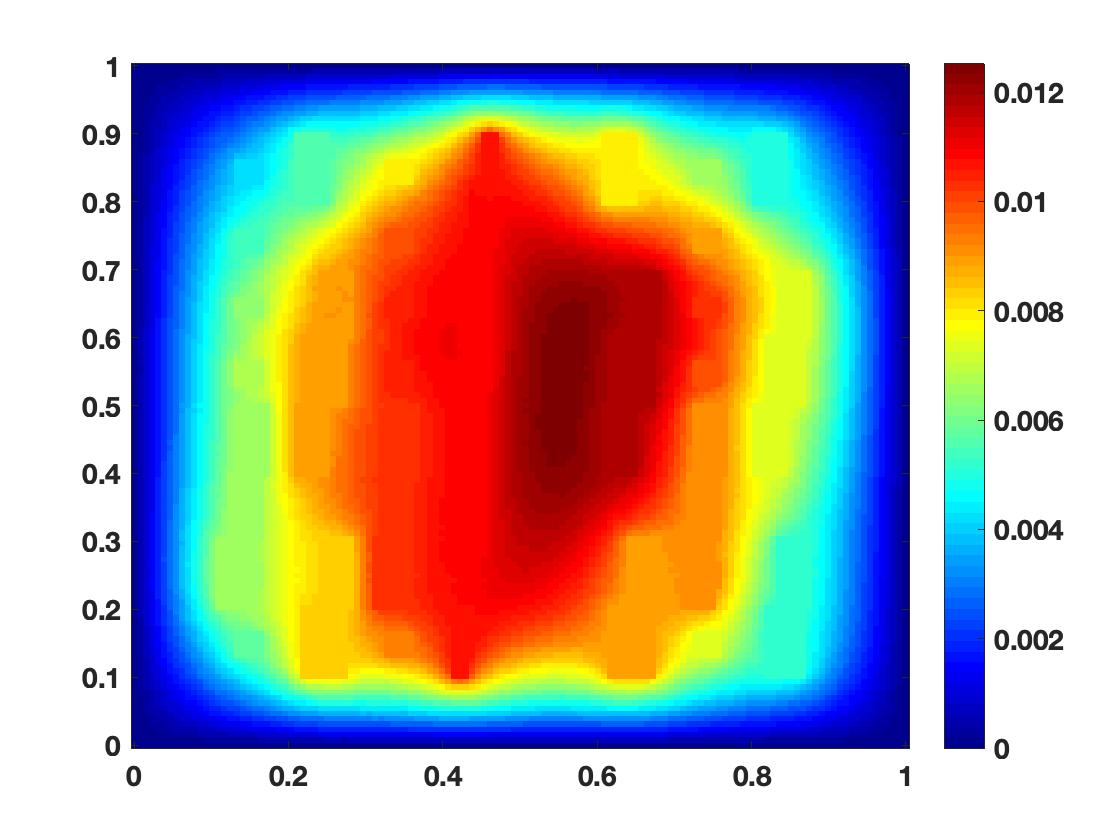}
		\includegraphics[trim={3cm 1.8cm 2.8cm 2cm},clip,width=0.24 \textwidth]{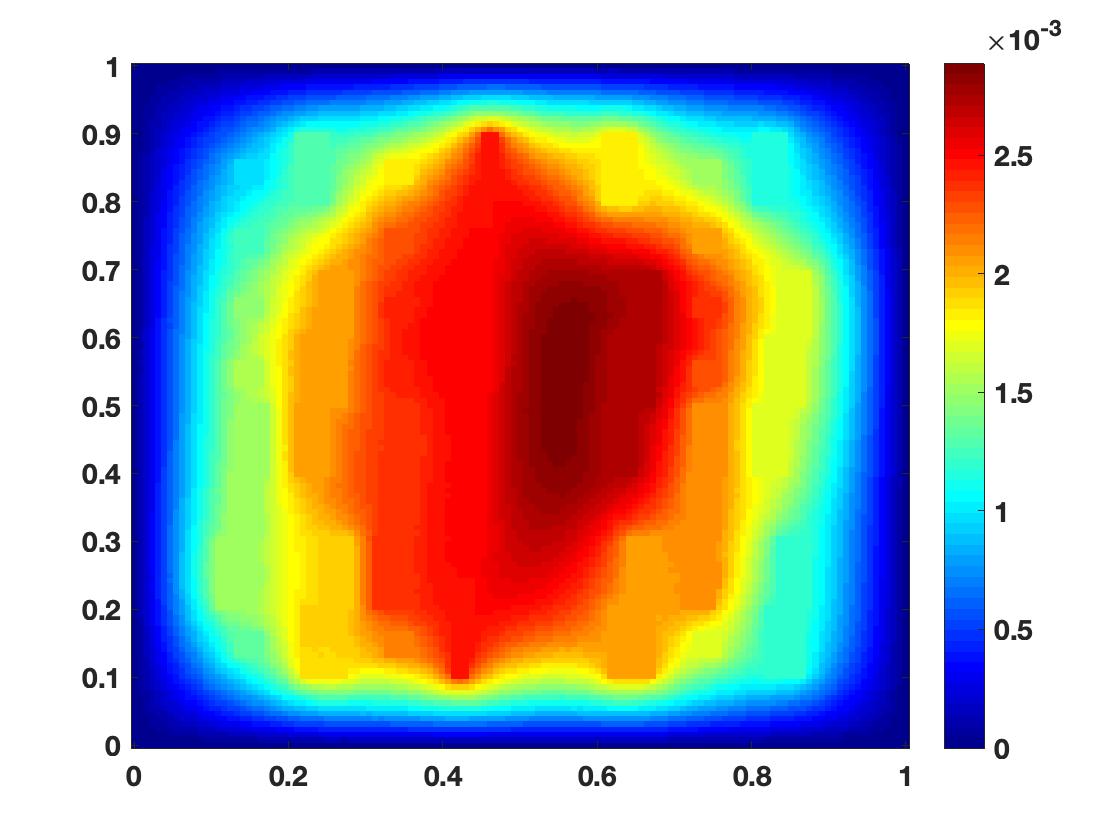}\\
		
		\caption{Numerical solutions $ U_k^n$ for $n=1, 3, 5, 10$ from Algorithm \ref{algorithm:wavelet+parareal} with ${\Delta T}=10^{-2}$ and ${\delta t}=10^{-3}$,  Crank-Nicolson scheme: iteration number $k=0$ (top), $k=2$ (middle) and $k=4$ (bottom).}
		\label{fig:PararealSol_ZeroSource_level2_CN}
	\end{figure}

	The convergence history of Algorithm \ref{algorithm:wavelet+parareal} in $L^2(D)$-norm and $H^1_{\kappa}(D)$-norm is presented in Tables  \ref{error:L2_parareal_ZeroSource_l2_CN} and \ref{error:H1_parareal_ZeroSource_l2_CN}. One observes that it takes 4 iterations to converge under $L^2(D)$-norm and 3 iterations to converge under $H_{\kappa}^1(D)$-norm when using Algorithm \ref{algorithm:wavelet+parareal} with Crank-Nicolson scheme.
Comparing Table \ref{error:L2_parareal_ZeroSource_l2} with Table \ref{error:L2_parareal_ZeroSource_l2_CN}, we can see that Algorithm \ref{algorithm:wavelet+parareal} with the Crank-Nicolson scheme yields a better accuracy than that with the backward Euler scheme. We conclude that Algorithm \ref{algorithm:wavelet+parareal} with backward Euler scheme converges faster than that with Crank-Nicolson scheme, while Algorithm \ref{algorithm:wavelet+parareal} with Crank-Nicolson scheme generate solutions with a higher accuracy for Problem (\ref{eqn:pde}) with zero source term.
\begin{table}[H]
	\begin{center}
	\begin{tabular}{|c|c|c|c|c|c|c|}
	\hline
	$T^n$ & $\text{Rel}^{\text{EW}}_{L^2} (T^n)$  &$\text{Rel}^0_{L^2}(T^n)$& $\text{Rel}^1_{L^2}(T^n)$&$\text{Rel}^2_{L^2}(T^n)$& $\text{Rel}^3_{L^2}(T^n)$& $\text{Rel}^4_{L^2}(T^n)$
	\\ \hline
	0.01 & 0.1915 & 17.9465 & 13.3079 & 0.1677 & 0.1677 & 0.1677 \\
0.02 & 0.2209 & 36.8558 & 29.6961 & 0.7619 & 0.1607 & 0.1607 \\
0.03 & 0.3817 & 59.7687 & 50.4219 & 1.8702 & 0.2242 & 0.1687 \\
0.04 & 0.5603 & 86.6988 & 75.8706 & 3.5561 & 0.3974 & 0.1805 \\
0.05 & 0.7439 & 118.2106 & 106.8652 & 6.2307 & 0.5572 & 0.2087 \\
0.06 & 0.9293 & 155.0529 & 144.3583 & 10.1676 & 0.7313 & 0.2443 \\
0.07 & 1.1155 & 198.1194 & 189.4784 & 15.6176 & 1.0374 & 0.2676 \\
0.08 & 1.3020 & 248.4594 & 243.5539 & 22.8847 & 1.5818 & 0.2717 \\
0.09 & 1.4886 & 307.3003 & 308.1424 & 32.3420 & 2.4362 & 0.2607 \\
0.10 & 1.6751 & 376.0776 & 385.0672 & 44.4367 & 3.6758 & 0.2492 \\
	\hline
	\end{tabular}
	\end{center}
	\vspace{-.4cm}
	\caption{Convergence history of Algorithm \ref{algorithm:wavelet+parareal}  in relative $L^2(D)$ error for $f=0$: Crank-Nicolson scheme with ${\Delta T}=10^{-2}$ and ${\delta t}=10^{-3}$.}
	\label{error:L2_parareal_ZeroSource_l2_CN}
	\end{table}

\begin{table}[H]
	\begin{center}
	\begin{tabular}{|c|c|c|c|c|c|c|}
	\hline
	$T^n$ & $\text{Rel}^{\text{EW}}_{H_{\kappa}^1} (T^n)$&$\text{Rel}^0_{H_{\kappa}^1}(T^n)$& $\text{Rel}^1_{H_{\kappa}^1}(T^n)$& $\text{Rel}^2_{H_{\kappa}^1}(T^n)$& $\text{Rel}^3_{H_{\kappa}^1}(T^n)$& $\text{Rel}^4_{H_{\kappa}^1}(T^n)$
	\\ \hline
0.01 & 7.0217 & 28.3350 & 16.7674 & 7.0138 & 7.0138 & 7.0138 \\
0.02 & 7.0053 & 39.6766 & 31.9136 & 7.3167 & 7.0129 & 7.0129 \\
0.03 & 7.0010 & 61.1794 & 51.3326 & 8.0179 & 7.0264 & 7.0138 \\
0.04 & 7.0014 & 87.7871 & 76.1885 & 8.7646 & 7.0801 & 7.0154 \\
0.05 & 7.0065 & 119.2332 & 106.9788 & 10.1730 & 7.1238 & 7.0207 \\
0.06 & 7.0162 & 156.0967 & 144.4364 & 13.0098 & 7.1297 & 7.0285 \\
0.07 & 7.0307 & 199.2282 & 189.6019 & 17.7065 & 7.1231 & 7.0341 \\
0.08 & 7.0501 & 249.6645 & 243.7722 & 24.5026 & 7.1574 & 7.0358 \\
0.09 & 7.0742 & 308.6303 & 308.4945 & 33.6707 & 7.3146 & 7.0357 \\
0.10 & 7.1032 & 377.5617 & 385.5910 & 45.5892 & 7.7242 & 7.0368 \\
	\hline
	\end{tabular}
	\end{center}
	\vspace{-.4cm}
	\caption{Convergence history of Algorithm \ref{algorithm:wavelet+parareal}  in relative $H_{\kappa}^1 (D) $ error for $f=0$: Crank-Nicolson scheme with ${\Delta T}=10^{-2}$ and ${\delta t}=10^{-3}$.}
	\label{error:H1_parareal_ZeroSource_l2_CN}
	\end{table}
	
\section{Conclusion}\label{sec:conclusion}
We propose in this paper a new efficient algorithm for parabolic problems with heterogeneous coefficients. This algorithm is named as the Wavelet-based Edge Multiscale Parareal (WEMP) Algorithm, which incorporates parareal algorithm into the Wavelet-based Edge Multiscale Finite Element Methods (WEMsFEMs). Therefore, it can handle parabolic problems with heterogeneous coefficients more efficiently if multiple processors are available. We derive the convergence rate of this algorithm, and verify its performance by several numerical tests. Future work includes the investigation of WEMP algorithm for time-fractional diffusion problems with heterogeneous coefficients.

	\bibliographystyle{abbrv}
	\bibliography{reference}
\end{document}